\documentclass[smallextended,referee,envcountsect,]{svjour3}
\smartqed
\journalname{JOTA}

\usepackage{graphicx}
\usepackage{microtype}
\usepackage{verbatim}
\usepackage{times}
\usepackage{mathtools}
\usepackage{caption}
\usepackage{subcaption}
\usepackage{amsfonts}
\usepackage{amssymb}
\usepackage{enumitem}
\usepackage[enable]{easy-todo}
\usepackage{color,soul}
\usepackage{etoolbox}
\usepackage{appendix}
\usepackage{algorithm}
\usepackage{algpseudocode}
 \usepackage{url}
\AtEndEnvironment{proof}{\phantom{}\qed}

\newcommand{\cK}{\mathcal{K}}
\newcommand{\Rd}{\mathbb{R}^d}
\newcommand{\abs}[1]{\vert #1 \rvert}
\newcommand{\bd}[3]{\mathcal{B}_{#1}(#2,#3)}
\newcommand{\scf}{\phi}

\newcommand{\dualscf}{\phi^{*}}

\newcommand{\scfi}{\psi}

\newcommand{\func}{l}
\newcommand{\obj}{f}
\newcommand{\comp}{r}
\newcommand{\sgn}{\operatorname{sgn}}
\newcommand{\grad}{\triangledown}
\newcommand{\norm}[1]{\lVert #1\rVert }
\newcommand{\dualnorm}[1]{\lVert #1\rVert_* }
\newcommand{\inner}[2]{\langle #1,#2\rangle }
\newcommand{\ex}[1]{\mathbb{E}[#1]}

\newcommand{\sequ}[1]{\{#1_t\}}
\newcommand{\proxi}{\mathcal{P}_\cK}
\newcommand{\Grad}{\mathcal{G}_\cK}

\DeclareMathOperator*{\argmin}{\arg\min}

\newtheorem{assumption}{Assumption}
\begin{document}

\title{Adaptive Stochastic Optimisation of Nonconvex Composite Objectives}

\author{Weijia Shao, Fikret Sivrikaya, Sahin Albayrak}

\institute{Weijia Shao, Corresponding author \at
            weijia.shao@campus.tu-berlin.de\\
            Sahin Albayrak\at 
            Technische Universit\"at Berlin \\
            Ernst-Reuter-Platz 7 10587, Berlin Germany\\
            Fikret Sivrikaya\at 
            GT-ARC Gemeinn\"utzige GmbH\\
            Ernst-Reuter-Platz 7 10587, Berlin Germany\\
}

\date{Received: date / Accepted: date}

\maketitle

\begin{abstract}
In this paper, we propose and analyse a family of generalised stochastic composite mirror descent algorithms. With adaptive step sizes, the proposed algorithms converge without requiring prior knowledge of the problem. Combined with an  entropy-like update-generating function, these algorithms perform gradient descent in the space equipped with the maximum norm, which allows us to exploit the low-dimensional structure of the decision sets for high-dimensional problems. Together with a sampling method based on the Rademacher distribution and variance reduction techniques, the proposed algorithms guarantee a logarithmic complexity dependence on dimensionality for zeroth-order optimisation problems.
\end{abstract}

\institute{Technische Universit\"at Berlin, Ernst-Reuter-Platz 7 10587, Berlin Germany}

\section{Introduction}
\label{sec1}
In this work, we study the following stochastic optimisation problem
\[
    \min_{x\in\cK} \{\obj(x)\coloneqq\func(x)+\comp(x)=\mathbb{E}_\xi[\func(x;\xi)+\comp(x)]\},
\]
where $\func$ is a black-box, smooth, possibly nonconvex function, $\comp$ is a white box convex function, and $\cK\subseteq \Rd$ is a closed convex set. In many real-world applications, $\comp$ and $\cK$ are sparsity promoting, such as the black-box adversarial attack \cite{chen2018ead}, model agnostic methods for explaining machine learning models \cite{natesan2020model} and sparse cox regression \cite{liu2018zeroth}. Despite the low dimensional structure restricted by $\comp$ and $\cK$, standard stochastic mirror descent methods \cite{lan2020first} and the conditional gradient methods \cite{huang2020accelerated} have oracle complexity depending linearly on $d$ and are not optimal for high dimensional problems.

The gradient descent algorithm is dimensionality independent when the first-order information is available \cite{nesterov2003introductory}. For black-box objective functions, stronger dependence of the oracle complexity on dimensionality is caused by the biased gradient estimation \cite{jamieson2012query}. In \cite{wang2018stochastic}, the authors have proposed a LASSO-based gradient estimator for zeroth-order optimisation of unconstrained convex objective functions. Under the assumption of sparse gradients, the standard stochastic gradient descent with a LASSO-based gradient estimator has a weaker complexity dependence on dimensionality. The sparsity assumption has been further examined for nonconvex problems in \cite{balasubramanian2021zeroth}, which proves a similar oracle complexity of the zeroth-order stochastic gradient method with Gaussian smoothing. 

The critical issue of the algorithms mentioned above is the requirement of sparse gradients, which can not be expected in every application. We wish to improve the dependence on dimensionality by exploiting the low-dimensional structure defined by the objective function and constraints. For convex problems, this can be achieved by employing the mirror descent method with distance generating functions that are strongly convex w.r.t. $\norm{\cdot}_1$, such as the exponentiated gradient \cite{kivinen1997exponentiated,warmuth2007winnowing} or the $p$-norm algorithm \cite{duchi2015optimal}. However, a few problems arise if we apply these methods directly to optimising nonconvex functions. First, since these methods are essentially the gradient descent in $(\Rd, \norm{\cdot}_\infty)$, the convergence of the mirror descent algorithm requires variance reduction techniques in that space. Existing variance reduction techniques \cite{cutkosky2019momentum,lan2012optimal,pham2020proxsarah} are developed for the standard Euclidean space, and deriving convergence from the equivalence of the norms in $\Rd$ introduces additional complexity depending on $d$ \cite{ghadimi2016accelerated}. Secondly, the exponentiated gradient \cite{kivinen1997exponentiated} method and its extensions \cite{warmuth2007winnowing} work only for decision sets in the form of a simplex or cross-polytope with a known radius. Therefore, they can hardly be applied to general cases. The $p$-norm algorithm is more flexible and has an efficient implementation for $\ell_1$ regularised problems \cite{shalev2011stochastic}. However, handling $\ell_2$ regularised problems with the $p$-norm algorithm is challenging.

The ultimate target of this paper is to improve the complexity dependence on dimensionality. To achieve this, we first extend and analyse the adaptive stepsizes \cite{duchi2011adaptive,li2019convergence} for the stochastic composite mirror descent (SCMD) in a finite-dimensional normed space and prove that the convergence can be guaranteed without knowing the smoothness of $\func$. Then we improve the convergence by removing its dependence on the radius of the decision set achieved by adding a Frank-Wolfe style update step to SCMD. Combining the adaptive algorithms and an entropy-like distance-generating function allows us to perform gradient descent in $(\Rd,\norm{\cdot}_\infty)$. To improve the gradient estimation in that space, we use the mini-batch approach \cite{ghadimi2016mini} and show that the additional complexity introduced by switching the norms depends on $\ln d$ instead of $d$. Furthermore, we replace the gradient estimation methods applied in \cite{balasubramanian2021zeroth} and \cite{shamir2017optimal} with a smoothing method based on the Rademacher distribution. Our analysis shows that the total number of oracle calls required by our algorithms for finding an $\epsilon$-stationary point is bounded by $\mathcal{O}(\frac{\ln d}{\epsilon^4})$, which improves the complexity bound $\mathcal{O}(\frac{d}{\epsilon^4})$ attained by proximal stochastic gradient descent (ZO-PSGD) \cite{lan2020first}. We further improve the proposed algorithms by generalising the stochastic recursive momentum (STORM) algorithm proposed in \cite{NEURIPS2019_b8002139,NEURIPS2021_ac10ff19}. With modified stepsizes and the entropy-like distance generating function, our generalised version of STORM ensures an oracle complexity upper bounded by 
$\mathcal{\tilde O}(\frac{(\ln d)^2}{\epsilon^3})$ \footnote{We use $\Tilde{O}(\cdot)$ to hide the logarithmic terms involving $\epsilon$.}, which improves the oracle complexity of $\mathcal{\tilde O}(\frac{d^{\frac{3}{4}}}{\epsilon^3})$ achieved by the STORM based algorithm \cite{huang2022accelerated}. In addition to the theoretical analysis, we also demonstrate the performance of the developed algorithms in experiments on generating contrastive explanations of deep neural networks \cite{NEURIPS2018_c5ff2543}. 

The contributions of this paper are summarised as follows:
\begin{itemize}
    \item We generalise the adaptive step size for SCMD in the finite-dimensional Banach space.
    \item We combine SCMD with a Frank-Wolfe style update to remove its convergence dependence on the radius of the decision set.
    \item We analyse mini-batch and STORM in a finite-dimensional Banach space without using the Euclidean norm.
    \item Combining SCMD, the variance reduction techniques, an entropy-like distance-generating function and a Rademacher distribution-based sampling method, we obtain a family of zeroth-order optimisation algorithms for composite objective functions that have a logarithmic complexity dependence on dimensionality. 
\end{itemize}

The rest of the paper is organised as follows. Section~\ref{sec2} reviews related work. In section~\ref{sec3}, we present and analyse the algorithms based on mini-batch.  Section~\ref{secvr} generalises the variance reduction techniques. Section~\ref{sec4} demonstrates the empirical performance of the proposed algorithms. Finally, we conclude our work in Section~\ref{sec5}.
\section{Related Work}
\label{sec2}
Zeroth-order optimisation of nonconvex objective functions has many applications in machine learning and signal processing \cite{liu2020primer}. Algorithms for unconstrained nonconvex problems have been studied in \cite{ghadimi2013stochastic,lian2016comprehensive,nesterov2017random} and further enhanced with variance reduction techniques \cite{ji2019improved,liu2018zeroth1}. The high dimensional setting has been discussed in \cite{balasubramanian2021zeroth,wang2018stochastic}, in which algorithms with weaker complexity dependence on dimensionality are proposed. In practice, weaker dependence on dimensionality can also be achieved by applying the sparse perturbation techniques introduced in \cite{ohta2020sparse}.

It is popular to solve constrained problems with zeroth-order Frank-Wolfe algorithms \cite{balasubramanian2021zeroth,chen2020frank,huang2020accelerated}, which require the smoothness of the objective functions. We are motivated by the applications of adversarial attack and explanation methods based on the $\ell_1$ and $\ell_2$ regularisation \cite{chen2018ead,NEURIPS2018_c5ff2543,natesan2020model}, for which the objective functions contain non-smooth components. Our work is based on exploiting the low dimensional structure of the decision set, which has been discussed in \cite{gentile2003robustness,kivinen1997exponentiated,langford2009sparse,shalev2011stochastic,warmuth2007winnowing,shao_optimistic_2022} for online and stochastic optimization of convex functions and further extended for zeroth-order convex optimization in \cite{duchi2015optimal,shamir2017optimal}. To efficiently implement both $\ell_1$ and $\ell_2$ regularised problems, we use an entropy-like function as the distance-generating function in SCMD. Similar versions of the entropy-like function have previously been applied to online convex optimisation \cite{cutkosky2017online,orabona2013dimension,shao_optimistic_2022}.

Variance reduction techniques have been well-studied for unconstrained stochastic optimisation. Early approaches \cite{allen2016improved,johnson2013accelerating,lei2017non,nitanda2016accelerated,reddi2016stochastic} are based on checkpoints, at which the algorithms obtain accurate gradient evaluation. \cite{liu2018zeroth} has applied this idea for zeroth-order optimisation to improve the iteration complexity. The SARAH framework proposed in \cite{nguyen2017sarah} uses recursive gradients to reduce variance, which is also the key idea of the SPIDER algorithm for zeroth-order optimisation proposed in \cite{fang2018spider}. \cite{ji2019improved} has improved SPIDER by using a per-coordinate gradient estimation, which could be expensive for high-dimensional problems. Both SPIDER and SARAH can be extended for composite objectives \cite{ji2019improved,pham2020proxsarah}. 

All of the algorithms mentioned above require tuning some hyperparameters. The STORM algorithm and its variant \cite{NEURIPS2019_b8002139,NEURIPS2021_ac10ff19} use an adaptive stepsize \cite{duchi2011adaptive,li2019convergence,ward2020adagrad} and recursive gradients to reduce the variance in stochastic gradient descent. The Acc-ZOM \cite{huang2022accelerated} algorithm extends STORM for zeroth-order optimisation with constraints, however, it reintroduces a stepsize-like hyperparameter that has to be set proportional to the smoothness of the objective function. Despite the claimed adaptivity, the algorithms mentioned above still have some control parameters that need to be tuned in practice. Our algorithm generalises STORM for non-Euclidean geometry and uses a different stepsize scheduling to remove the control parameters.

\section{Generalised Adaptive Stochastic Composite Mirror Descent}
\label{sec3}
We start the theoretical analysis by introducing some important results of stochastic methods in a finite-dimensional vector space $\mathbb{X}$ equipped with some norm $\norm{\cdot}$. Let $\mathbb{X}_*$ be the dual space with dual norm $\dualnorm{\cdot}$. The bilinear map combining vectors from $\mathbb{X}_*$ and $\mathbb{X}$ is denoted by $\inner{\cdot}{\cdot}$. Based on the algorithms in the general setting, we then construct and analyse the corresponding zeroth-order algorithms in $\Rd$.
\subsection{Adaptive Stochastic Composite Mirror Descent}
\label{sec:adascmd}
Similar to the previous works on stochastic nonconvex optimisation \cite{lan2020first}, the following standard properties of the objective function $\func$ are assumed.
\begin{assumption}
\label{asp:smooth}
For any realisation $\xi$, $\func(\cdot;\xi)$ is $G$-Lipschitz and has $L$-Lipschitz continuous gradients with respect to $\norm{\cdot}$, i.e. \[\dualnorm{\grad \func(x;\xi)-\grad \func(y;\xi)}\leq L\norm{x-y},\] for all $x,y\in\mathbb{X}$, which implies
\[
\abs{\func(y;\xi)-\func(x;\xi)-\inner{\grad \func(x;\xi)}{y-x}}\leq \frac{L}{2}\norm{x-y}^2.
\]
\end{assumption}
\begin{assumption}
\label{asp:sg}
For any $x\in\mathbb{X}$, the stochastic gradient at $x$ is unbiased, i.e. \[\ex{\grad \func(x;\xi)}=\grad \func(x).\]
\end{assumption}
Assumption \ref{asp:smooth} and \ref{asp:sg} imply the $G$-smoothness and $L$-smoothness of $\func$ due to the inequalities
\[
\abs{\func(x)-\func(y)}\leq \ex{\abs{\func(x;\xi)-\func(y;\xi)}}\leq G\norm{x-y},
\]
and
\[
\dualnorm{\grad \func(x)-\grad \func(y)}\leq \ex{\dualnorm{\grad \func(x;\xi)-\grad \func(y;\xi)}}\leq L\norm{x-y}.
\]
Our idea is based on SCMD, which iteratively updates the decision variable following the rule given by 
\begin{equation}
\label{eq:update_md}
x_{t+1}=\arg\min_{x\in \cK}\inner{d_t}{x}+\comp(x)+\eta_{t}\bd{\scf}{x}{x_t},
\end{equation}
where $d_t$ is an estimation of the gradient $\grad f(x_t)$ and $\scf$ is a distance-generating function, i.e. $1$-strongly convex w.r.t. $\norm{\cdot}$. 
Define the generalised projection operator
\[
\proxi(x,d,\eta)=\arg\min_{y\in \cK}\inner{d}{y}+\comp(y)+\eta\bd{\scf}{y}{x}
\]
and the generalised gradient map
\[
\Grad(x,d,\eta) = \eta(x-\proxi(x,d,\eta)).
\]
Following the literature on the stochastic optimisation \cite{balasubramanian2021zeroth,lan2020first}, our goal is to find an $\epsilon$-stationary point $x_\tau$, i.e. $\ex{\norm{\Grad(x_\tau,\grad \func(x_\tau),\eta_{\tau})}^2}\leq \epsilon^2$. Given a sequence of estimated gradients, the convergence of SCMD is upper bounded by the following proposition, the proof of which can be found in the appendix.
\begin{proposition}
\label{lemma:omd}
Let $\sequ{d}$ be any sequence in $\mathbb{X}_*$ and $\sequ{x}$ be the sequence generated by \eqref{eq:update_md} with a distance-generating function $\scf$. Then, for any $\func$ satisfying the assumptions~\ref{asp:smooth} and \ref{asp:sg}, we have
\begin{equation}
\label{lemma:md:eq0} 
\begin{split}
&\ex{\sum_{t=1}^T\norm{\Grad(x_t,\grad \func(x_t),\eta_{t})}^2}\\
\leq& 6\sum_{t=1}^T\ex{\sigma_t^2}+4\ex{\sum_{t=1}^T\eta_t(\obj(x_t)-\obj(x_{t+1}))}\\
&+\ex{\sum_{t=1}^T\eta_t(2L-\eta_t)\norm{x_{t+1}-x_t}^2},\\
\end{split}
\end{equation}
where we denote by $\ex{\sigma_t^2}=\ex{\dualnorm{d_t-\grad \func(x_t)}^2}$ the variance of the gradient estimation.
\end{proposition}
Setting $\eta_1,\ldots,\eta_T=2L$, the convergence of SCMD depends on the convergence of the variance terms $\sequ{\sigma^2}$, which requires variance reduction techniques. 

In practice, it is difficult to obtain prior knowledge about $L$. To avoid the need for expensive tuning, we propose an adaptive algorithm with a similar convergence guarantee. The idea is similar to the adaptive stepsizes for unconstrained stochastic optimisation 
\cite{li2019convergence}, which sets $\eta_t=\sqrt{\sum_{s=1}^{t-1}\dualnorm{d_s}^2+\beta}$ for some $\beta>0$ to control the last term in \eqref{lemma:md:eq0}. For composite objectives, $\norm{x_{t+1}-x_t}^2$ depends not only on $d_t$ but also on $\grad \comp(x_{t+1})$, for which we set $\eta_t\propto\sqrt{\sum_{s=1}^{t-1}\norm{\Grad(x_t,d_t,\eta_t)}^2+1}$. To analyse the proposed method, we assume that the feasible decision set is contained in a closed ball.
\begin{assumption}
\label{asp:compact}
There is some $D>0$ such that $\norm{\proxi(x,d,\eta)}\leq D$ holds for all $\eta>0$, $x\in\cK$ and $d\in \mathbb{X}$.
\end{assumption}
Assumption~\ref{asp:compact} is typical in many composite optimisation problems with regularisation terms in their objective functions. In the following lemma, we propose and analyse the adaptive SCMD. Due to the compactness of the decision set, we can also assume that the objective function takes values from $[0,R]$.
\begin{assumption}
\label{asp:obj_value}
There is some $R>0$ such that $\obj(x)\in [0,R]$ holds for all $x\in\cK$.
\end{assumption}
\begin{theorem}
\label{lemma:adamd}
Given Assumptions \ref{asp:smooth}, \ref{asp:sg}, \ref{asp:compact} and \ref{asp:obj_value}, we define the sequence of stepsizes  
\[
\begin{split}
&\alpha_{t}=(\sum_{s=1}^{t-1}\lambda_s^2\alpha_s^2\norm{x_s-x_{s+1}}^2+1)^{\frac{1}{2}}\\
&\eta_t=\eta\alpha_t\\
\end{split}
\]
for some $0<\eta\leq\lambda_t\leq \lambda$. Furthermore we assume $D\eta\geq \frac{1}{2}$. Then we have 
\[
\begin{split}
&\ex{\sum_{t=1}^T\norm{\Grad(x_t,\grad \func(x_t),\eta_{t})}^2} \leq 13\sum_{t=1}^T\ex{\sigma_t^2}+C,
\end{split}
\]
where we define $C=33\lambda^2R^2+\frac{64L^2D}{\eta }(1+2D\eta)$.
\end{theorem}
\paragraph{Sketch of the proof} The proof starts with the direct application of Proposition~\ref{lemma:omd}. The focus is then to control the term $\sum_{t=1}^T\eta_t(2L-\eta_t)\norm{x_{t+1}-x_t}^2$. Since the sequence $\sequ{\eta}$ is increasing, we assume that $\eta_t>L$ starts from some index $t_0$. Then we only need to consider those stepsizes $\eta_1,\ldots,\eta_{t_0-1}$. Adding up $\sum_{t=1}^{t_{0}-2}\norm{x_{t+1}-x_t}^2$ yields a value proportional to $\eta_{t_0-1}$. Thus, the whole term is upper bounded by a constant. The complete proof can be found in the appendix.

The parameter $\eta$ is required when $\scf$ is not $1$-strongly convex. For locally strongly convex functions, where the decision set is implicitly defined, $\eta$ could be unknown. Therefore, we use a control parameter $\lambda_t$ in practice. Unlike gradient descent, we can not use the generalised gradient $\Grad(x_t,d_t,\eta_t)$ for setting $\eta_t$, which causes the convergence dependence on $D$ and requirement of the assumption on $D\eta $. Next, we apply the following Frank-Wolfe style update step to remove the dependence on $D$ and the assumption on $D\eta$. 
\begin{equation}
\begin{split}
\label{eq:update_md_m}
v_{t}=&\arg\min_{x\in \cK}\inner{d_t}{x}+\comp(x)+\eta\alpha_{t}\bd{\scf}{x}{x_t}\\
x_{t+1}=&(1-\frac{\alpha_t}{\alpha_{t+1}})x_t+\frac{\alpha_t}{\alpha_{t+1}}v_t\\
\end{split}
\end{equation}
\begin{theorem}
\label{lemma:adamd_p}
Let $\sequ{d}$ be any sequence in $\mathbb{X}_*$, $\sequ{x}$ the sequence generated by \eqref{eq:update_md_m} with a distance generating function $\scf$, and stepsize 
\[
\begin{split}
\alpha_{t+1}=&\max\{\sqrt{\sum_{s=1}^{t}\lambda_{s}^2\alpha_s^2\norm{v_s-x_s}^2},1\}\\
\eta_{t+1}=&\eta\alpha_{t+1}
\end{split}
\]
where $\eta>0$ is a constant and we set $\lambda\geq \lambda_t\geq \eta$ for all $t$.
Then, for any $\func$ satisfying Assumptions~\ref{asp:smooth}, \ref{asp:sg} and \ref{asp:compact}, we have
\[
\begin{split}    
\ex{\sum_{t=1}^T\norm{\Grad(x_t,\grad \func(x_t),\eta_t)}^2}\leq 8R^2\lambda^2+8R\eta+\frac{16L^2}{\eta^2}+10\sum_{t=1}^T\ex{\sigma_t^2},
\end{split}
\]
where we define $\sigma_t^2=\dualnorm{d_t-\grad f(x_t)}^2$.
\end{theorem}
With the adaptive stepsizes, no prior information about the problem is required. The convergence rate depends on the sequence of the variance-like quantity $\sequ{\sigma^2}$. We cannot obtain a converging sequence of $\sequ{\sigma^2}$ without using any variance reduction techniques or making any further assumptions. In the finite-dimensional vector space, where all norms are equivalent, we can surely reduce the variance by taking the average of the gradient estimation over a mini-batch. Since our idea is to perform gradient descent in $(\Rd,\norm{\cdot}_\infty)$, directly using the equivalence between $\norm{\cdot}_\infty$ and $\norm{\cdot}_2$ would introduce an additional dependence on dimensionality. The following lemma proves an upper bound for $(\mathbb{X},\norm{\cdot})$ using the smoothness of $\norm{\cdot}^2$.
\begin{lemma}
\label{lemma:gvr}
Let $(\mathbb{X},\norm{\cdot})$ be a finite-dimensional vector space. Assume that $\norm{\cdot}^2$ is $M$-strongly smooth w.r.t. $\norm{\cdot}$. Let $X_1,\ldots,X_m$ be independent random vectors in $\mathbb{X}$ such that $\ex{X_i}=\mu$ and $\ex{\norm{X_i-\mu}^2}\leq \sigma^2$ hold for all $i=1,\ldots,m$. Then we have
\begin{equation}
    \label{eq:vr}
    \begin{split}
    \ex{\norm{\frac{1}{m}\sum_{i=1}^mX_i-\mu}^2}\leq \frac{M}{2m}\sigma^2.
    \end{split}
\end{equation}
\end{lemma}
Lemma~\ref{lemma:gvr} allows us to analyse the mini-batch technique in $(\Rd,\norm{\cdot}_p)$ with $p=2\ln d$. Since the chosen $p$-norm is close to the maximum norm and $(4\ln d-2)$-strongly smooth, we can obtain a tighter bound. 

\subsection{Zeroth-Order Optimisation}
\label{sec:2p}
In \cite{balasubramanian2021zeroth}, the authors have proposed the two points estimation with Gaussian smoothing for estimating the gradient, the variance of which depends on $(\ln d)^2$. We argue that the logarithmic dependence on $d$ can be avoided. Our argument starts with reviewing the two points gradient estimation in the general setting. In this subsection, we simply assume $\mathbb{X}=\mathbb{X}_*$ and denote the inner product in $\mathbb{X}$ by $\inner{\cdot}{\cdot}$. Given a smoothing parameter $\nu>0$, some constant $\delta>0$, and a random vector $u\in\mathbb{X}$, we consider the two points estimation of the gradient given by 
\begin{equation}
\label{eq:grad_est}
\begin{split}
g_\nu(x;\xi)=\frac{\delta}{\nu}(\func(x+\nu u;\xi)-\func(x;\xi))u.
\end{split}
\end{equation}
To derive a general bound on the variance without specifying the distribution of $u$, we make the following assumption.
\begin{assumption}
\label{asp:sample}
Let $\mathcal{D}$ be a distribution with $\operatorname{supp}(\mathcal{D})\subseteq\mathbb{X}$. For $u\sim \mathcal{D}$, there is some $\delta>0$ such that 
\[
\mathbb{E}_u[\inner{g}{u}u]=\frac{g}{\delta}. 
\]
\end{assumption}
Given the existence of $\grad \func(x,\xi)$, Assumption~\ref{asp:sample} implies \[
\mathbb{E}_u[\inner{\grad \func(x,\xi)}{u}\delta u]=\grad \func(x,\xi).\] Together with the smoothness of $\func(\cdot,\xi)$, we obtain an estimation of $\grad \func(\cdot,\xi)$ with a controlled variance, which is described in the following lemma. Its proof can be found in the appendix.
\begin{lemma}
\label{lemma:bias}
Let $C$ be the constant such that $\norm{x}\leq C\dualnorm{x}$ holds for all $x\in\mathbb{X}$.
Define \[
\grad\func_\nu(x)=\mathbb{E}_u[\frac{\delta}{\nu}(\func(x+\nu u)-\func(x))u].
\]
Then the following inequalities hold for all $x\in\mathbb{X}$ and $\func$ satisfying Assumptions~\ref{asp:smooth}, \ref{asp:sg}, and \ref{asp:sample}.
\begin{enumerate}[label=\alph*)]
\item $\dualnorm{\grad \func_\nu(x)-\grad \func(x)}\leq \frac{\delta\nu C^2L}{2}\mathbb{E}_u[\dualnorm{u}^3]$.
\item $\mathbb{E}_u[\dualnorm{g_\nu(x;\xi)}^2]\leq  \frac{C^4L^2\delta^2\nu^2}{2}\mathbb{E}_u[\dualnorm{u}^6]+2\delta^2\mathbb{E}_u[\inner{\grad \func(x;\xi)}{u}^2\dualnorm{u}^2]$.
\end{enumerate}
\end{lemma}
For a realisation $\xi$ and a fixed decision variable $x_t$, $\ex{\sigma_t^2}$ can be upper bounded by combining the inequalities in Lemma~\ref{lemma:bias}. While most terms of the upper bound can be easily controlled by manipulating the smoothing parameter $\nu$, it is difficult to deal with the term $\delta^2\mathbb{E}_u[\inner{\grad \func(x;\xi)}{u}^2\dualnorm{u}^2]$. Intuitively, if we draw $u_1,\ldots,u_d$ from i.i.d. random variables with zero mean, $\delta^{-2}$ is related to the variance. However, small $\ex{\dualnorm{u}^k}$ indicates that $u_i$ must be centred around $0$, i.e., $\delta$ has to be large. 

\subsection{Mini-Batch Composite Mirror Descent for Non-Euclidean Geometry}
With the results in Subsections~\ref{sec:adascmd} and \ref{sec:2p}, we can construct a zeroth-order adaptive exponentiated gradient descent (ZO-AdaExpGrad)~\footnote{The distance-generating function is a symmetric version of the entropy function, which is used in the exponentiated gradient descent. Therefore, we also name our algorithms in the same way.} algorithm for decision sets contained in $(\Rd,\norm{\cdot}_1)$. We first analyse the sampling methods based on the Rademacher distribution adapted to the geometry of the maximum norm. 
\begin{lemma}
\label{lemma:variance}
Suppose that $\func$ is $L$-smooth w.r.t. $\norm{\cdot}_2$ and $\ex{\norm{\grad \func(x)-\grad \func(x,\xi)}^2_2}\leq \sigma^2$ for all $x\in\cK$. Let $u_{1},\ldots,u_{d}$ be independently sampled from the Rademacher distribution and
\[
g_\nu(x;\xi)=\frac{1}{\nu}(\func(x+\nu u;\xi))-\func(x;\xi))u
\]
be an estimation of $\grad f(x)$. Then we have
\begin{equation}
\label{lemma:variance:eq1}
    \begin{split}
        \ex{\norm{g_\nu(x;\xi)-\grad\func_\nu(x)}_\infty^2}\leq  \frac{3\nu^2d^2L^2}{2}+10\norm{\grad \func(x)}_2^2+8\sigma^2.\\
    \end{split}
\end{equation}
\end{lemma}
The dependence on $d^2$ in the first term of \eqref{lemma:variance:eq1} can be removed by choosing $v\propto \frac{1}{d}$, while the rest depends only on the variance of the stochastic gradient and the squared $\ell_2$ norm of the gradient. The upper bound in \eqref{lemma:variance:eq1} is better than the bound $(\ln d)^2(\norm{\grad \func(x)}^2_1+\norm{\grad \func(x;\xi)-\grad \func(x)}^2_1)$ attained by Gaussian smoothing \cite{balasubramanian2021zeroth}. Note that $g_\nu(x;\xi)$ is an unbiased estimator of $\grad \func_\nu(x)$. 
\begin{algorithm}
	\caption{ZO-AdaExpGrad}
    \label{alg:AdaExpGrad}
	\begin{algorithmic}
	\Require $m>0$, $\nu>0$, $x_1$ arbitrary and a sequence of positive values $\sequ{\eta}$
	\State Define $\scf:\Rd\to \mathbb{R}, x\mapsto \sum_{i=1}^d((\abs{x_i}+\frac{1}{d})\ln(d\abs{x_i}+1)-\abs{x_i})$
	\For{$t=1,\ldots,T$}
	    \State Sample $u_{t,j,i}$ from Rademacher distribution for $j=1,\ldots m$ and $i=1,\ldots d$
	    \State $d_t\coloneqq\frac{1}{m\nu}\sum_{j=1}^{m}(\func(x_t+\nu u_{t,j};\xi_{t,j})-\func(x_t;\xi_{t,j}))u_{t,j}$
	    \State $x_{t+1}=\arg\min_{x\in \cK}\inner{d_t}{x}+\comp(x)+\eta_{t}\bd{\scf}{x}{x_t}$
    \EndFor
    \State Sample $\tau$ from uniform distribution over $\{1,\ldots,T\}$.
    \State \textbf{Return} $x_{\tau}$
    \end{algorithmic}
\end{algorithm}
Algorithm~\ref{alg:AdaExpGrad} describes the adaptive composite mirror descent algorithm with an average of estimated gradient vectors
\begin{equation}
    \label{eq:mb}
    d_t=\frac{1}{m\nu}\sum_{j=1}^{m}(\func(x_t+\nu u_{t,j};\xi_{t,j})-\func(x_t;\xi_{t,j}))u_{t,j},
\end{equation}
and the update-generating function given by 
\begin{equation}
    \label{eq:reg}
\scf:\Rd\to \mathbb{R}, x\mapsto \sum_{i=1}^d((\abs{x_i}+\frac{1}{d})\ln(d\abs{x_i}+1)-\abs{x_i})
\end{equation}
to update $x_{t+1}$ at iteration $t$. The next lemma proves the strict convexity of $\scf$.
\begin{lemma}
\label{lemma:property}
For all $x, y \in \Rd$, we have 
\[
\scf(y)-\scf(x)\geq  \inner{\grad\scf(x)}{y-x}+\frac{1}{2(\max\{\norm{x}_1,\norm{y}_1\}+1)}\norm{y-x}_1^2
\]
\end{lemma}
The proof of Lemma~\ref{lemma:property} can be found in the appendix. If the feasible decision set is contained in an $\ell_1$ ball with radius $D$, then the function $\scf$ defined in \eqref{eq:reg} is $\frac{1}{D+1}$-strongly convex w.r.t $\norm{\cdot}_1$. With $\scf$, update~\eqref{eq:update_md} is equivalent to mirror descent with stepsize $\frac{\eta_t}{D+1}$ and the distance-generating function $(D+1)\scf$. The performance of Algorithm~\ref{alg:AdaExpGrad} is described in the following theorem. 
\begin{theorem}
\label{thm:zo_ada_md}
Assume \ref{asp:smooth}, \ref{asp:sg} for $\norm{\cdot}=\norm{\cdot}_2$, \ref{asp:compact} for $\norm{\cdot}=\norm{\cdot}_1$ and \ref{asp:obj_value}. Then running Algorithm~\ref{alg:AdaExpGrad} with $m=2Te(2\ln d-1)$, $\nu=\frac{1}{d\sqrt{T}}$ and  $\eta_1=,\ldots,=\eta_T=2L(D+1)$ guarantees
\begin{equation}
\label{lemma:main:eq0}
\begin{split}
\ex{\norm{\Grad(x_\tau,\grad \func(x_\tau),\eta_\tau)}^2_1}\leq&\sqrt{\frac{2e(2\ln d-1)}{mT}}(6V+8LR),\\
\end{split}
\end{equation}
where we define $V=10G^2+8\sigma^2+2L^2$.
Furthermore, setting 
\[
\begin{split}
\lambda_t=&\frac{1}{\max\{\norm{x_t}_1,\norm{x_{t+1}}_1\}+1}\\
\alpha_t=&\sqrt{\sum_{s=1}^{t-1}\lambda_s^2\alpha_s^2\norm{x_{s+1}-x_s}^2_1+1},
\end{split}
\]
we have 
\[
\begin{split}
&\ex{\norm{\Grad(x_\tau,\grad \func(x_\tau),\eta_\tau)}_1^2} \leq\sqrt{\frac{2e(2\ln d-1)}{mT}}(13V+C),
\end{split}
\]
where we define $C=33R^2+192L^2D(D+1)$.
\end{theorem}
A similar algorithm can be constructed using update rule~\eqref{eq:update_md_m}. 
\begin{algorithm}
	\caption{ZO-AdaExpGrad$+$}
    \label{alg:AdaExpM}
	\begin{algorithmic}
	\Require $m>0$, $\nu>0$, $x_1$ arbitrary and a sequence of positive values $\sequ{\eta}$
	\State Define $\scf:\Rd\to \mathbb{R}, x\mapsto \sum_{i=1}^d((\abs{x_i}+\frac{1}{d})\ln(d\abs{x_i}+1)-\abs{x_i})$
	\For{$t=1,\ldots,T$}
	    \State Sample $u_{t,j,i}$ from Rademacher distribution for $j=1,\ldots m$ and $i=1,\ldots d$
	    \State $d_t\coloneqq\frac{1}{m\nu}\sum_{j=1}^{m}(\func(x_t+\nu u_{t,j};\xi_{t,j})-\func(x_t;\xi_{t,j}))u_{t,j}$
	    \State $v_t=\arg\min_{x\in \cK}\inner{g_t}{x}+\comp(x)+\eta\alpha_t\bd{\scf}{x}{x_t}$
        \State $x_{t+1}=(1-\frac{\alpha_{t}}{\alpha_{t+1}})x_t+\frac{\alpha_{t}}{\alpha_{t+1}}v_t$
    \EndFor
    \State Sample $\tau$ from uniform distribution over $\{1,\ldots,T\}$.
    \State \textbf{Return} $x_{\tau}$
    \end{algorithmic}
\end{algorithm}
\begin{theorem}
\label{thm:zo_ada_md_p}
Assume \ref{asp:smooth}, \ref{asp:sg} for $\norm{\cdot}=\norm{\cdot}_2$, \ref{asp:compact} for $\norm{\cdot}=\norm{\cdot}_1$ and \ref{asp:obj_value}. Then running algorithm~\ref{alg:AdaExpGrad} with $m=2Te(2\ln d-1)$, $\nu=\frac{1}{d\sqrt{T}}$ and step size
\[
\begin{split}
\lambda_t=&\frac{1}{\max\{\norm{x_t}_1,\norm{v_{t}}_1\}+1}\\
\alpha_t=&\sqrt{\max\{\sum_{s=1}^{t-1}\lambda_s^2\alpha_s^2\norm{v_{s}-x_s},1\}},
\end{split}
\]
we have 
\[
\ex{\norm{\Grad(x_\tau,\grad \func(x_\tau),\eta_{\tau})}_1^2} \leq \sqrt{\frac{2e(2\ln d-1)}{mT}}(10V+C),
\]
where we define $C=8R^2+8R+16L^2(D+1)^2$ and $V=10G^2+8\sigma^2+2L^2$.
\end{theorem}
The total number of oracle calls for finding an $\epsilon$-stationary point is upper bounded by $\mathcal{O}(\frac{\ln d}{\epsilon^4})$, which has a weaker dependence on dimensionality compared to $\mathcal{O}(\frac{d}{\epsilon^4})$ achieved by ZO-PSGD \cite{lan2020first}. 

The convergence dependence on $D$ of Algorithm~\ref{alg:AdaExpM} is due to the local strong convexity of the symmetric entropy function. This can be avoided by using update-generating function $\frac{1}{2(p-1)}\norm{\cdot}_p^2$ for $p=1+\frac{1}{2\ln d-1}$. Since the mirror map at $x$ depends on $\norm{x}_p$, it is difficult to handle the popular $\ell_2$ regulariser. Our algorithms have an efficient implementation for Elastic Net regularisation, which is described in the appendix. 
\section{Generalised Stochastic Recursive Momentum}
\label{secvr}
In this section, we extend the STORM algorithm \cite{NEURIPS2019_b8002139,NEURIPS2021_ac10ff19} to our setting. Similar to the previous section, we start with analysing the adaptive momentum for the general SCMD in a finite-dimensional Banach space $(\mathbb{X},\norm{\cdot})$. 
\subsection{Generalised Stochastic Recursive Momentum}  
Following \cite{NEURIPS2019_b8002139}, we run SCMD with stochastic recursive gradient $d_t$ given by  
\begin{equation}
\begin{split}
    \label{eq:moment}
    d_t=&g_t+(1-\gamma_t)(d_{t-1}-m_{t})\\
\end{split}
\end{equation}
The first step is to generalise the key technical lemma for analysing STORM \cite[Lemma 2]{NEURIPS2019_b8002139} using the smoothness of $\dualnorm{\cdot}^2$.
\begin{lemma}
\label{lemma:storm}
Let $\sequ{d}$ be recursively defined according to \eqref{eq:moment} with sequences of random vectors $\sequ{g}$ and $\sequ{m}$. Assume the $M$-strongly smoothness of $\dualnorm{\cdot}^2$, $\ex{g_t|d_{t-1}}=\mu_t$ and $\ex{m_t|d_{t-1}}=\mu_{t-1}$. Define $\epsilon_{t}=d_t-\mu_t$. Setting $\gamma_t=\frac{2}{1+\tau_t}$, $\tau_t=(1+t\gamma)^{-\frac{2}{3}}$ and $\gamma\leq 1$, we have 
    \begin{equation}
    \label{lemma:storm:eq0}
    \begin{split}
        \ex{\sum_{t=1}^T \dualnorm{\epsilon_{t}}^2}\leq&\sum_{t=1}^T\frac{6M}{5\tau_t}\ex{\dualnorm{g_t-\mu_t}^2}+\sum_{t=1}^T\frac{3M(\tau_t-1)^2}{5\tau_t}\ex{\dualnorm{g_t-m_t}^2}.\\
        \end{split}
    \end{equation}
\end{lemma}
Similar to the analysis in \cite{NEURIPS2019_b8002139}, the first term on the right-hand side of \eqref{lemma:storm:eq0} can be upper bounded by $\mathcal{O}(T^\frac{1}{3})$. The next theorem proves the convergence of update rule \eqref{eq:update_md_m} with $\sequ{d}$ generated by \eqref{eq:moment} for both zeroth and first-order algorithms. To use different norms for analysing $\epsilon_{t}$ and $\dualnorm{d_t-\grad \func(x_t)}$, we simply assume the inequality in Lemma~\ref{lemma:gvr}. 
\begin{theorem}
\label{lemma:ada_storm}
Let $\sequ{d}$, $\sequ{g}$ and $\sequ{m}$ be recursively defined according to \eqref{eq:moment} with $\ex{g_t|d_{t-1}}=\mu_t$ and $\ex{m_t|d_{t-1}}=\mu_{t-1}$. Assume there are constants $M$, $C_1$ and $C_2$ such that
\[
    \begin{split}
        \ex{\sum_{t=1}^T \dualnorm{\epsilon_{t}}^2}\leq&\sum_{t=1}^T\frac{6M}{5\tau_t}\ex{\dualnorm{g_t-\mu_t}^2}+\sum_{t=1}^T\frac{3M(\tau_t-1)^2}{5\tau_t}\ex{\dualnorm{g_t-m_t}^2}\\
        \end{split}
\]
and 
\[
\ex{\dualnorm{g_t-m_t}^2}\leq C_1\norm{x_t-x_{t-1}}^2+C_2
\]
hold for all $t$.
Let $\sequ{x}$ be the sequence generated by \eqref{eq:update_md_m} with recursively defined parameters
\[
\begin{split}
\tau_{t}=&(1+\gamma t)^{\frac{2}{3}}\\
\gamma_t=&\frac{2}{1+\tau_t}\\
\beta_t=&\max\{1,\frac{(\tau_{t}-1)}{\sqrt{\tau_{t}}}\}\\
\alpha_{t}=&\sqrt{\beta_t(1+\sum_{s=1}^{t-1}\lambda_{s}^2\alpha_s^2\norm{v_s-x_s}^2)}\\
\eta_{t}=&\eta\alpha_t
\end{split}
\]
where we set $\gamma\leq 1$, $\lambda\geq \lambda_t\geq \eta>0$ for all $t$.
Then, for any $\func$ satisfying Assumptions~\ref{asp:smooth}, \ref{asp:sg}, and \ref{asp:compact}, we have
\[
\begin{split}
&\ex{\sum_{t=1}^T\norm{\Grad(x_t,\grad \func(x_t),\eta_t)}^2}\\
\leq&16R^2\lambda^2T^\frac{1}{3}+\frac{4\eta^2}{\lambda^2}+16R\eta+\frac{32L^2}{\eta^2}\\
&+\frac{648M\tilde\sigma^2T^\frac{1}{3}}{5\gamma^\frac{2}{3}}+\frac{108MC_2T^{\frac{5}{3}}}{5}+36T\bar\sigma\\
&+\frac{108C_1MT^{\frac{1}{3}}}{5}\ln\frac{54C_1MT^\frac{2}{3}}{5\eta^2},\\
\end{split}
\]
where we assume $\ex{\dualnorm{g_t-\mu_t}^2}\leq \tilde\sigma^2$ and $\ex{\dualnorm{\mu_t-\grad \func(x_t)}^2}\leq \bar\sigma^2$.
\end{theorem}
For first-order algorithm, where we set $g_t=\grad \func(x_t,\xi_t)$ and $m_t=\grad \func(x_{t-1},\xi_t)$, we have $C_1=L^2$ and $\bar\sigma=C_2=0$ for $L$-smooth $\func$. With the two points estimation of the gradients, $C_1$ is related to the sampling method, while $C_2$ and $\bar\sigma$ are controlled by the smoothing parameter, which is proved in the following lemma.
\begin{lemma}
\label{lemma:bias_diff}
Let $C$ be the constant such that $\norm{x}\leq C\dualnorm{x}$ holds for all $x\in\mathbb{X}$. Define the gradient estimation
\[
g_\nu(x;\xi,u)=\frac{\delta}{\nu}(\func(x+\nu u,\xi)-\func(x,\xi))u.
\]
Then the following inequality holds for all $x,y\in\mathbb{X}$ and $\func$ satisfying Assumptions~\ref{asp:smooth}, \ref{asp:sg}, and \ref{asp:sample}.
\[
\begin{split}
&\mathbb{E}_u[{\dualnorm{g_\nu(x;\xi,u)-g_\nu(y;\xi,u)}^2}]\\
\leq &\frac{3C^4L^2\delta^2\nu^2}{4}\mathbb{E}_u[\dualnorm{u}^6]+3\delta^2\mathbb{E}_u[\inner{\grad \func(x;\xi)-\grad \func(y;\xi)}{u}^2\dualnorm{u}^2].\\
\end{split}
\]
\end{lemma}

\subsection{Zeroth-Order Stochastic Recursive Gradient}
Algorithm~\ref{alg:AdaExpStorm} describes a zeroth-order algorithm based on \eqref{eq:update_md_m} and \eqref{eq:moment}. Its performance is analysed in the following theorem. 
\begin{algorithm}
	\caption{ZO-ExpStorm}
    \label{alg:AdaExpStorm}
	\begin{algorithmic}
	\Require $m>0$, $\nu>0$, $x_1$ arbitrary and a sequence of positive values $\sequ{\eta}$
	\State Define $\scf:\Rd\to \mathbb{R}, x\mapsto \sum_{i=1}^d((\abs{x_i}+\frac{1}{d})\ln(d\abs{x_i}+1)-\abs{x_i})$
	\For{$t=1,\ldots,T$}
	    \State Sample $u_{t,j,i}$ from Rademacher distribution for $j=1,\ldots m$ and $i=1,\ldots d$
	    \State $g_t\coloneqq\frac{1}{m\nu}\sum_{j=1}^{m}(\func(x_t+\nu u_{t,j};\xi_{t,j})-\func(x_t;\xi_{t,j}))u_{t,j}$
	    \State $m_t\coloneqq\frac{1}{m\nu}\sum_{j=1}^{m}(\func(x_{t-1}+\nu u_{t,j};\xi_{t,j})-\func(x_{t-1};\xi_{t,j}))u_{t,j}$
        \State $d_t=g_t+(1-\gamma_t)(d_{t-1}-m_{t})$\\
	    \State $v_t=\arg\min_{x\in \cK}\inner{d_t}{x}+\comp(x)+\eta\alpha_t\bd{\scf}{x}{x_t}$
        \State $x_{t+1}=(1-\frac{\alpha_{t}}{\alpha_{t+1}})x_t+\frac{\alpha_{t}}{\alpha_{t+1}}v_t$
    \EndFor
    \State Sample $\tau$ from uniform distribution over $\{1,\ldots,T\}$.
    \State \textbf{Return} $x_{\tau}$
    \end{algorithmic}
\end{algorithm}
\begin{theorem}
\label{thm:zo_exp_storm}
Assume \ref{asp:smooth}, \ref{asp:sg} for $\norm{\cdot}=\norm{\cdot}_2$, \ref{asp:compact} for $\norm{\cdot}=\norm{\cdot}_1$ and \ref{asp:obj_value}. 
Then running Algorithm~\ref{alg:AdaExpStorm} with $m\geq 1$, $\nu\leq d^{-1}T^{-\frac{2}{3}}$ and recursively defined paramters
\[
\begin{split}
\tau_{t}=&(1+\frac{t}{m})^{\frac{2}{3}}\\
\gamma_t=&\frac{2}{1+\tau_t}\\
\beta_t=&\max\{1,\frac{(\tau_{t}-1)}{\sqrt{\tau_{t}}}\}\\
\alpha_{t}=&\sqrt{\beta_t(1+\sum_{s=1}^{t-1}\lambda_{s}^2\alpha_s^2\norm{v_s-x_s}_1^2)},\\
\end{split}
\]
we have 
\[
\begin{split}
&\ex{\norm{\Grad(x_\tau,\grad \func(x_\tau),\eta_\tau)}_1^2}\\
\leq&16R^2T^{-\frac{2}{3}}+T^{-1}(4+16R+32L^2(D+1)^2)\\
&+\frac{648e^2(4\ln d-2)^2(5G^2+4\sigma^2)m^{-\frac{1}{3}}T^{-\frac{2}{3}}}{5}+\frac{192e(2\ln d-1)L^2T^{-\frac{2}{3}}}{5}\\
&+72L^2T^{-\frac{2}{3}}+\frac{648L^2e(2\ln d-1)T^{-\frac{2}{3}}}{5}\ln\frac{324(D+1)^2L^2e(2\ln d-1)T^\frac{2}{3}}{5}.\\
\end{split}
\]
\end{theorem}
Theorem~\ref{thm:zo_exp_storm} gives an oracle complexity of $\mathcal{\tilde O}(\frac{(\ln d)^2}{\epsilon^{3}})$. Unlike the first-order algorithms, the estimated gradient is not Lipschtz-continuous. To ensure that $\ex{\norm{g_t-m_t}_\infty^2}$ can be upper bounded by $\ex{C_1\norm{x_t-x_{t-1}}^2_1}$, we still need to sample a mini-batch for $g_t$ and $m_t$ in practice. 
\section{Experiments}
\label{sec4}
We examine the performance of our algorithms for generating the contrastive explanations of classification models \cite{NEURIPS2018_c5ff2543}, which consist of a set of pertinent positive (PP) features and a set of pertinent negative (PN) features\footnote{The source code is available at \url{https://github.com/VergiliusShao/highdimzo}}. For a given sample $x_0\in \mathcal{X}$ and classification model $f:\mathcal{X}\to \mathbb{R}^K $, the contrastive explanation can be found by solving the following optimisation problem \cite{NEURIPS2018_c5ff2543} 
\[
\begin{split}
\min_{x\in\mathcal{\cK}}\quad &l_{x_0}(x)=c_{x_0}(x)+\ln(1+\exp(-c_{x_0}(x)))+\lambda_1\norm{x}_1+\frac{\lambda_2}{2}\norm{x}_2^2.\\
\end{split}
\]
Let $k_0=\arg\max_{i}f(x_0)_i$ represent the prediction of $x_0$. The cost function $c_{x_0}$ for finding PP is then given by
\[
\begin{split}
c_{x_0}(x)=&\max_{i\neq k_0}f(x)_i-f(x)_{k_0}\\
\end{split}
\]
and PN is modelled by the following cost function
\[
\begin{split}
c_{x_0}(x)=&f(x_0+x)_{k_0}-\max_{i\neq k_0}f(x_0+x)_{i}\\
\end{split}
\]
In the experiments, we first train a LeNet model \cite{lecun1989handwritten} on the MNIST dataset \cite{lecun1989handwritten} and a ResNet$20$ model \cite{7780459} on the CIFAR-$10$ dataset \cite{krizhevsky2009learning}, which attain a test accuracy of $96\%$, $91\%$, respectively. For each class of the images, we randomly pick $20$ correctly classified images from the test dataset and generate PP and PN for them. We set $\lambda_1=\lambda_2=2^{-4}$ for MNIST dataset, and choose $\{x\in\Rd|0\leq x_i\leq x_{0,i}\}$ and $\{x\in\Rd|x_i\geq 0, x_i+x_{0,i}\leq 1\}$ as the decision set for PP and PN, respectively. For CIFAR-$10$ dataset, we set $\lambda_1=\lambda_2=2^{-1}$. ResNet$20$ takes normalised data as input, and images in CIFAR-$10$ do not have an obvious background colour. Therefore, we choose $\{x\in\Rd|\min\{0,x_{0,i}\}\leq x_i\leq\max\{0,x_{0,i}\}\}$ and $\{x\in\Rd|0\leq (x_i+x_{0,i})\nu_i+\mu_i\leq 1\}$, where $\nu_i$ and $\mu_i$ are the mean and variance of the dimension $i$ of the training data, as the decision set for PP and PN, respectively. The search for PP and PN starts from $x_0$ and the centre of the decision set, respectively.

Our baseline methods are ZO-PSGD \cite{lan2020first}, Acc-ZOM \cite{huang2022accelerated} and AO-ExpGrad \cite{shao_optimistic_2022}. We fix the mini-batch size $m=200$ for all candidate algorithms to conduct a fair comparison study. Following the analysis of \cite[Corollary 6.10]{lan2020first}, the optimal oracle complexity $\mathcal{O}(\frac{d}{\epsilon^2})$ of ZO-PSGD is obtained by setting $m=dT$ and $\nu=T^{-\frac{1}{2}}d^{-1}=m^{-\frac{1}{2}}d^{-\frac{1}{2}}$. The smoothing parameters for ZO-ExpGrad, ZO-AdaExpGrad, ZO-ExpGrad$+$ and AO-ExpGrad are set to $\nu=m^{-\frac{1}{2}}(2e(2\ln d-1))^{\frac{1}{2}}d^{-1}$ according to Theorems~\ref{thm:zo_ada_md}, \ref{thm:zo_ada_md_p} and the experiment setting in \cite{shao_optimistic_2022}. We choose $\nu=d^{-1}T^{-\frac{2}{3}}$ for ZO-ExpStorm and Acc-ZOM according to Theorem~\ref{thm:zo_exp_storm} and \cite[Theorem 1]{huang2022accelerated}. For ZO-PSGD, ZO-ExpGrad, multiple constant stepsizes $\eta_t\in\{10^i|0\leq i\leq 3\}$ are tested. Acc-ZOM has an important stepsize-like hyperparameter $\gamma$ that should be set proportional to the smoothness of the loss function to ensure convergence. We examine the performance of Acc-ZOM with multiple $\gamma\in\{10^{-i}|0\leq i\leq 3\}$. The rest of the control parameters of Acc-ZOM are set according to \cite[Theorem 3 and Section 8.1]{huang2022accelerated}.

\begin{figure}[htbp]
\centering
\begin{subfigure}{.5\textwidth}
  \centering
  \includegraphics[width=\linewidth]{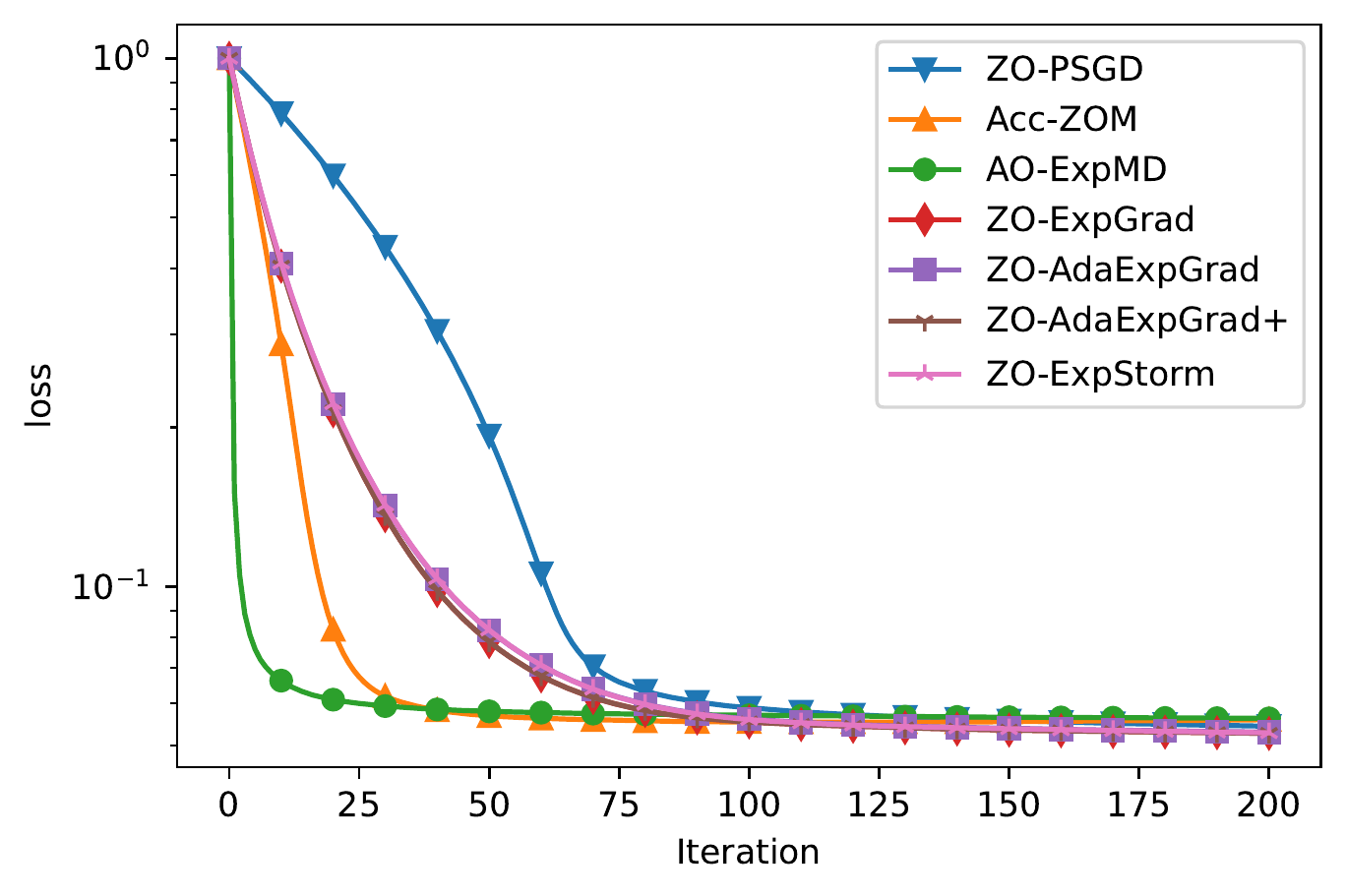}
\caption{Convergence for Generating \textbf{PN}}%
\label{fig:bb-pn-mnist}
\end{subfigure}%
\begin{subfigure}{.5\textwidth}
  \centering
  \includegraphics[width=\linewidth]{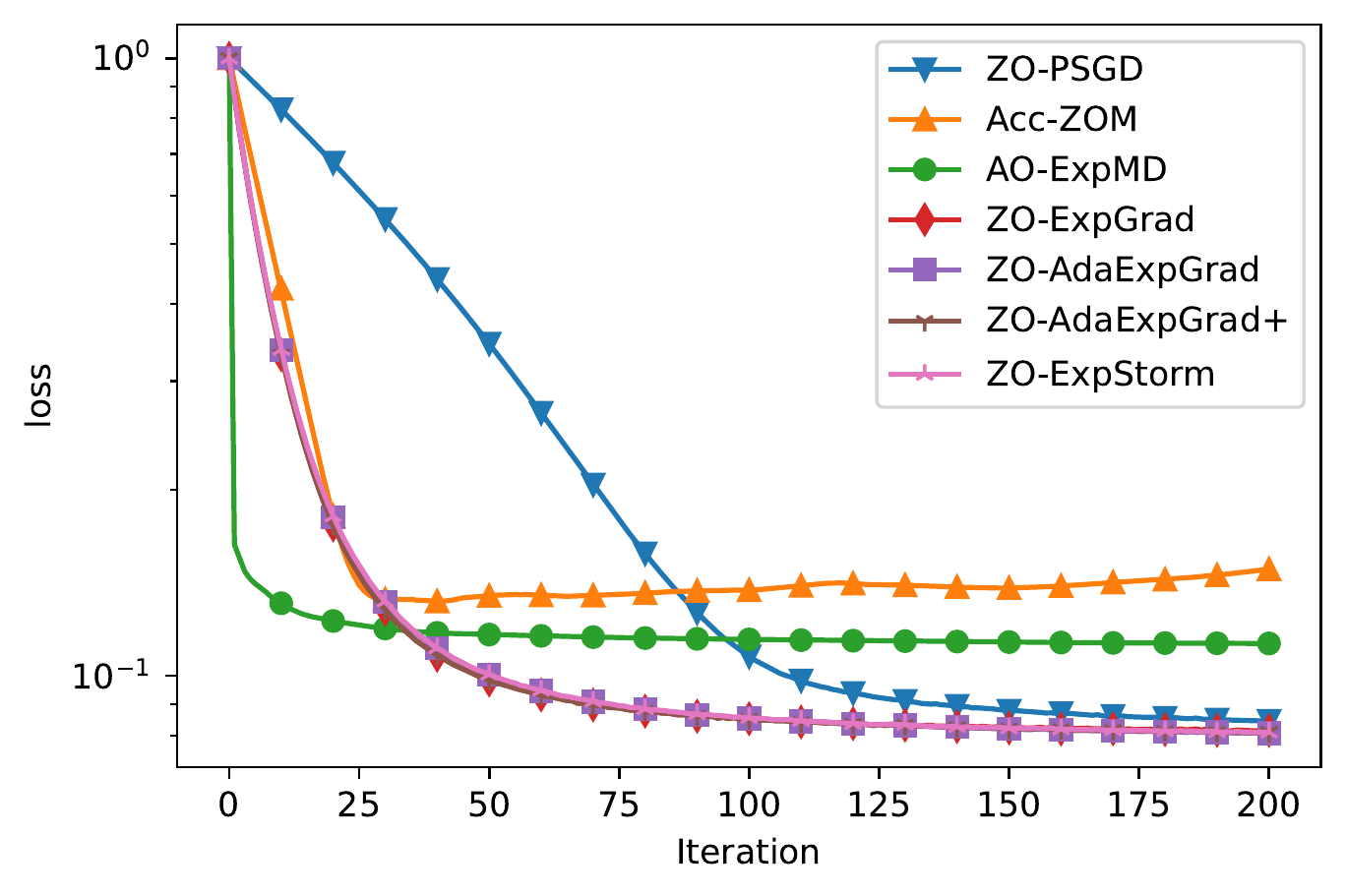}
\caption{Convergence for Generating  \textbf{PP}}%
\label{fig:bb-pp-mnist}
\end{subfigure}
\caption{Black Box Contrastive Explanations on MNIST}
\label{fig:BB-mnist}
\end{figure}

Figure~\ref{fig:BB-mnist} presents the convergence behaviour of the candidate algorithms with the best choice of hyperparameters, averaging over 200 images from the MNIST dataset. 
For generating PN, AO-ExpMD, which is an accelerated mirror descent algorithm with the entropy-like distance-generating function for convex problems, quickly converges to a saddle point and is then outperformed by other algorithms. 
Acc-ZOM converges fast in the first $30$ iterations and is slightly outperformed afterwards by our proposed algorithms. 
For generating PP, the algorithms based on the entropy-like distance-generating function have clear advantages in the first 50 iterations. 
AO-ExpGrad converges to a saddle point and then is outperformed by our proposed algorithms, which achieve the best overall performance. 

\begin{figure}[htbp]
\centering
\begin{subfigure}{.5\textwidth}
  \centering
  \includegraphics[width=\linewidth]{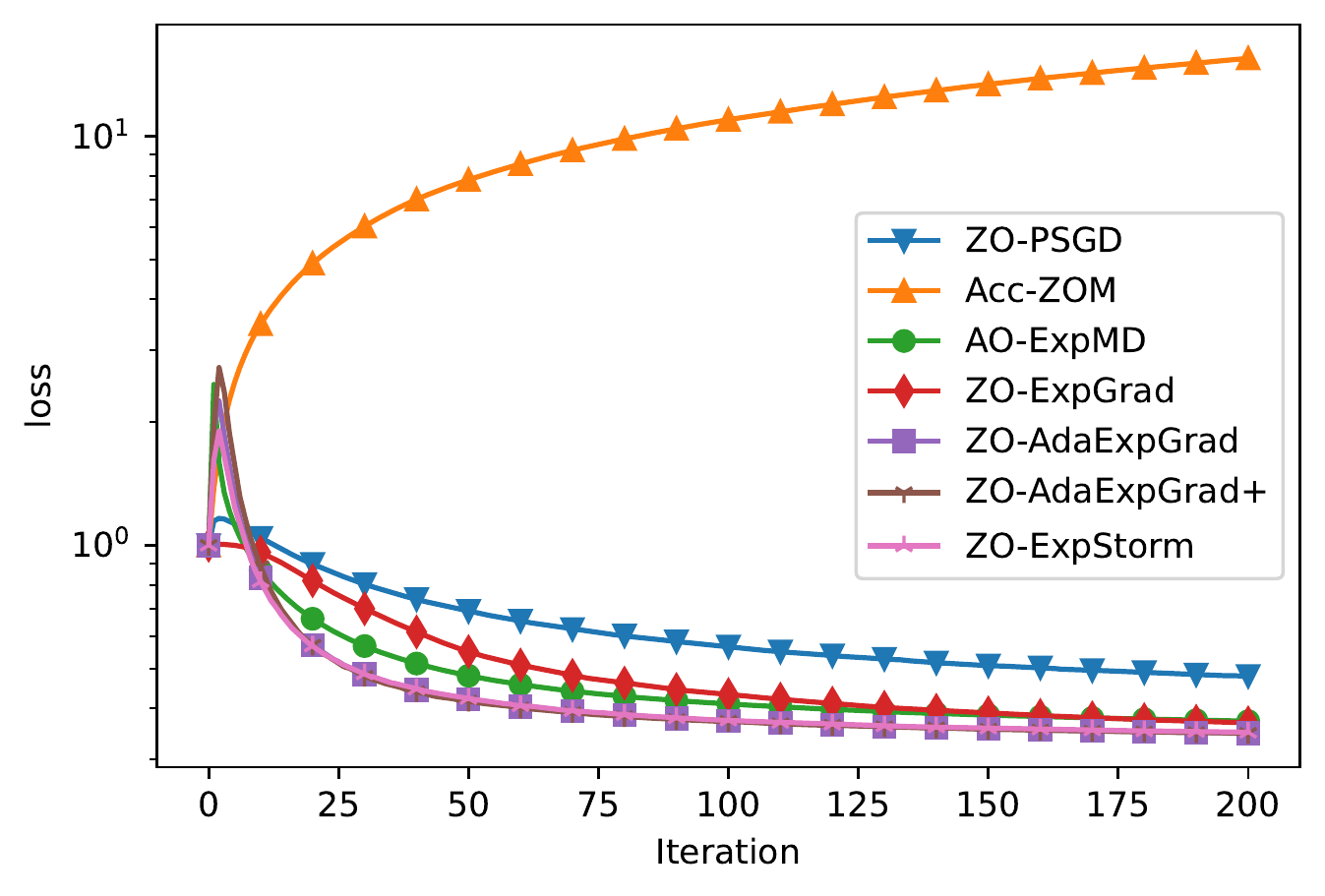}
\caption{Convergence for Generating \textbf{PN}}%
\label{fig:bb-pn}
\end{subfigure}%
\begin{subfigure}{.5\textwidth}
  \centering
  \includegraphics[width=\linewidth]{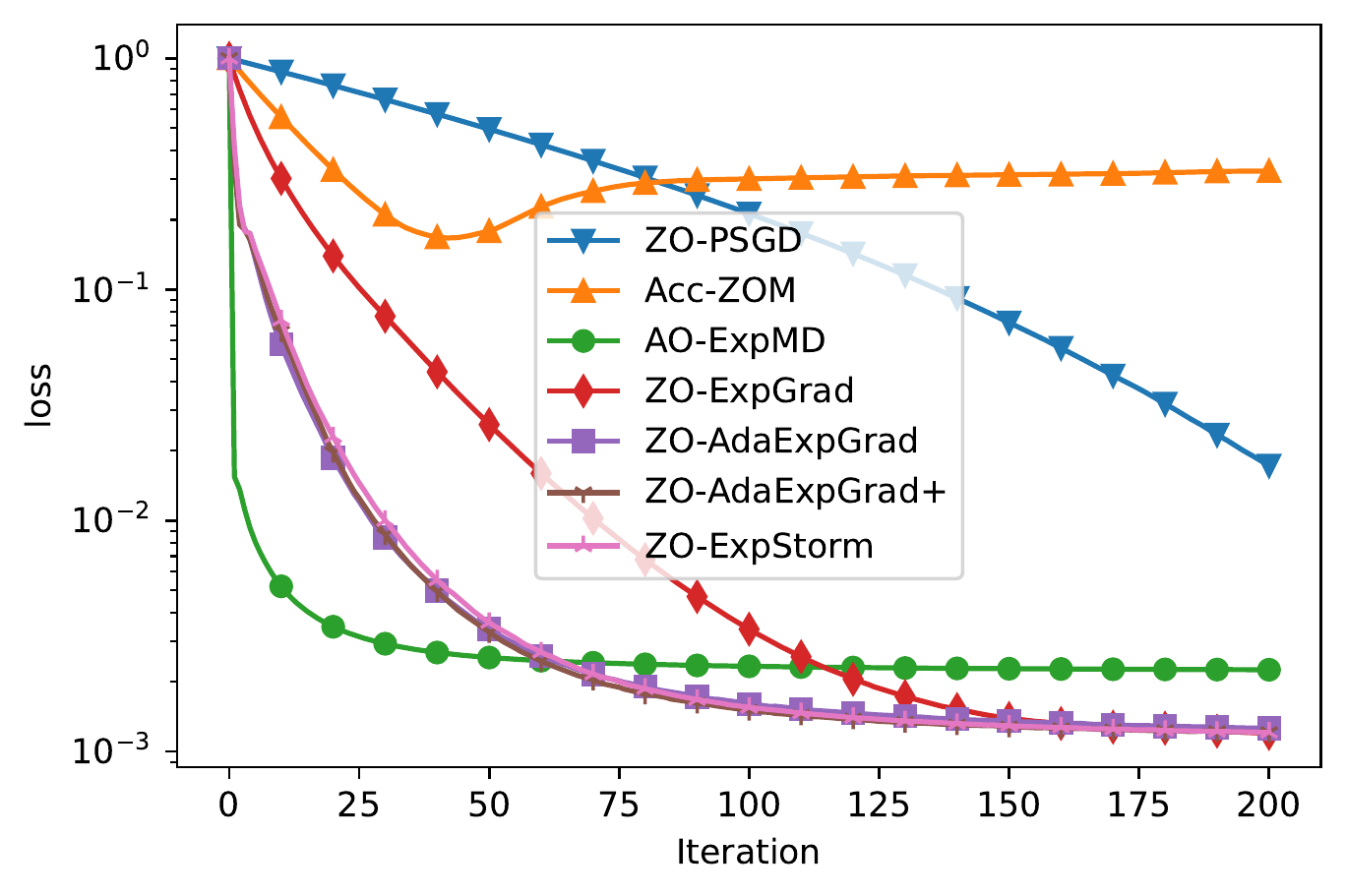}
\caption{Convergence for Generating  \textbf{PP}}%
\label{fig:bb-pp}
\end{subfigure}
\caption{Black Box Contrastive Explanations on CIFAR-$10$}
\label{fig:BB}
\end{figure}

Figure~\ref{fig:BB} depicts the convergence behaviour of candidate algorithms averaging over $200$ images from the CIFAR-$10$ dataset, which has higher dimensionality than the MNIST dataset. 
As observed, the advantage of our algorithms becomes more significant. Acc-ZOM, which has a decent performance on the MNIST dataset, fails to converge despite tuning hyperparameters. The experimental results of Acc-ZOM with different $\gamma$ can be found in the appendix. 

Furthermore, choices of stepsizes have a clear impact on the performances of both ZO-ExpGrad and ZO-PSGD, which are plotted in the appendix. Notably, ZO-AdaExpGrad and ZO-AdaExpGrad$+$ converge as fast as ZO-ExpGrad with well-tuned stepsizes in the experiments on MNIST. They have a significant advantage over ZO-ExpGrad for CIFAR-$10$. Overall, the proposed algorithms outperform the state-of-the-art algorithms. However, compared to each other, they perform similarly except for generating PP for CIFAR-$10$. Some zoomed-in plots can be found in the appendix. The STORM-based algorithm does not have significant advantages in our experiments. Possible reasons include the low variance in the maximum-normed space and non-Lipschitz continuity of the gradient estimation.

\section{Conclusion}
Motivated by applications in black-box adversarial attacks and generating model-agnostic explanations of machine learning models, we propose and analyse a family of generalised adaptive SCMD algorithms and their applications in the zeroth-order optimisation of nonconvex objective functions. Combining several algorithmic ideas, such as the entropy-like distance generating function, the sampling method based on the Rademacher distribution and the variance reduction method for non-Euclidean geometry, our algorithms have an oracle complexity depending logarithmically on dimensionality without prior knowledge about the problem. The performance of our algorithms is firmly backed by theoretical analysis and examined in experiments for generating explanations of machine learning models. 

The variance reduction techniques do not have a clear advantage in our experiments. This could be caused by the non-smoothness of the loss, and the low variance involved in the gradient estimation, which has been improved by SCMD in the maximum-normed space. As a future research direction, we plan to systematically examine the performance of the proposed algorithms in experiments with additional real-world applications, such as untargeted adversarial attacks \cite{guo2019simple} and training deep neural networks. 
\label{sec5}
\subsubsection*{Acknowledgements}
A preliminary version of this article will appear in the proceedings of the 8\textsuperscript{th} Annual Conference on Machine Learning, Optimization and Data Science (LOD 2022) with the title ``Adaptive Zeroth-Oder Optimisation of Nonconvex Composite Objectives".

\section*{Declarations}
\subsection*{Funding}
The research leading to these results received funding from the German Federal Ministry for Economic Affairs and Climate Action under Grant Agreement No. 01MK20002C.
\subsection*{Code availability}
The implementation of the experiments and all algorithms involved in the experiments are available on GitHub \url{https://github.com/VergiliusShao/highdimzo}.
\subsection*{Availability of data and materials}
The source code generating synthetic data, creating neural networks and model training are available on GitHub \url{https://github.com/VergiliusShao/highdimzo}. The MNIST data can be found in \url{http://yann.lecun.com/exdb/mnist/}. The CIFAR-10 data are collected from \url{https://www.cs.toronto.edu/~kriz/cifar.html}.
\subsection*{Conflicts of Interests and Competing Interests}
The authors declare that they have no conflicts of interests or competing interests.
\subsection*{Ethics Approval}
Not Applicable.
\subsection*{Consent to Participate}
Not Applicable
\subsection*{Consent for Publication}
Not Applicable
\subsection*{Authors' Contributions}
Conceptualization: WS; Methodology: WS; Formal analysis and investigation: WS; Software: WS; Validation: WS, FS; Visualization: WS; Writing - original draft preparation: WS; Writing - review and editing: WS, FS; Funding acquisition: SA; Resources: SA; Supervision: FS, SA.

%
%
%
\bibliographystyle{plain}
\bibliography{lib}
\appendix
\section{Missing Proofs}
\subsection{Proof of Proposition~\ref{lemma:omd}}
\begin{proof}[Proof of Proposition~\ref{lemma:omd}]
First of all, we have
\begin{equation}
\label{lemma:md:eq1}
\begin{split}
&\obj(x_{t+1})-\obj(x_t)\\
\leq &\inner{\grad \func(x_t)+\grad \comp(x_{t+1})}{x_{t+1}-x_t}+\frac{L}{2}\norm{x_{t+1}-x_t}^2\\
\leq &\inner{\eta_t\grad\scf(x_{t+1})-\eta_t\grad\scf(x_{t})}{x_t-x_{t+1}}\\
&+\inner{\grad \func(x_t)-g_t}{x_{t+1}-x_t}+\frac{L}{2}\norm{x_{t+1}-x_t}^2\\
\leq &-\eta_{t}\norm{x_{t+1}-x_t}^2+\inner{\grad \func(x_t)-d_t}{x_{t+1}-x_t}+\frac{L}{2}\norm{x_{t+1}-x_t}^2\\
\leq &-\eta_{t}\norm{x_{t+1}-x_t}^2+\frac{1}{\eta_{t}}\sigma_t^2+\frac{\eta_{t}\norm{x_t-x_{t+1}}^2}{4}+\frac{L}{2}\norm{x_{t+1}-x_t}^2\\
= &-\frac{\eta_{t}}{2}\norm{x_{t+1}-x_t}^2+\frac{1}{\eta_{t}}\sigma_t^2+(\frac{L}{2}-\frac{\eta_{t}}{4})\norm{x_{t+1}-x_t}^2\\
= &-\frac{1}{2\eta_{t}}\norm{\Grad(x_t,d_t,\eta_t)}^2+\frac{1}{\eta_{t}}\sigma_t^2+(\frac{L}{2}-\frac{\eta_{t}}{4})\norm{x_{t+1}-x_t}^2,\\
\end{split}
\end{equation}
where the first inequality uses the $L$-smoothness of $\func$ and the convexity of $\comp$, the second inequality follows from the optimality condition of the update rule, the third inequality is obtained from the strongly convexity of $\scf$ and the fourth line follows from the definition of dual norm. It follows from the $\frac{1}{\eta_t}$ Lipschitz continuity \cite[Lemma 6.4]{lan2020first} of $\proxi(x_t,\cdot,\eta_{t})$ that $\Grad(x_t,\cdot,\eta_{t})$ is $1$-Lipschitz. Thus, we obtain
\begin{equation}
\label{lemma:md:eq2}
\begin{split}
&\norm{\Grad(x_t,\grad \func(x_t),\eta_{t})}^2\\
\leq &2\norm{\Grad(x_t,\grad \func(x_t),\eta_{t})-\Grad(x_t,d_t,\eta_{t})}^2+2\norm{\Grad(x_t,d_t,\eta_{t})}^2\\
\leq &2\sigma_t^2+2\norm{\Grad(x_t,d_t,\eta_{t})}^2\\
\leq &6\sigma_t^2+4\eta_t(\obj(x_t)-\obj(x_{t+1}))+\eta_t(2L-\eta_t)\norm{x_{t+1}-x_t}^2.\\
\end{split}
\end{equation}
Adding up from $1$ to $T$ and taking expectation, we have
\begin{equation}
\label{lemma:md:eq3}
\begin{split}
&\ex{\sum_{t=1}^T\norm{\Grad(x_t,\grad \func(x_t),\eta_{t})}^2}\\ \leq&6\sum_{t=1}^T\ex{\sigma_t^2}+4\ex{\sum_{t=1}^T\eta_t(\obj(x_t)-\obj(x_{t+1}))}+\ex{\sum_{t=1}^T\eta_t(2L-\eta_t)\norm{x_{t+1}-x_t}^2},\\
\end{split}
\end{equation}
which is the claimed result.
\end{proof}

\subsection{Proof of Theorem~\ref{lemma:adamd}}
\begin{proof}[Proof of Theorem~\ref{lemma:adamd}]
Applying proposition~\ref{lemma:omd}, we obtain
\begin{equation}
\label{lemma:adamd:eq1}
\begin{split}
&\ex{\sum_{t=1}^T\norm{\Grad(x_t,\grad \func(x_t),\eta_{t})}^2}\\ \leq&6\sum_{t=1}^T\ex{\sigma_t^2}+4\ex{\sum_{t=1}^T\eta_t(\obj(x_t)-\obj(x_{t+1}))}+\ex{\sum_{t=1}^T\eta_t(2L-\eta_t)\norm{x_{t+1}-x_t}^2}.\\
\end{split}
\end{equation}
W.l.o.g., we can assume $\obj(x_0)=0$ and $\eta_1\geq\eta_0>0$, since they are artefacts in the analysis.
The second term of the upper bound above can be rewritten into
\begin{equation}
\label{lemma:adamd:eq2}
\begin{split}
\sum_{t=1}^T\eta_t(\obj(x_t)-\obj(x_{t+1}))= &\eta_1\obj(x_0)-\eta_T\obj(x_{T+1})+\sum_{t=1}^T(\eta_t-\eta_{t-1})\obj(x_t)\\
\leq &R\sum_{t=1}^T(\eta_t-\eta_{t-1})\\
\leq &R\eta_T\\
\leq &4\lambda^2R^2+\frac{1}{16\lambda^2}\eta_T^2,\\
\end{split}
\end{equation}
where the first inequality follows from $\obj(x_0)=0$, $\eta_t\geq \eta_{t-1}$, $\obj(x_{t})\geq 0$ and the last line uses the H\"older's inequality.
Using the definition of $\eta_T$, we have
\begin{equation}
    \label{lemma:adamd:eq2.1}
    \begin{split}
    \frac{1}{16\lambda^2}\eta_T^2\leq&\frac{\eta^2}{16\lambda^2}\sum_{t=1}^T\lambda_s^2\alpha_s^2\norm{x_{s+1}-x_s}^2+\frac{1}{16}    \\
    \leq&\frac{1}{16}\sum_{t=1}^T\norm{\Grad(x_t,d_t,\eta_t)}+\frac{1}{16}\\
    \leq&\frac{1}{8}\sum_{t=1}^T\norm{\Grad(x_t,\grad \func(x_t),\eta_t)}+\frac{1}{8}\sum_{t=1}^T\sigma_t^2+\frac{1}{16}\\
    \end{split}
\end{equation}
Next, we define the index
\[ 
\begin{split}
t_0=\begin{cases}
    \min\{1\leq t\leq T|\eta_t>L\} ,& \text{if }\{1\leq t\leq T|\eta_t>2L\}\neq \emptyset\\
    T,              & \text{otherwise}.
\end{cases}
\end{split}
\]
Then, the third term in \eqref{lemma:adamd:eq1} can be bounded by
\begin{equation}
\label{lemma:adamd:eq3}
\begin{split}
&\sum_{t=1}^T\eta_t(2L-\eta_t)\norm{x_{t+1}-x_t}^2\\
=&\sum_{t=1}^{t_0-1}\eta_t(2L-\eta_t)\norm{x_{t+1}-x_t}^2+\sum_{t=t_0}^T\eta_t(2L-\eta_t)\norm{x_{t+1}-x_t}^2\\
=& 2L\sum_{t=1}^{t_0-1}\frac{\alpha_t\eta_t\norm{x_{t+1}-x_t}^2}{\alpha_t}\\
=& \frac{2\sqrt{2}L}{\eta}\sum_{t=1}^{t_0-1}\frac{\eta_t^2\norm{x_{t+1}-x_t}^2}{\sqrt{2\sum_{s=1}^{t-1}\lambda_s^2\alpha_s^2\norm{x_{s+1}-x_s}^2+2}}\\
\leq& 8LD\sum_{t=1}^{t_0-1}\frac{\eta_t^2\norm{x_{t+1}-x_t}^2}{\sqrt{\sum_{s=1}^{t-1}\lambda_s^2\alpha_s^2\norm{x_{s+1}-x_s}^2+8\eta^2\alpha_t^2D^2}}\\
\leq& 8LD\sum_{t=1}^{t_0-1}\frac{\eta_t^2\norm{x_{t+1}-x_t}^2}{\sqrt{\sum_{s=1}^{t}\eta^2\alpha_s^2\norm{x_{s+1}-x_s}^2}}\\
\leq& 16LD\sqrt{\sum_{t=1}^{t_0-1}\eta^2\alpha_t^2\norm{x_{t+1}-x_t}^2}\\
\leq& 16LD(\alpha_{t_0-1}+2\eta D\alpha_{t_0-1})\\
\leq& \frac{16LD}{\eta}(1+2D\eta)\eta_{t_0-1}\\
\leq& \frac{32L^2D}{\eta}(1+2D\eta),\\
\end{split}
\end{equation}
where we use the assumption $\eta D\geq \frac{1}{2}$ for the first inequality, apply \cite[lemma 6]{shao_optimistic_2022} for the third inequality and the rest inequalities follow from the assumptions on $\lambda_t$,$D$ and $\eta_{t_0-1}$.  
Combining \eqref{lemma:adamd:eq1}, \eqref{lemma:adamd:eq2}, \eqref{lemma:adamd:eq2.1} and \eqref{lemma:adamd:eq3}, we have
\begin{equation}
\label{lemma:adamd:eq4}
\begin{split}
&\ex{\sum_{t=1}^T\norm{\Grad(x_t,\grad \func(x_t),\eta_{t})}^2}\\ 
\leq&13\sum_{t=1}^T\ex{\sigma_t^2}+(\frac{1}{2}+32\lambda^2R^2)+\frac{64L^2D}{\eta }(1+2D\eta).\\
\end{split}
\end{equation}
For simplicity and w.l.o.g., we can assume $\lambda^2R^2\geq \frac{1}{2}$. Define $C=33\lambda^2R^2+\frac{64L^2D}{\eta }(1+2D\eta)$, we obtain the claimed result.
\end{proof}

\subsection{Proof of Theorem~\ref{lemma:adamd_p}}
\begin{proof}[Proof of Theorem~\ref{lemma:adamd_p}]
From the smoothness of $\func$, it follows
\[
\begin{split}
&\eta_{t+1}\obj(x_{t+1})-\obj(x_t)\\
\leq &\eta_{t+1}\inner{\grad \func(x_t)}{x_{t+1}-x_t}+\eta_{t+1}\comp(x_{t+1})-
\eta_{t+1}\comp(x_t)+\frac{L\eta_{t+1}}{2}\norm{x_{t+1}-x_t}^2\\
\leq &\eta_{t}\inner{\grad \func(x_t)}{v_t-x_t}+\eta_{t}(\comp(v_t)-\comp(x_t))+\frac{L\eta\alpha_{t}^2}{2\alpha_{t+1}}\norm{v_t-x_t}^2\\
\leq &\eta_{t}\inner{\grad\scf(v_t)-\grad\scf(x_{t})}{x_t-v_t}+\eta_{t}\inner{\grad \func(x_t)-d_t}{v_{t}-x_t}+\frac{L\eta\alpha_{t}^2}{2\alpha_{t+1}}\norm{v_{t}-x_t}^2\\
\leq &-\eta_{t}^2\norm{v_t-x_t}^2+\eta_{t}\inner{\grad \func(x_t)-d_t}{v_t-x_t}+\frac{L\eta\alpha_{t}^2}{2\alpha_{t+1}}\norm{v_{t}-x_t}^2\\
\leq &-\eta_{t}^2\norm{v_t-x_t}^2+\sigma_t^2+\frac{\eta_{t}^2\norm{x_t-v_t}^2}{4}+\frac{L\eta\alpha_{t}^2}{2\alpha_{t+1}}\norm{v_t-x_t}^2\\
= &-\frac{\eta_{t}^2}{2}\norm{v_t-x_t}^2+\sigma_t^2+(\frac{L\eta\alpha_{t}^2}{2\alpha_{t+1}}-\frac{\eta_{t}^2}{4})\norm{v_{t}-x_t}^2,\\
\end{split}
\]
where the first inequality uses the $L$-smoothness of $\func$, the second inequality follows from the update rule and the convexity of $\comp$, the third inequality follows from the optimality condition of the update rule, the fourth inequality is obtained from the strong convexity of $\scf$, and the fifth line follows from the definition of dual norm. 
Rearranging and adding up from $1$ to $T$, we obtain 
\begin{equation}
\begin{split}    
\label{lemma:md_p:eq1}
0\leq &\sum_{t=1}^T\eta_{t+1}(\obj(x_{t+1})-\obj(x_{t}))\\
&+\sum_{t=1}^T(\frac{L\eta\alpha_{t}^2}{2\alpha_{t+1}}-\frac{\eta_{t}^2}{4})\norm{v_{t}-x_t}^2\\
&-\sum_{t=1}^T\frac{\eta_{t}^2}{2}\norm{v_t-x_t}^2+\sum_{t=1}^T\sigma_t^2.\\
\end{split}
\end{equation}
Next, since $\sequ{\eta}$ is monotone increasing and $\obj$ and takes values from $[0,R]$, we can further upper bound the first term of \eqref{lemma:md_p:eq1} by
\[
\begin{split}    
&\eta\sum_{t=1}^T\alpha_{t+1}(\obj(x_{t})-\obj(x_{t+1}))\\
=&\eta\sum_{t=1}^T\obj(x_t)(\alpha_{t+1}-\alpha_{t})+\alpha_1\eta\obj(x_1)-\alpha_{T+1}\eta\obj(x_{T+1})\\
\leq&\eta\sum_{t=1}^T\obj(x_t)(\alpha_{t+1}-\alpha_{t})+R\eta\\
\leq&R\eta \sum_{t=1}^T(\alpha_{t+1}-\alpha_{t})+R\eta\\
\leq&R\eta\alpha_{T+1}+R\eta.\\
\leq&R^2\lambda^2+R\eta+\frac{\eta^2\alpha_{T+1}^2}{4\lambda^2},\\
\end{split}
\]
where the first and second inequalities use assumption \ref{asp:obj_value}, the third inequality follows from the monotonicity of $\sequ{\alpha}$, and the last inequality follows from the H\"older's inequality.
W.l.o.g, we can assume $\sum_{t=1}^T\alpha_t^2\lambda_{t}^2\norm{v_t-x_t}^2\geq 1$. Thus, we obtain
\begin{equation}
\label{lemma:md_p:eq2}
    \begin{split}
        \sum_{t=1}^T\eta_{t+1}(\obj(x_{t})-\obj(x_{t+1}))\leq &R^2\lambda^2+R\eta+\frac{1}{4}\sum_{t=1}^T\eta_{t}^2\norm{v_t-x_t}^2
    \end{split}
\end{equation}
To bound the second term of \eqref{lemma:md_p:eq1}, we define the index
\[
\begin{split}
t_0=\begin{cases}
    \max\{1\leq t\leq T|\eta_{t+1}\leq 2L\} ,& \text{if }\{1\leq t\leq T|\eta_{t+1}\leq 2L\}\neq \emptyset\\
    T,              & \text{otherwise}.
\end{cases}
\end{split}
\]
Since $\sequ{\eta}$ is increasing, we have 
\begin{equation}
\label{lemma:md_p:eq3}
    \begin{split}
        &\sum_{t=1}^T(\frac{L\eta\alpha_{t}^2}{2\alpha_{t+1}}-\frac{\eta_{t}^2}{4})\norm{v_{t}-x_t}^2\\
        =&\sum_{t=1}^{t_0}(\frac{L\eta\alpha_{t}^2}{2\alpha_{t+1}}-\frac{\eta_{t}^2}{4})\norm{v_{t}-x_t}^2+\sum_{t=t_0+1}^T(\frac{L\eta\alpha_{t}^2}{2\alpha_{t+1}}-\frac{\eta_{t}^2}{4})\norm{v_{t}-x_t}^2\\
        =&\frac{L\eta}{2}\sum_{t=1}^{t_0}\frac{\alpha_{t}^2}{\alpha_{t+1}}\norm{v_{t}-x_t}^2\\
        \leq&\frac{L}{2\eta}\sum_{t=1}^{t_0}\frac{\eta^2\alpha_{t}^2\norm{v_{t}-x_t}^2}{\sqrt{\sum_{s=1}^{t}\alpha_{s}^2\lambda_s^2\norm{v_{s}-x_s}^2}}\\
        \leq&\frac{L}{2\eta}\sum_{t=1}^{t_0}\frac{\alpha_{t}^2\lambda_t^2\norm{v_{t}-x_t}^2}{\sqrt{\sum_{s=1}^{t}\alpha_{s}^2\lambda_s^2\norm{v_{s}-x_s}^2}}\\
        \leq&\frac{L}{\eta}\sqrt{\sum_{t=1}^{t_0}\alpha_{t}^2\lambda_t^2\norm{v_{t}-x_t}^2}\\
        \leq&\frac{L\eta_{t_0+1}}{\eta^2}\\
        \leq&\frac{2L^2}{\eta^2}.\\
    \end{split}
\end{equation}
Combining \eqref{lemma:md_p:eq1}, \eqref{lemma:md_p:eq2} and  \eqref{lemma:md_p:eq3} we obtain
\[
\begin{split}    
0\leq &8R^2\lambda^2+8R\eta+\frac{16L^2}{\eta^2}-2\sum_{t=1}^T\norm{\Grad(x_t,d_t,\eta_t)}^2+8\sum_{t=1}^T\sigma_t^2.\\
\end{split}
\]
From the Lipschitz continuity of $\Grad(x,\cdot,\eta)$ \cite[lemma 6.4]{lan2020first}, we have
\[
\sum_{t=1}^T\norm{\Grad(x_t,\grad \func(x_t),\eta_t)}^2\leq 2\sum_{t=1}^T\norm{\Grad(x_t, d_t,\eta_t)}^2+2\sum_{t=1}^T\sigma_t^2.
\]
Combining the inequalities above and rearranging and taking the average, we have
\[
\begin{split}    
\sum_{t=1}^T\norm{\Grad(x_t,\grad \func(x_t),\eta_t)}^2\leq 8R^2\lambda^2+8R\eta+\frac{16L^2}{\eta^2}+10\sum_{t=1}^T\sigma_t^2,
\end{split}
\]
which is the claimed result.
\end{proof}

\subsection{Proof of Lemma~\ref{lemma:gvr}}
\begin{proof}[Proof of Lemma~\ref{lemma:gvr}]
From the $M$-smoothness of $\norm{\cdot}^2$, it follows
\begin{equation}
    \label{lemma:gvr:eq1}
    \norm{x+y}^2\leq \norm{x}^2+\inner{g_x}{y}+\frac{M}{2}\norm{y}^2,
\end{equation}
for all $x,y\in\mathbb{X}$ and $g_x\in\partial \norm{\cdot}^2(x)$.
Next, let $X$ and $Y$ be independent random vectors in $\mathbb{X}$ with $\ex{X}=\ex{Y}=0$. Using \eqref{lemma:gvr:eq1}, we have
\begin{equation}
    \label{lemma:gvr:eq2}
\begin{split}
    \ex{\norm{X+Y}^2}\leq& \ex{\norm{X}^2}+\ex{\inner{g_X}{Y}}+\frac{M}{2}\ex{\norm{Y}^2}\\
     = &\ex{\norm{X}^2}+\inner{\ex{g_X}}{\ex{Y}}+\frac{M}{2}\ex{\norm{Y}^2}\\
     = &\ex{\norm{X}^2}+\frac{M}{2}\ex{\norm{Y}^2},\\
\end{split}
\end{equation}\
Note that $X_1-\mu,\ldots,X_m-\mu$ are i.i.d. random variables with zero mean. Combining \eqref{lemma:gvr:eq2} with a simple induction on $m$, we obtain
\begin{equation}
    \label{lemma:vr:eq3}
\begin{split}
    \ex{\norm{\sum_{i=1}^m(X_i-\mu)}^2}\leq&\ex{\norm{\sum_{i=1}^{m-1}(X_i-\mu)}^2}+\frac{M}{2}\ex{\norm{X_m-\mu}^2}\\
    \leq&\ex{\norm{\sum_{i=1}^{m-1}(X_i-\mu)}^2}+\frac{M}{2}\sigma^2\\
    \leq&\frac{Mm}{2}\sigma^2\\
\end{split}
\end{equation}
The desired result is obtained by dividing both sides by $m^2$.
\end{proof}

\subsection{Proof of Lemma~\ref{lemma:bias}}
\begin{proof}[Proof of Lemma~\ref{lemma:bias}]
Let $\grad \func_\nu(x)$ be as defined in \eqref{eq:grad_est}, then we have
\begin{equation}
\label{lemma:bias:eq3}
\begin{split}
&\dualnorm{\grad \func_\nu(x)-\grad \func(x)}\\
= &\dualnorm{ \mathbb{E}_u[\frac{\delta}{\nu}(\func(x+ \nu u)-\func(x))u]-\grad \func(x)}\\
= &\frac{\delta}{\nu}\dualnorm{ \mathbb{E}_u[(\func(x+\nu u)-\func(x)-\inner{\grad \func(x)}{\nu u})u]}\\
\leq &\frac{\delta}{\nu} \mathbb{E}_u[(\func(x+\nu u)-\func(x)-\inner{\grad \func(x)}{\nu u})\dualnorm{u}]\\
\leq &\frac{\delta\nu C^2L}{2}\mathbb{E}_u[\dualnorm{u}^3]\\
\end{split}
\end{equation}
where the second equality follows from the assumption~\ref{asp:sample}, the third line uses the Jensen's inequality and the last line follows the $L$ smoothness of $\func$.
Next, we have
\begin{equation}
\label{lemma:bias:eq4}
\begin{split}
&\mathbb{E}_u[{\dualnorm{g_\nu(x;\xi)}^2}]\\
=&\mathbb{E}_u[\frac{\delta^2}{\nu^2}\abs{\func(x+\nu u;\xi)-\func(x;\xi)}^2\dualnorm{u}^2]\\
=&\frac{\delta^2}{\nu^2}\mathbb{E}_u[(\func(x+\nu u;\xi)-\func(x;\xi)-\inner{\grad \func(x;\xi)}{\nu u}+\inner{\grad \func(x;\xi)}{\nu u})^2\dualnorm{u}^2]\\
\leq&\frac{2\delta^2}{\nu^2}\mathbb{E}_u[(\func(x+\nu u;\xi)-\func(x;\xi)-\inner{\grad \func(x;\xi)}{\nu u})^2\dualnorm{u}^2]\\
&+\frac{2\delta^2}{\nu^2}\mathbb{E}_u[\inner{\grad \func(x;\xi)}{\nu u}^2\dualnorm{u}^2]\\
\leq &\frac{C^4L^2\delta^2\nu^2}{2}\mathbb{E}_u[\dualnorm{u}^6]+2\delta^2\mathbb{E}_u[\inner{\grad \func(x;\xi)}{u}^2\dualnorm{u}^2],\\
\end{split}
\end{equation}
which is the claimed result.
\end{proof}

\subsection{Proof of Lemma~\ref{lemma:variance}}
\begin{proof}[Proof of Lemma~\ref{lemma:variance}]
We clearly have $\ex{uu^\top}=I$. From lemma~\ref{lemma:bias} with the constant $C=\sqrt{d}$ and $\delta=1$, it follows 
\begin{equation}
    \label{lemma:variance:eq2}
    \begin{split}
        \ex{\norm{g_\nu(x;\xi)}_\infty^2}\leq &\ex{\frac{d^2L^2\nu^2}{2}\mathbb{E}_u[\norm{u}_\infty^6]+2\mathbb{E}_u[\inner{\grad \func(x;\xi)}{u}^2\norm{u}_\infty^2]}\\
        \leq &\ex{\frac{d^2L^2\nu^2}{2}+2\mathbb{E}_u[\inner{\grad \func(x;\xi)}{u}^2]}\\
        \leq &\frac{d^2L^2\nu^2}{2}+2\ex{\norm{\grad \func(x;\xi)}_2^2}\\
        \leq &\frac{d^2L^2\nu^2}{2}+4\ex{\norm{\grad \func(x)-\grad \func(x;\xi)}_2^2}+4\norm{\grad \func(x)}_2^2\\
        \leq &\frac{d^2L^2\nu^2}{2}+4\sigma^2+4\norm{\grad \func(x)}_2^2\\
    \end{split}
\end{equation}
where the second inequality uses the fact the $\norm{u}_\infty\leq 1$ and the third inequality follows from the Khintchine inequality. The variance is controlled by
\begin{equation}
    \label{lemma:variance:eq3}
    \begin{split}
        &\ex{\norm{g_\nu(x;\xi)-\grad\func_\nu(x)}_\infty^2}\\
        \leq &2\ex{\norm{g_\nu(x;\xi)}_\infty^2}+2\norm{\grad\func_\nu(x)}_\infty^2\\
        \leq & \nu^2d^2L^2+8(\norm{\grad \func(x)}_2^2+\sigma^2)+2\norm{\grad \func(x)}_\infty^2+ 2\norm{\grad \func(x)-\grad\func_\nu(x)}_\infty^2\\
        \leq & \nu^2d^2L^2+8(\norm{\grad \func(x)}_2^2+\sigma^2)+2\norm{\grad \func(x)}_\infty^2+ \frac{\nu^2d^2 L^2}{2}\\
        \leq & \frac{3\nu^2d^2L^2}{2}+10\norm{\grad \func(x)}_2^2+8\sigma^2,\\
    \end{split}
\end{equation}
which is the claimed result.
\end{proof}

\subsection{Proof of Lemma~\ref{lemma:property}}
\begin{proof}[Proof of Lemma~\ref{lemma:property}]
We first show that each component of $\scf$ is twice continuously differentiable. Define $\scfi:\mathbb{R}\mapsto\mathbb{R}:x\mapsto (\abs{x}+\frac{1}{d})\ln(d\abs{x}+1)-\abs{x}$. It is straightforward that $\scfi$ is differentiable at $x\neq 0$ with \[\scfi'(x)= \ln (d\abs{x}+1)\sgn(x).\] For any $h\in\mathbb{R}$, we have
\[
\begin{split}
\scfi(0+h)-\scfi(0)=&(\abs{h}+\frac{1}{d})\ln(d\abs{h}+1)-\abs{h}\\
\leq &(\abs{h}+\frac{1}{d})d\abs{h}-\abs{h}\\
=& dh^2,
\end{split}
\]
where the first inequality uses the fact $\ln x\leq x-1$.
Furthermore, we have 
\[
\begin{split}
\scfi(0+h)-\scfi(0)=&(\abs{h}+\frac{1}{d})\ln(d\abs{h}+1)-\abs{h}\\
\geq & (\abs{h}+\frac{1}{d})(\frac{\abs{h}}{\abs{h}+\frac{1}{d}})-\abs{h}\\
\geq & 0,
\end{split}
\]
where the first inequality uses the farc $\ln x\geq 1-\frac{1}{x}$.
Thus, we have 
\[
0\leq \frac{\scfi(0+h)-\scfi(0)}{h}\leq dh
\]
for $h>0$ and 
\[
 dh\leq \frac{\scfi(0+h)-\scfi(0)}{h}\leq 0
\]
for $h<0$, from which it follows $\lim_{h\to 0} \frac{\scfi(0+h)-\scfi(0)}{h}=0$.
Similarly, we have for $x\neq 0$ \[
\scfi''(x)=\frac{1}{\abs{x}+\frac{1}{d}}.
\]
Let $h\neq 0$, then we have
\[
\frac{\scfi'(0+h)-\scfi'(0)}{h}=\frac{ \ln(d\abs{h}+1)\sgn(h)}{h}=\frac{ \ln(d\abs{h}+1)}{\abs{h}}.
\]
From the inequalities of the logarithm, it follows
\[
\frac{1}{\abs{h}+\frac{1}{d}}\leq\frac{\scfi'(0+h)-\scfi'(0)}{h}\leq d.
\]
Thus, we obtain $\scfi''(0)=d$.
Since $\scfi$ is twice continuously differentiable with $\scfi''(x)>0$ for all $x\in\mathbb{R}$, $\scf$ is strictly convex, and we have, for all $x,y\in \Rd$, there is a $c\in [0,1]$ such that 
\begin{equation}
\label{lemma:property:eq1}    
\begin{split}
\scf(y)-\scf(x)=& \grad \scf(x)(y-x)+\sum_{i=1}^d\frac{1}{\abs{cx_i+(1-c)y_i}+\frac{1}{d}}(x_i-y_i)^2.\\
\end{split}
\end{equation}
For all $v\in \Rd$, we have
\begin{equation}
\label{lemma:property:eq2}    
\begin{split}
    &\sum_{i=1}^d \frac{v_i^2}{\abs{cx_i+(1-c)y_i}+\frac{1}{d}}\\
    =& \sum_{i=1}^d \frac{v_i^2}{\abs{cx_i+(1-c)y_i}+\frac{1}{d}}\frac{\sum_{i=1}^d(\abs{cx_i+(1-c)y_i}+\frac{1}{d})}{\sum_{i=1}^d(\abs{cx_i+(1-c)y_i}+\frac{1}{d})}\\
    \geq &\frac{1}{\sum_{i=1}^d(\abs{cx_i+(1-c)y_i}+\frac{1}{d})}(\sum_{i=1}^d \abs{v_i})^2\\
    \geq &\frac{1}{c\norm{x}_1+(1-c)\norm{y}_1+1}(\sum_{i=1}^d \abs{v_i})^2\\
    = &\frac{1}{\max\{\norm{x}_1,\norm{y}_1\}+1}\norm{v}_1^2,
\end{split}
\end{equation}
where the first inequality follows from the Cauchy-Schwarz inequality. Combining \eqref{lemma:property:eq1} and \eqref{lemma:property:eq2}, we obtain the claimed result.
\end{proof}

\subsection{Proof of Theorem~\ref{thm:zo_ada_md}}
\begin{proof}[Proof of Theorem~\ref{thm:zo_ada_md}]
First, assume w.l.o.g. $d\geq e$. Then, for $p=2\ln d$, the squared $p$ norm is $2p-2$ strongly smooth \cite{orabona2015generalized}. Define \[
g_{t,i}=\frac{1}{\nu}(\func(x_t+\nu u_{t,j};\xi_{t,j})-\func(x_t;\xi_{t,j}))u_{t,j}.
\] 
Clearly, $g_{t,1},\dots,g_{t,m}$ are unbiased estimation of $\grad \func_\nu(x_t)$. It follow from lemma~\ref{lemma:gvr} and lemma~\ref{lemma:variance} that
\[
\begin{split}
\ex{\norm{d_t-\grad \func_\nu(x_t)}_\infty^2}\leq&\ex{\norm{d_t-\grad \func_\nu(x_t)}_p^2}\\
\leq&\frac{2\ln d-1}{m}\max_i\ex{\norm{g_{t,i}-\grad \func_\nu(x_t)}_p^2}\\
\leq &\frac{e(2\ln d-1)}{m}\max_i\ex{\norm{g_{t,i}-\grad \func_\nu(x_t)}_\infty^2}\\
\leq &\frac{e(2\ln d-1)}{m}(\frac{3\nu^2d^2L^2}{2}+10\norm{\grad \func(x_t)}_2^2+8\sigma^2).\\
\end{split}
\]
Using lemma~\ref{lemma:variance} and the distribution of $u$, we obtain
\[
\norm{\grad\func_\nu(x_t)-\grad \func(x_t)}_\infty^2\leq\frac{\nu^2 d^2L^2}{4}.
\]
For $m\geq 2e(2\ln d-1)$, we have
\begin{equation}
\label{lemma:main:eq2}
\begin{split}
\ex{\norm{d_t-\grad \func(x_t)}_\infty^2}\leq&2\ex{\norm{d_t-\grad \func_\nu(x_t)}_\infty^2}+2\ex{\norm{\grad \func_\nu(x_t)-\grad \func(x_t)}_\infty^2}\\
\leq&\frac{e(2\ln d-1)}{m}(20\norm{\grad \func(x_t)}_2^2+16\sigma^2)+2\nu^2 d^2L^2\\
\leq&\frac{2e(2\ln d-1)}{m}(10G^2+8\sigma^2)+\frac{2L^2}{T}\\
\leq&\sqrt{\frac{2e(2\ln d-1)}{mT}}(10G^2+8\sigma^2+2L^2).\\
\end{split}
\end{equation}
where the last inequality follows from $m=2Te(2\ln d-1)$.

Next, we analyse constant stepsizes. Note that the potential function defined in \eqref{eq:reg} is $\frac{1}{D+1}$ strongly convex w.r.t. to $\norm{\cdot}_1$.
Our algorithm can be considered as a mirror descent with a distance-generating function given by $(D+1)\scf$, stepsizes $\frac{\eta_t}{D+1}=2L$. Applying proposition~\ref{lemma:omd} with stepsizes $2L$, we have
\begin{equation}
\label{lemma:main:eq3}
\begin{split}
\ex{\norm{\Grad(x_\tau,\grad \func(x_\tau),\eta_\tau)}^2_1}\leq&\frac{6}{T}\sum_{t=1}^T\ex{\sigma_t^2}+\frac{8L}{T}(F(x_1)-F^*)\\
\leq&\sqrt{\frac{2e(2\ln d-1)}{mT}}(6V+8LR),\\
\end{split}
\end{equation}
where we define $V=10G^2+8\sigma^2+2L^2$.
To analyze the adaptive stepsizes, lemma~\ref{lemma:adamd} can be applied with distance generating function $(D+1)\scf$, stepsizes $\frac{\alpha_t}{D+1}$ and
\[
\lambda_t=\frac{1}{\max\{\norm{x_t}_1,\norm{x_{t+1}}_1\}+1}.
\]
It holds clearly $0< \eta=\frac{1}{D+1}\leq \lambda_t\leq 1=\lambda$. Then we obtain
\begin{equation}
\begin{split}
&\ex{\frac{1}{T}\sum_{t=1}^T\norm{\Grad(x_t,\grad \func(x_t),\eta_{t})}_1^2} \leq13V\sqrt{\frac{2e(2\ln d-1)}{mT}}+\frac{C}{T}.
\end{split}
\end{equation}
where we define $C=33R^2+192L^2D(D+1)$.
\end{proof}

\subsection{Proof of Theorem~\ref{thm:zo_ada_md_p}}
\begin{proof}[Proof of Theorem~\ref{thm:zo_ada_md_p}]
Using the same argument in the proof of theorem~\ref{thm:zo_ada_md}, we have
\begin{equation}
\label{lemma:ada_md_p:eq2}
\ex{\norm{d_t-\grad \func(x_t)}_\infty^2}\leq\sqrt{\frac{2e(2\ln d-1)}{mT}}(10G^2+8\sigma^2+2L^2).
\end{equation}
Next, theorem~\ref{lemma:adamd_p} can be applied with distance generating function $(D+1)\scf$, stepsizes $\frac{\alpha_t}{D+1}$ and
\[
\lambda_t=\frac{1}{\max\{\norm{x_t}_1,\norm{v_{t}}_1\}+1}.
\]
Setting $0< \eta=\frac{1}{D+1}\leq \lambda_t\leq 1=\lambda$, we obtain
\begin{equation}
\begin{split}
&\ex{\norm{\Grad(x_\tau,\grad \func(x_\tau),\eta_{\tau})}_1^2} \leq 10V\sqrt{\frac{2e(2\ln d-1)}{mT}}+\frac{C}{T},
\end{split}
\end{equation}
where we define $C=8R^2+8R+16L^2(D+1)^2$ and $V=10G^2+8\sigma^2+2L^2$.
\end{proof}

\subsection{Proof of Lemma~\ref{lemma:storm}}
\begin{proof}[Proof of Lemma~\ref{lemma:storm}]
First, define $z_t=\gamma_t(g_t-\mu_t)+(1-\gamma_t)(g_t-m_t-\mu_t+\mu_{t-1})$. We can upper bound $\ex{\dualnorm{\epsilon_t}^2}$ as follows
\begin{equation}
\label{lemma:storm:eq4}
    \begin{split}
       \ex{\dualnorm{\epsilon_t}^2}=&\ex{\norm{g_t+(1-\gamma_t)(d_{t-1}-m_t)-\mu_t}^2}\\
        =&\ex{\dualnorm{z_t+(1-\gamma_t)\epsilon_{t-1}}^2}\\
        \leq &\ex{\dualnorm{(1-\gamma_t)\epsilon_{t-1}}^2+\inner{u_{(1-\gamma_t)\epsilon_{t-1}}}{z_t}+\frac{M}{2}\dualnorm{z_t}^2}\\
        \leq &\ex{\dualnorm{(1-\gamma_t)\epsilon_{t-1}}^2}+\ex{\inner{u_{(1-\gamma_t)\epsilon_{t-1}}}{z_t}}+\frac{M}{2}\ex{\dualnorm{z_t}^2},
    \end{split}
\end{equation}
where $u_{(1-\gamma_t)\epsilon_{t-1}}\in \partial \dualnorm{\cdot}^2((1-\gamma_t)\epsilon_{t-1})$. From the tower rule, we have
\[
\begin{split}
\ex{\inner{u_{(1-\gamma_t)\epsilon_{t-1}}}{z_t}}=&\ex{\ex{\inner{u_{(1-\gamma_t)\epsilon_{t-1}}}{z_t}|d_{t-1}}}\\
=&\ex{\inner{u_{(1-\gamma_t)\epsilon_{t-1}}}{\ex{z_t|d_{t-1}}}}\\
=&0.
\end{split}
\]
Thus, we can further rewrite \eqref{lemma:storm:eq4} as
\begin{equation}
\label{lemma:storm:eq5}
    \begin{split}
       \ex{\dualnorm{\epsilon_t}^2}\leq &\ex{\dualnorm{(1-\gamma_t)\epsilon_{t-1}}^2}+\frac{M}{2}\ex{\dualnorm{z_t}^2}\\
       \leq &\ex{\dualnorm{(1-\gamma_t)\epsilon_{t-1}}^2}+M\ex{\dualnorm{\gamma_t(g_t-\mu_t)}^2}\\
       &+M\ex{\dualnorm{(1-\gamma_t)(g_t-m_t-\mu_t+\mu_{t-1})}^2}\\
       \leq &\ex{\dualnorm{(1-\gamma_t)\epsilon_{t-1}}^2}+M\ex{\dualnorm{\gamma_t(g_t-\mu_t)}^2}\\
       &+M\ex{\dualnorm{(1-\gamma_t)(g_t-m_t-\ex{g_t|d_{t-1}}+\ex{m_t|d_{t-1}})}^2}\\
       \leq &\ex{\dualnorm{(1-\gamma_t)\epsilon_{t-1}}^2}+M\ex{\dualnorm{\gamma_t(g_t-\mu_t)}^2}\\
       &+2M\ex{\dualnorm{(1-\gamma_t)(g_t-m_t)}^2}\\
    \end{split}
\end{equation}
where the last inequality follows from the Jensen's inequality.
Dividing both sides of \eqref{lemma:storm:eq5} by $2\gamma_t-\gamma_t^2$ and rearranging, we obtain
\begin{equation}
\label{lemma:storm:eq6}
    \begin{split}
       \ex{\dualnorm{\epsilon_{t-1}}^2}\leq &(2\gamma_t-\gamma_t^2)^{-1}\ex{\dualnorm{\epsilon_{t-1}}^2-\dualnorm{\epsilon_{t}}^2}+\frac{M\gamma_t}{2-\gamma_t}\ex{\dualnorm{g_t-\mu_t}^2}\\
       &+\frac{2M(1-\gamma_t)^2}{2\gamma_t-\gamma_t^2}\ex{\dualnorm{g_t-m_t}^2}\\
       &+\frac{2M(1-\gamma_t)^2}{2\gamma_t-\gamma_t^2}\ex{\dualnorm{\mu_t-\mu_{t-1}}^2}\\
       =&\frac{(\tau_t+1)^2}{4\tau_t}\ex{\dualnorm{\epsilon_{t-1}}^2-\dualnorm{\epsilon_{t}}^2}+\frac{M}{\tau_t}\ex{\dualnorm{g_t-\mu_t}^2}\\
       &+\frac{M(\tau_t-1)^2}{2\tau_t}\ex{\dualnorm{g_t-m_t}^2},\\
    \end{split}
\end{equation}
where we set $\gamma_t= \frac{2}{1+\tau_t}$. W.l.o.g. we assume $\epsilon_{0}=0$ and $\frac{(\tau_T+1)^2}{4\tau_T}\geq 1$.  
Summing up \eqref{lemma:storm:eq6} from $1$ to $T$, we obtain
    \begin{equation}
    \label{lemma:storm:eq7}
    \begin{split}
        \ex{\sum_{t=1}^T \dualnorm{\epsilon_{t-1}}^2}\leq&\ex{\sum_{t=1}^T\frac{(\tau_t+1)^2}{4\tau_t}(\dualnorm{\epsilon_{t-1}}^2-\dualnorm{\epsilon_{t}}^2)}\\
       &+\sum_{t=1}^T\frac{M}{\tau_t}\ex{\dualnorm{g_t-\mu_t}^2}\\
       &+\sum_{t=1}^T\frac{M(\tau_t-1)^2}{2\tau_t}\ex{\dualnorm{g_t-m_t}^2}.\\
        \end{split}
    \end{equation}
    The first team of \eqref{lemma:storm:eq7} can be rewritten into
    \begin{equation}
    \label{lemma:storm:eq8}
    \begin{split}
    &\ex{\sum_{t=1}^T\frac{(\tau_t+1)^2}{4\tau_t}(\dualnorm{\epsilon_{t-1}}^2-\dualnorm{\epsilon_{t}}^2)}\\
    =&\frac{(\tau_1+1)^2}{4\tau_1}\ex{\dualnorm{\epsilon_{0}}}-\frac{(\tau_T+1)^2}{4\tau_T}\ex{\dualnorm{\epsilon_{T}}}\\
    &+\sum_{t=2}^{T}(\frac{(\tau_t+1)^2}{4\tau_t}-\frac{(\tau_{t-1}+1)^2}{4\tau_{t-1}})\ex{\dualnorm{\epsilon_{t-1}}}\\
   \leq&-\ex{\dualnorm{\epsilon_{T}}}+\sum_{t=2}^{T}(\frac{(\tau_t+1)^2}{4\tau_t}-\frac{(\tau_{t-1}+1)^2}{4\tau_{t-1}})\ex{\dualnorm{\epsilon_{t-1}}},\\
    \end{split}
    \end{equation}
    where the second inequality uses the assumption $\dualnorm{\epsilon_{0}}=0$ and $\frac{(\tau_T+1)^2}{4\tau_T}\geq 1$. For $\tau_t=(1+\gamma t^{\frac{2}{3}})$, we clearly have $\frac{(\tau_t+1)^2}{4\tau_t}-\frac{(\tau_{t-1}+1)^2}{4\tau_{t-1}}\leq  \frac{1}{4}(\tau_t-\tau_{t-1})$. Using the concavity of $x\mapsto \frac{1}{4}(1+x)^\frac{2}{3}$ and the fact $\gamma\leq 1$, we have
    \begin{equation}
    \label{lemma:storm:eq9}
    \begin{split}
    \frac{1}{4}(1+\gamma (t+1))^{\frac{2}{3}}-\frac{1}{4}(1+\gamma t)^{\frac{2}{3}}\leq \frac{1}{6}\frac{\gamma}{(1+t\gamma)^{\frac{1}{3}}}\leq \frac{1}{6}.
    \end{split}
    \end{equation}
    for all $t$. Combining \eqref{lemma:storm:eq7},\eqref{lemma:storm:eq8} and \eqref{lemma:storm:eq9}, we obtain, 
    \[
    \begin{split}
        \ex{\sum_{t=1}^T \dualnorm{\epsilon_{t}}^2}\leq&\sum_{t=1}^T\frac{6M}{5\tau_t}\ex{\dualnorm{g_t-\mu_t}^2}+\sum_{t=1}^T\frac{3M(\tau_t-1)^2}{5\tau_t}\ex{\dualnorm{g_t-m_t}^2},\\
        \end{split}
    \]
which is the claimed result.
\end{proof}

\subsection{Proof of Theorem~\ref{lemma:ada_storm}}
\begin{proof}[Proof of Theorem~\ref{lemma:ada_storm}]
First of all, $\sequ{\alpha}$ is an increasing sequence. Using a similar argument as the proof of \ref{lemma:adamd}, we have
\begin{equation}
\begin{split}    
\label{lemma:md_m:eq1}
0\leq&R\eta\alpha_{T+1}+R\eta\\
&+\sum_{t=1}^T(\frac{L\eta\alpha_{t}^2}{2\alpha_{t+1}}-\frac{\eta_{t}^2}{4})\norm{v_{t}-x_t}^2\\
&-\sum_{t=1}^T\frac{\eta_{t}^2}{2}\norm{v_t-x_t}^2+\sum_{t=1}^T\sigma_t^2.\\
\end{split}
\end{equation}
Using the definition of the step size, we have
\begin{equation}
\label{lemma:md_m:eq2}
    \begin{split}
    R\eta\alpha_{T+1}\leq&R^2\lambda^2\beta_{T+1}+\frac{\eta^2}{4\lambda^2}\sum_{t=1}^T\beta_t\lambda_t^2\alpha_t^2\norm{x_t-v_t}^2+\frac{\eta^2}{4\lambda^2},\\
    \leq&R^2\lambda^2\beta_{T+1}+\frac{\eta^2}{4\lambda^2}\sum_{t=1}^T\beta_t\lambda_t^2\alpha_t^2\norm{x_t-v_t}^2+\frac{\eta^2}{4\lambda^2}\\
    \leq&R^2\lambda^2T^\frac{1}{3}+\frac{1}{4}\sum_{t=1}^T\norm{\Grad(x_t,d_t,\eta_t)}^2+\frac{\eta^2}{4\lambda^2},\\
    \end{split}
\end{equation}
Since $\beta_t \geq 1$, we use the same argument as the proof of lemma \ref{lemma:adamd} and obtain
\begin{equation}
\label{lemma:md_m:eq3}
    \begin{split}
        \sum_{t=1}^T(\frac{L\eta\alpha_{t}^2}{2\alpha_{t+1}}-\frac{\eta_{t}^2}{4})\norm{v_{t}-x_t}^2\leq\frac{2L^2}{\eta^2}.\\
    \end{split}
\end{equation}
Combining \eqref{lemma:md_m:eq1}, \eqref{lemma:md_m:eq2} and \eqref{lemma:md_m:eq3}, we obtain
\begin{equation}
\label{lemma:md_m:eq4}
\begin{split}    
\ex{\sum_{t=1}^T\norm{\Grad(x_t,\grad \func(x_t),\eta_t)}^2}&\leq \ex{2\sum_{t=1}^T\norm{\Grad(x_t,d_t,\eta_t)}^2+2\sum_{t=1}^2\sigma_t^2}\\
\leq&16R^2\lambda^2T^\frac{1}{3}+\frac{4\eta^2}{\lambda^2}+16R\eta+\frac{32L^2}{\eta^2}\\
&+18\ex{\sum_{t=1}^T\sigma_t^2}-\ex{2\sum_{t=1}^T\norm{\Grad(x_t,d_t,\eta_t)}^2}.
\end{split}
\end{equation}
Next, combining lemma \eqref{lemma:storm}, the assumptions on $g_t$, $m_t$, and $\mu_t$, we obtain, 
\begin{equation}
    \label{lemma:md_m:eq5}
    \begin{split}
        \ex{\sum_{t=1}^T \dualnorm{\sigma_{t}}^2}
        \leq&2\ex{\sum_{t=1}^T \dualnorm{\epsilon_{t}}^2}+2\ex{\sum_{t=1}^T \dualnorm{\grad \func(x_t)-\mu_t}^2}\\
        \leq&\frac{12M\tilde\sigma^2}{5}\sum_{t=1}^T\frac{1}{\tau_t}+\sum_{t=1}^T\frac{6C_1M(\tau_t-1)^2}{5\tau_t}\ex{\norm{x_t-x_{t-1}}^2}\\
        &+\frac{6MC_2}{5}\sum_{t=1}^T\frac{(\tau_t-1)^2}{\tau_t}+2T\bar\sigma.
        \end{split}
\end{equation}
For $\tau_t=(1+\gamma t)^\frac{2}{3}$, we apply \cite[lemma 3]{NEURIPS2021_ac10ff19} and obatin
\[
\frac{12M\tilde\sigma^2}{5}\sum_{t=1}^T\frac{1}{\tau_t}\leq\frac{12M\tilde\sigma^2}{5\gamma}\sum_{t=1}^T\frac{\gamma}{(\gamma t)^\frac{2}{3}}\leq \frac{36M\tilde\sigma^2T^\frac{1}{3}}{5\gamma^\frac{2}{3}}.
\]
To bound the second term,  we define the index
\[
\begin{split}
t_0=\max(\{1\leq t\leq T|\frac{6C_1M(\tau_t-1)^2}{5\tau_t}>\frac{1}{9}\eta_{t}^2\}\cup \{0\}). 
\end{split}
\]
Then we have
\begin{equation}
    \label{lemma:md_m:eq6}
    \begin{split}
        &\sum_{t=1}^{t_0}\frac{6C_1M(\tau_t-1)^2}{5\tau_t}\norm{x_t-x_{t-1}}^2\\
        \leq&\frac{6C_1M(\tau_{t_0}-1)}{5\sqrt{\tau_{t_0}}}\sum_{t=1}^{t_0}\frac{\alpha_t^2\norm{x_t-x_{t-1}}^2}{\sum_{s=1}^{t-1}\alpha_s^2\norm{v_s-x_s}^2+1}\\
        =&\frac{6C_1M(\tau_{t_0}-1)}{5\sqrt{\tau_{t_0}}}\sum_{t=2}^{t_0-1}\frac{\alpha_{t}^2\norm{v_{t}-x_{t}}^2}{\sum_{s=1}^{t}\alpha_s^2\norm{v_s-x_s}^2+1}\\
        \leq &\frac{6C_1M(\tau_{t_0}-1)}{5\sqrt{\tau_{t_0}}}\ln(\sum_{t=1}^{t_0-1}\alpha_{t}^2\norm{v_{t}-x_{t}}^2+1)\\
        \leq &\frac{6C_1M(\tau_{t_0}-1)}{5\sqrt{\tau_{t_0}}}\ln\alpha_{t_0}^2\\
        = &\frac{6C_1M(\tau_{T}-1)}{5\sqrt{\tau_{T}}}\ln\frac{\eta_{t_0}^2}{\eta^2}\\
        \leq &\frac{6C_1M(\tau_{T}-1)}{5\sqrt{\tau_{T}}}\ln\frac{54C_1M(\tau_T-1)^2}{5\eta^2\tau_T}\\
        \leq &\frac{6C_1MT^{\frac{1}{3}}}{5}\ln\frac{54C_1MT^\frac{2}{3}}{5\eta^2},\\
    \end{split}
\end{equation}
where the first inequality uses the definition of $\alpha_t$, the second inequality uses \cite[lemma 6]{shao_optimistic_2022}, the third inequality uses the fact that $\beta_{t_0}\geq 1$ and the fourth inequality follows from the definition of the index $t_0$.
We also have
\begin{equation}
    \label{lemma:md_m:eq7}
    \begin{split}
        &\sum_{t=t_0+1}^{T}\frac{6C_1M(\tau_t-1)^2}{5\tau_t}\norm{x_t-x_{t-1}}^2\\
        \leq&\sum_{t=t_0+1}^{T}\frac{\eta_t^2}{9}\norm{x_t-x_{t-1}}^2\\
        =&\sum_{t=t_0+1}^{T}\frac{\eta_{t-1}^2}{9}\norm{v_{t-1}-x_{t-1}}^2\\
        \leq&\sum_{t=1}^{T}\frac{\eta_{t}^2}{9}\norm{v_{t}-x_{t}}^2\\
        =&\frac{1}{9}\sum_{t=1}^{T}\norm{\Grad(x_t,d_t,\eta_t)}^2\\
    \end{split}
\end{equation}
Combining \eqref{lemma:md_m:eq4},\eqref{lemma:md_m:eq5}, \eqref{lemma:md_m:eq6} and \eqref{lemma:md_m:eq7}, we obtain
\begin{equation}
\label{lemma:md_m:eq8}
    \begin{split}
&\ex{\sum_{t=1}^T\norm{\Grad(x_t,\grad \func(x_t),\eta_t)}^2}\\
\leq&16R^2\lambda^2T^\frac{1}{3}+\frac{4\eta^2}{\lambda^2}+16R\eta+\frac{32L^2}{\eta^2}\\
&+\frac{648M\tilde\sigma^2T^\frac{1}{3}}{5\gamma^\frac{2}{3}}+\frac{108MC_2T^{\frac{5}{3}}}{5}+36T\bar\sigma\\
&+\frac{108C_1MT^{\frac{1}{3}}}{5}\ln\frac{54C_1MT^\frac{2}{3}}{5\eta^2},\\
\end{split}
\end{equation}
which is the claimed result.
\end{proof}

\subsection{Proof of Lemma~\ref{lemma:bias_diff}}
\begin{proof}[Proof of Lemma~\ref{lemma:bias_diff}]
Using the same argument as the proof of lemma \ref{lemma:bias}, we have
\begin{equation}
\label{lemma:bias_diff:eq1}
\begin{split}
&\mathbb{E}_u[{\dualnorm{g_\nu(x;\xi,u)-g_\nu(y;\xi,u)}^2}]\\
=&\mathbb{E}_u[\frac{\delta^2}{\nu^2}\abs{\func(x+\nu u;\xi)-\func(x;\xi)-\func(y+\nu u;\xi)+\func(y;\xi)}^2\dualnorm{u}^2]\\
\leq&\frac{3\delta^2}{\nu^2}\mathbb{E}_u[(\func(x+\nu u;\xi)-\func(x;\xi)-\inner{\grad \func(x;\xi)}{\nu u})^2\dualnorm{u}^2]\\
&+\frac{3\delta^2}{\nu^2}\mathbb{E}_u[(\func(y+\nu u;\xi)-\func(y;\xi)-\inner{\grad \func(y;\xi)}{\nu u})^2\dualnorm{u}^2]\\
&+\frac{3\delta^2}{\nu^2}\mathbb{E}_u[\inner{\grad \func(x;\xi)-\grad \func(y;\xi)}{\nu u}^2\dualnorm{u}^2]\\
\leq &\frac{3C^4L^2\delta^2\nu^2}{4}\mathbb{E}_u[\dualnorm{u}^6]+3\delta^2\mathbb{E}_u[\inner{\grad \func(x;\xi)-\grad \func(y;\xi)}{u}^2\dualnorm{u}^2],\\
\end{split}
\end{equation}
which is the claimed result.
\end{proof}

\subsection{Proof of Theorem~\ref{thm:zo_exp_storm}}
\begin{proof}[Proof of Theorem~\ref{thm:zo_exp_storm}]
First, applying lemma~\ref{lemma:gvr}, we obtain
\begin{equation}
\label{lemma:zo_exp_storm:eq1}
\begin{split}
    \ex{\sum_{t=1}^T \norm{\epsilon_{t}}_\infty^2}\leq &\ex{\sum_{t=1}^T \norm{\epsilon_{t}}_p^2}\\
    \leq&\sum_{t=1}^T\frac{6(4\ln d-2)}{5\tau_t}\ex{\norm{g_t-\mu_t}_p^2}\\
    &+\sum_{t=1}^T\frac{3(4\ln d-2)(\tau_t-1)^2}{5\tau_t}\ex{\norm{g_t-m_t}_p^2}\\
    \leq&\sum_{t=1}^T\frac{6e(4\ln d-2)}{5\tau_t}\ex{\norm{g_t-\mu_t}_\infty^2}\\
    &+\sum_{t=1}^T\frac{3e(4\ln d-2)(\tau_t-1)^2}{5\tau_t}\ex{\norm{g_t-m_t}_\infty^2}\\
\end{split}
\end{equation}
for $p=2\ln d$. Next applying 
Using the same argument in the proof of theorem~\ref{thm:zo_ada_md}, we have
\begin{equation}
\label{thm:zo_exp_storm:eq1}
\begin{split}
\ex{\norm{g_t-\grad \func_\nu(x_t)}_\infty^2}\leq&\frac{e(2\ln d-1)}{m}(10\norm{\grad \func(x_t)}_2^2+8\sigma^2)\\
\leq&\frac{e(2\ln d-1)}{m}(10G^2+8\sigma^2).\\
\end{split}
\end{equation}
Using lemma~\ref{lemma:bias_diff} with $\nu\leq d^{-1}T^{-\frac{2}{3}}$, we have
\begin{equation}
\label{lemma:zo_exp_storm:eq2}
\begin{split}
    &\mathbb{E}[\norm{g_t-m_t}_\infty^2]\\
=&\mathbb{E}[\norm{\frac{1}{m}\sum_{i=1}^m(g_\nu(x_t;\xi_{t,i},u_{t,i})-g_\nu(x_{t-1};\xi_{t,i},u_{t,i}))}_\infty^2]\\
\leq &\max_i\mathbb{E}[\norm{(g_\nu(x_t;\xi_{t,i},u_{t,i})-g_\nu(x_{t-1};\xi_{t,i},u_{t,i}))}_\infty^2]\\
\leq &\frac{3d^2L^2\nu^2}{4}+3\ex{\norm{\grad \func(x_t;\xi_t)-\grad \func(x_{t-1};\xi_t)}_2^2}\\
\leq &\frac{3L^2T^{-\frac{4}{3}}}{4}+3L^2\ex{\norm{x_t-x_{t-1}}_1^2}.\\
\end{split}
\end{equation}
Applying lemma~\ref{lemma:bias} with $\nu\leq d^{-1}T^{-\frac{1}{3}}$, we obtain $\norm{\grad \func_\nu(x_t)-\grad \func(x_t)}_\infty^2\leq 2L^2T^{-\frac{2}{3}}$.
Finally, we can apply theorem~\ref{thm:zo_exp_storm} with $\tilde\sigma^2\leq \frac{e(2\ln d-1)}{m}(10G^2+8\sigma^2)$, $\bar \sigma\leq 2L^2T^{-\frac{2}{3}}$, $M=e(4\ln d-2)$, $C_1=3L^2$ and $C_2=\frac{3}{4}L^2T^{-\frac{4}{3}}$. Together with the choice of hyperparameters, we obtain
\[
\begin{split}
&\ex{\sum_{t=1}^T\norm{\Grad(x_t,\grad \func(x_t),\eta_t)}_1^2}\\
\leq&16R^2T^\frac{1}{3}+4+16R+32L^2(D+1)^2\\
&+\frac{648e^2(4\ln d-2)^2(5G^2+4\sigma^2)T^\frac{1}{3}}{5m^\frac{1}{3}}+\frac{192e(2\ln d-1)L^2T^{\frac{1}{3}}}{5}\\
&+72L^2T^{\frac{1}{3}}+\frac{648L^2e(2\ln d-1)T^{\frac{1}{3}}}{5}\ln\frac{324(D+1)^2L^2e(2\ln d-1)T^\frac{2}{3}}{5}.\\
\end{split}
\]
We obtain the desired result by uniformly and randomly sampling $\tau$ from $1,\ldots, T$.
\end{proof}

\section{Efficient Implementation for Elastic Net Regularisation}
We consider the following updating rule
\begin{equation}
\label{eq:omd:impl}
\begin{split}
    y_{t+1} &=\grad \dualscf(\grad \scf(x_t)-\frac{g_t}{\eta_t})\\
    x_{t+1} &=\argmin_{x\in\cK}\comp(x)+\eta_t\bd{\scf}{x}{y_{t+1}}.
\end{split}
\end{equation}
It is easy to verify 
\[
(\grad \dualscf (\theta))_i =(\frac{1}{d}\exp(\abs{\theta_i})-\frac{1}{d})\sgn(\theta_i).
\]
Furthermore, \eqref{eq:omd:impl} is equivalent to the mirror descent update \eqref{eq:update_md} due to the relation
\[
\begin{split}
x_{t+1}=&\argmin_{x\in\cK}\comp(x)+\eta_t\bd{\scf}{x}{y_{t+1}}\\
=&\argmin_{x\in\cK}\comp(x)+\eta_t\scf(x)-\inner{\eta_t\grad\scf(y_{t+1})}{x}\\
=&\argmin_{x\in\cK}\comp(x)+\eta_t\scf(x)-\inner{\eta_t\grad \scf(x_t)-g_t}{x}\\
=&\argmin_{x\in\cK}\inner{g_t}{x}+\comp(x)+\eta_t\bd{\scf}{x}{x_t}.\\
\end{split}
\]
Next, We consider the setting of $\cK=\Rd$ and $\comp(x)=\gamma_1 \norm{x}_1+\frac{\gamma_2}{2}\norm{x}^2_2$. 
The minimiser of 
\[
\comp(x)+\eta_t\bd{\scf}{x}{y_{t+1}}
\]
in $\Rd$ can be simply obtained by setting the subgradient to $0$. For $\ln(d\abs{y_{i,t+1}}+1)\leq\frac{\gamma_1}{\eta_{t+1}}$, we set $x_{i,t+1}=0$. Otherwise, the $0$ subgradient implies $\sgn(x_{i,t+1})=\sgn(y_{i,t+1})$ and $\abs{x_{i,t+1}}$ given by the root of
\[
\begin{split}
\ln(d\abs{y_{i,t+1}}+1)=\ln(d\abs{x_{i,t+1}}+1)+\frac{\gamma_1}{\eta_{t}}+\frac{\gamma_2}{\eta_{t}}\abs{x_{i,t+1}}
\end{split}
\]
for $i=1,\ldots, d$. 
For simplicity, we set $a=\frac{1}{d}$, $b=\frac{\gamma_2}{\eta_{t}}$ and $c=\frac{\gamma_1}{\eta_{t}}-\ln(d\abs{y_{i,t+1}}+1)$. It can be verified that $\abs{x_{i,t+1}}$ is given by
\begin{equation}
\abs{x_{i,t+1}}=\frac{1}{b}W_0(ab\exp(ab-c))-a,
\end{equation}
where $W_0$ is the principle branch of the \textit{Lambert function} and can be well approximated \cite{iacono2017new}.
For $\gamma_2=0$, i.e. the $\ell_1$ regularised problem, $\abs{x_{i,t+1}}$ has the closed form solution
\begin{equation}
\abs{x_{i,t+1}}=\frac{1}{d}\exp(\ln(d\abs{y_{i,t+1}}+1)-\frac{\gamma_1}{\eta_{t}})-\frac{1}{d}.
\end{equation}
The implementation is described in Algorithm \ref{alg:reg}.
\begin{algorithm}
	\caption{Solving $\min _{x\in \Rd} \inner{g_t}{x}+\comp(x)+\eta_t\bd{\scf}{x}{x_{t}}$}
    \label{alg:reg}
	\begin{algorithmic}
	\For{$i=1,\ldots,d$}
	    \State $z_{i,t+1}=\ln (d \abs{x_{i,t}}+1)\sgn(x_{i,t})-\frac{g_{i,t}}{\eta_t}$
	    \State $y_{i,t+1}=(\frac{1}{d}\exp(\abs{z_{i,t+1}})-\frac{1}{d})\sgn(z_{i,t+1})$
	    \If{$\ln(d\abs{y_{i,t+1}}+1)\leq\frac{\gamma_1}{\eta_{t}}$}
	    \State $x_{t+1,i}\gets 0$
        \Else
        \State $a\gets\beta$
        \State $b\gets\frac{\gamma_2}{\eta_{t}}$
        \State $c\gets\frac{\gamma_1}{\eta_{t}}-\ln(d\abs{{y}_{t+1,i}}+1)$
        \State $x_{t+1,i}\gets\frac{1}{b}W_0(ab\exp(ab-c))-a$
        \EndIf
    \EndFor
    \State Return $x_{t+1}$
    \end{algorithmic}
\end{algorithm}

\subsection{Impact of the Choice of Stepsizes of PGD and Acc-ZOM}
\begin{figure}
\centering
\begin{subfigure}{.5\textwidth}
  \centering
  \includegraphics[width=\linewidth]{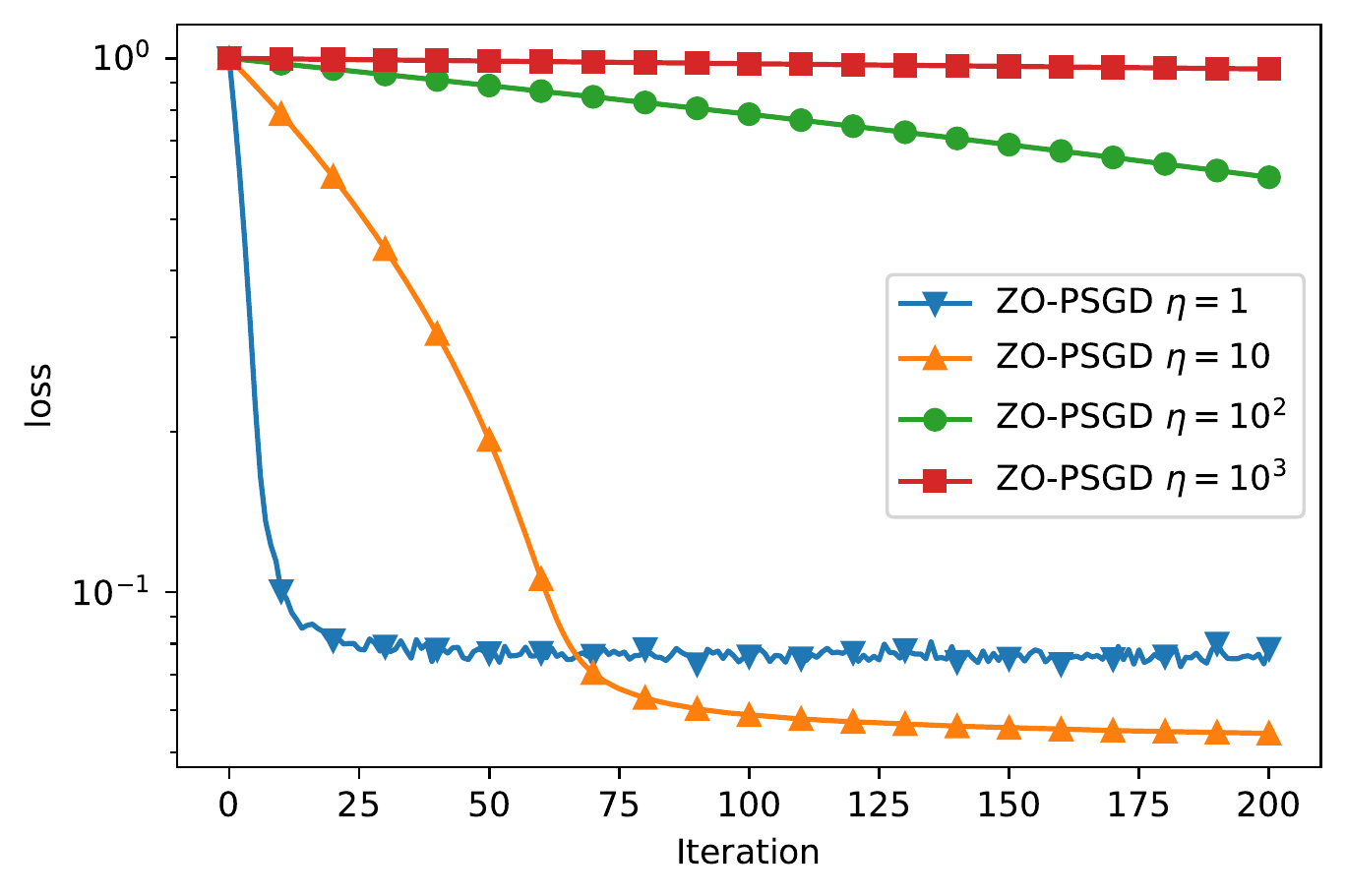}
\caption{Convergence for Generating \textbf{PN}}%
\label{fig:bb-pgd-pn-mnist}
\end{subfigure}%
\begin{subfigure}{.5\textwidth}
  \centering
  \includegraphics[width=\linewidth]{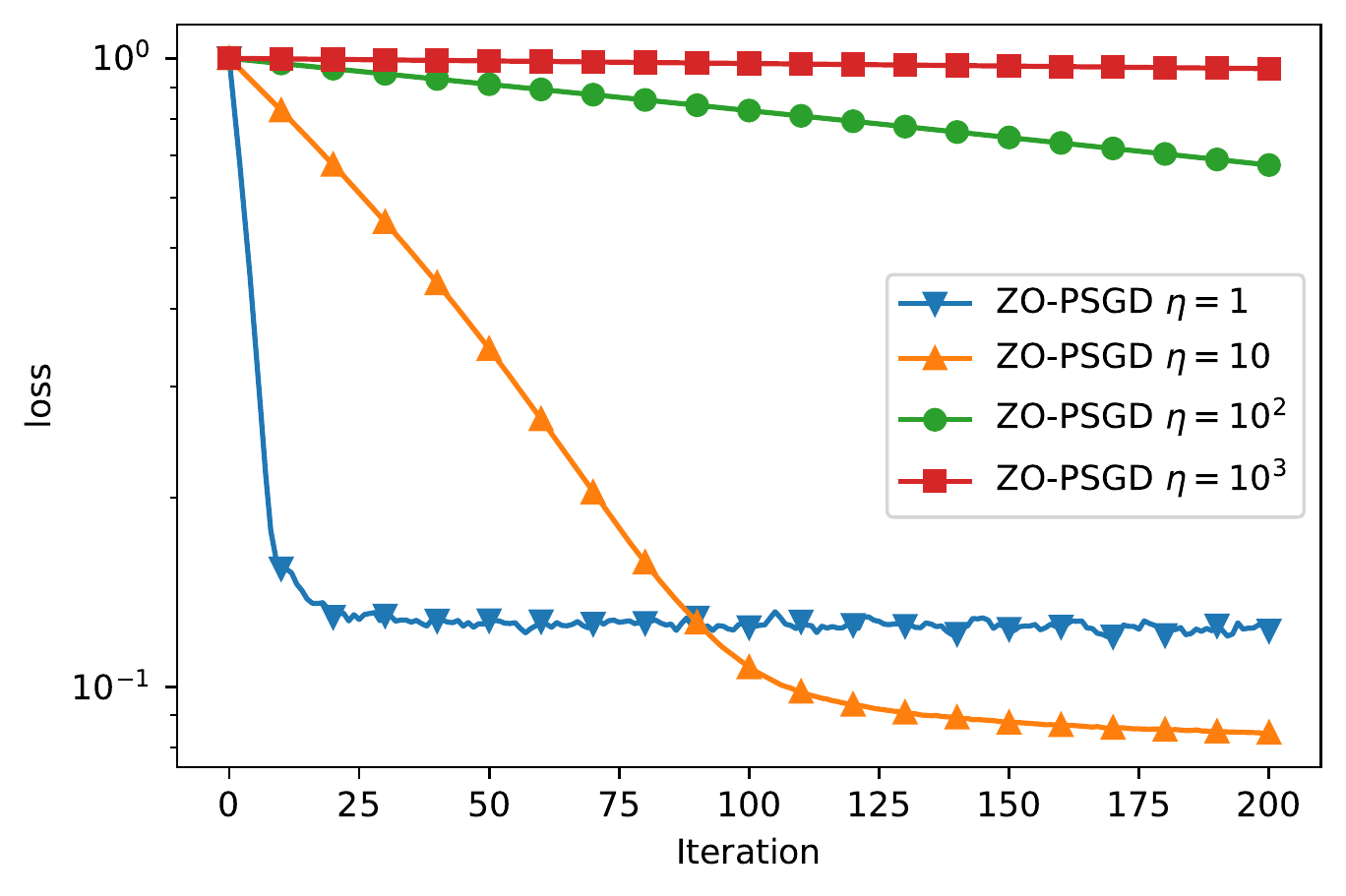}
\caption{Convergence for Generating  \textbf{PP}}%
\label{fig:bb-pgd-pp-mnist}
\end{subfigure}
\caption{Impact of step size on ZO-PSGD on MNIST}
\label{fig:BB-PGD-mnist}
\end{figure}
\begin{figure}
\centering
\begin{subfigure}{.5\textwidth}
  \centering
  \includegraphics[width=\linewidth]{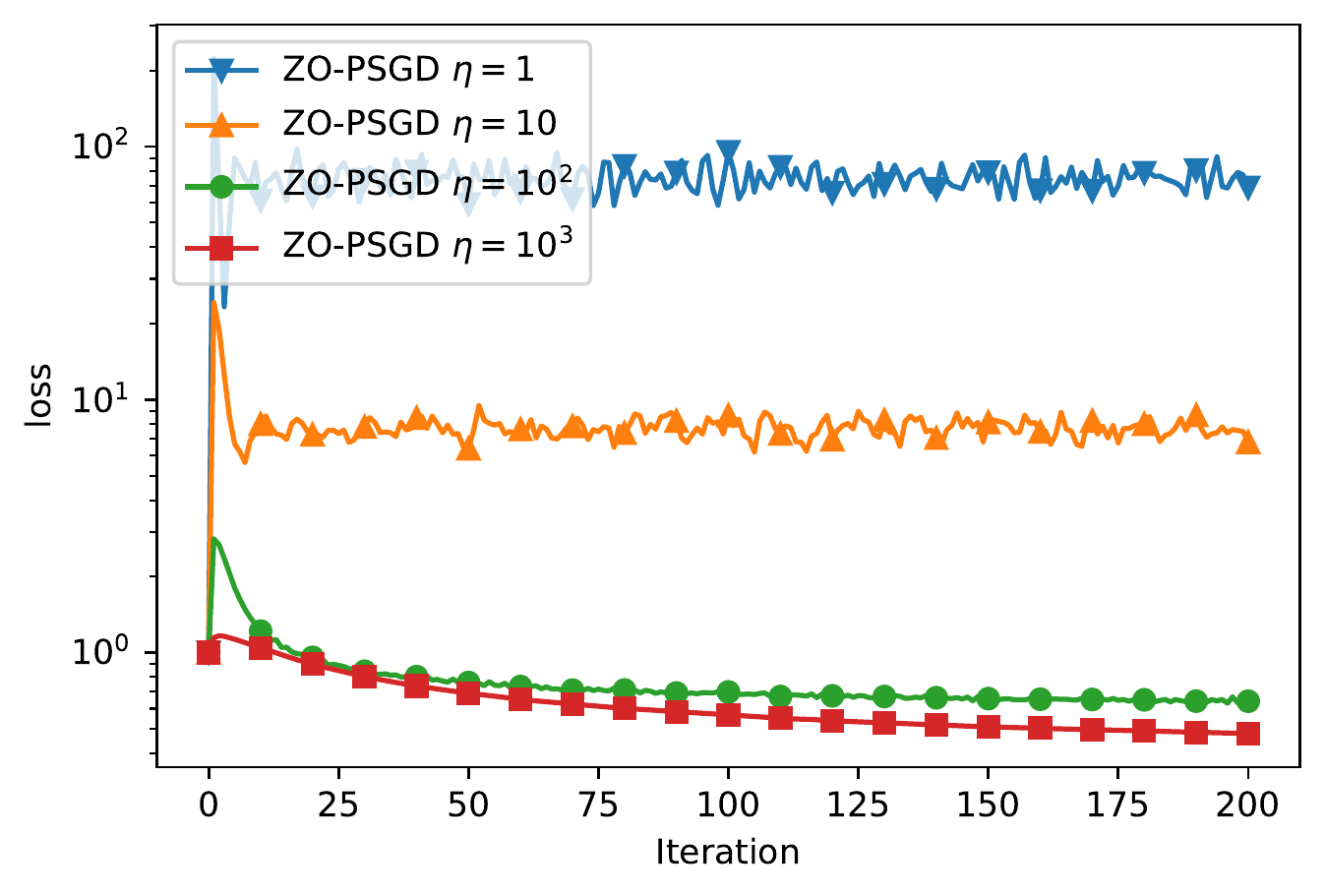}
\caption{Convergence for Generating \textbf{PN}}%
\label{fig:bb-pgd-pn}
\end{subfigure}%
\begin{subfigure}{.5\textwidth}
  \centering
  \includegraphics[width=\linewidth]{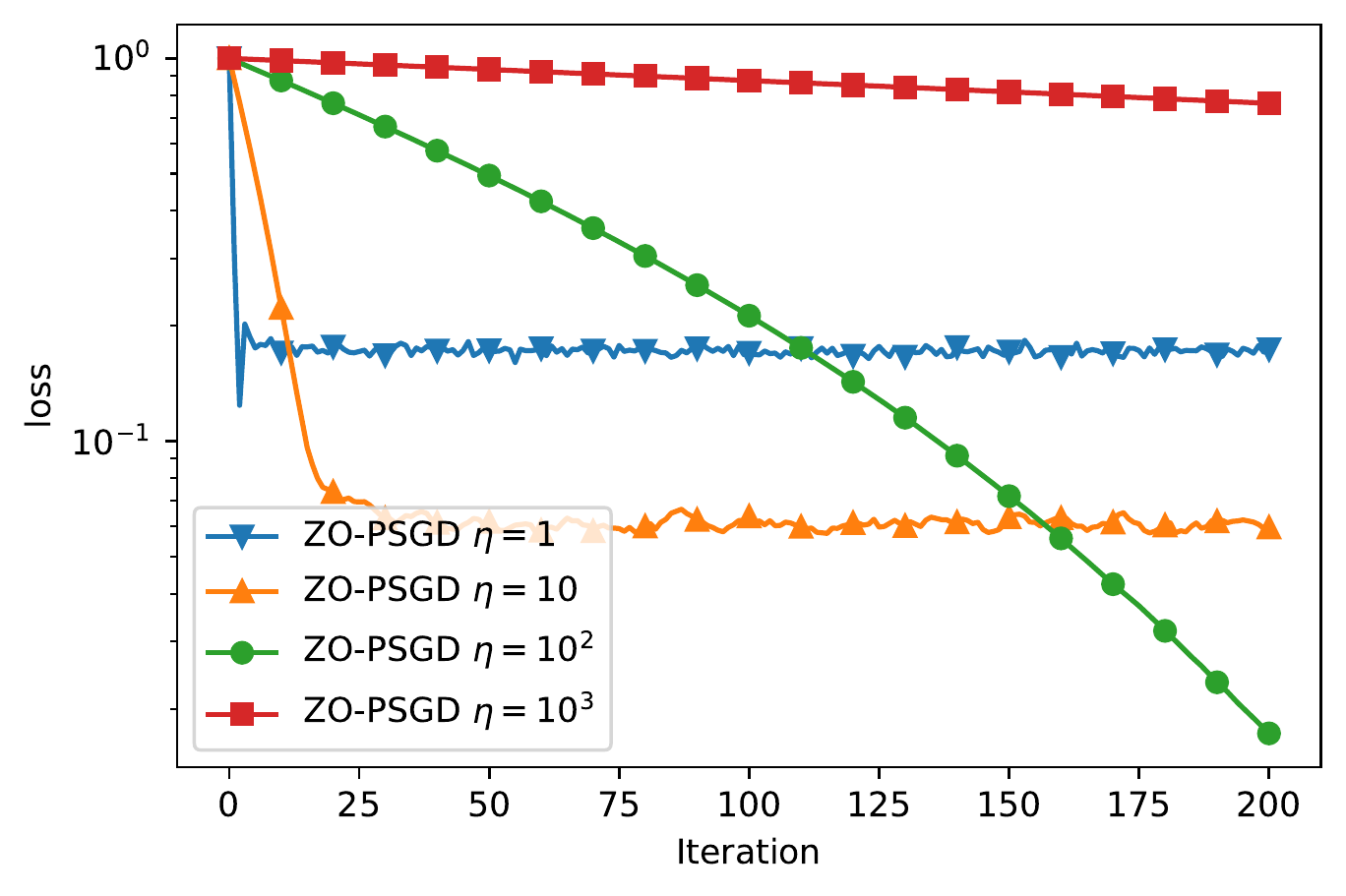}
\caption{Convergence for Generating  \textbf{PP}}%
\label{fig:bb-pgd-pp}
\end{subfigure}
\caption{Impact of step size on ZO-PSGD on CIFAR-$10$}
\label{fig:BB-PGD}
\end{figure}
\begin{figure}
\centering
\begin{subfigure}{.5\textwidth}
  \centering
  \includegraphics[width=\linewidth]{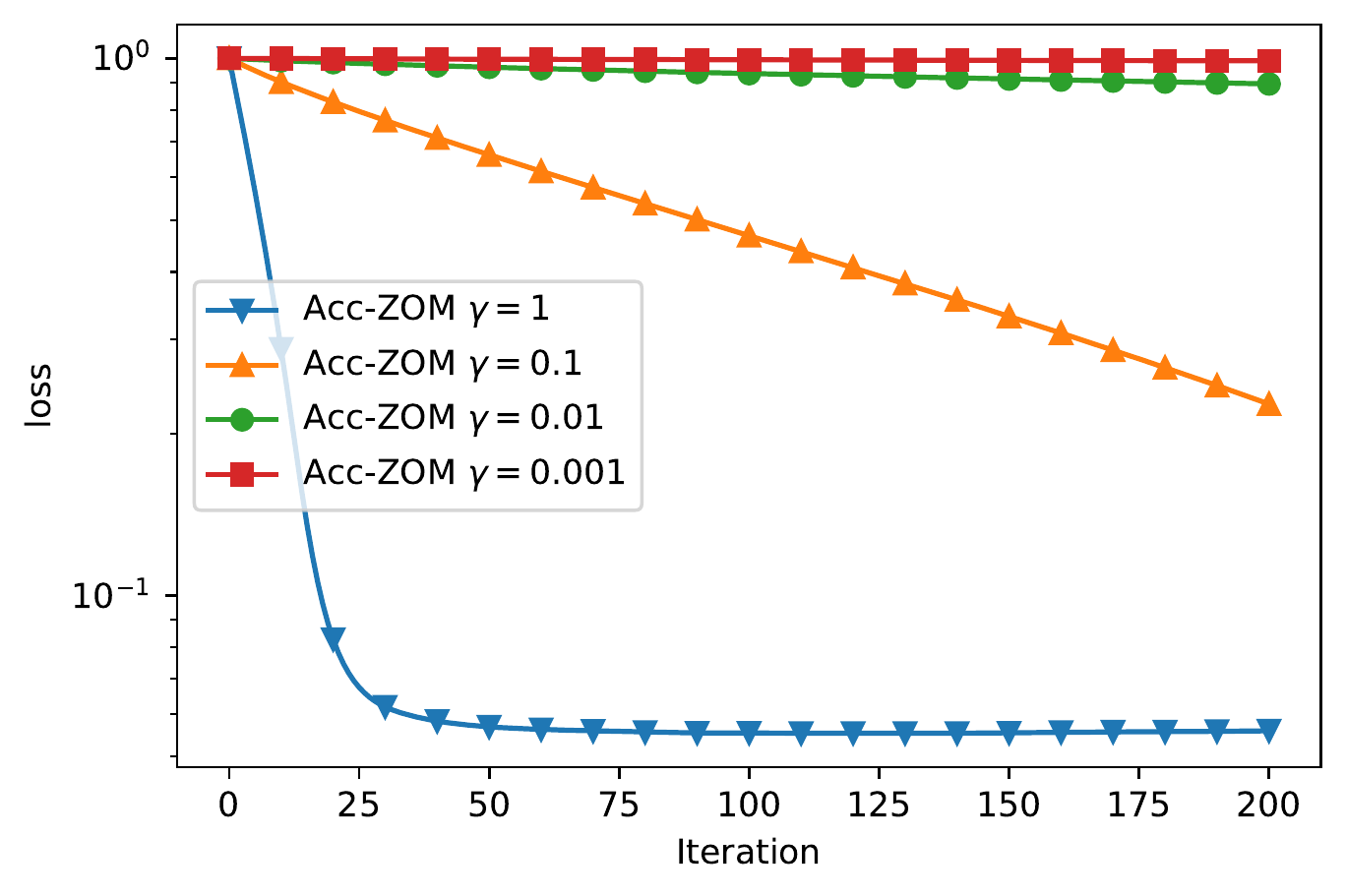}
\caption{Convergence for Generating \textbf{PN}}%
\label{fig:bb-zom-pn-mnist}
\end{subfigure}%
\begin{subfigure}{.5\textwidth}
  \centering
  \includegraphics[width=\linewidth]{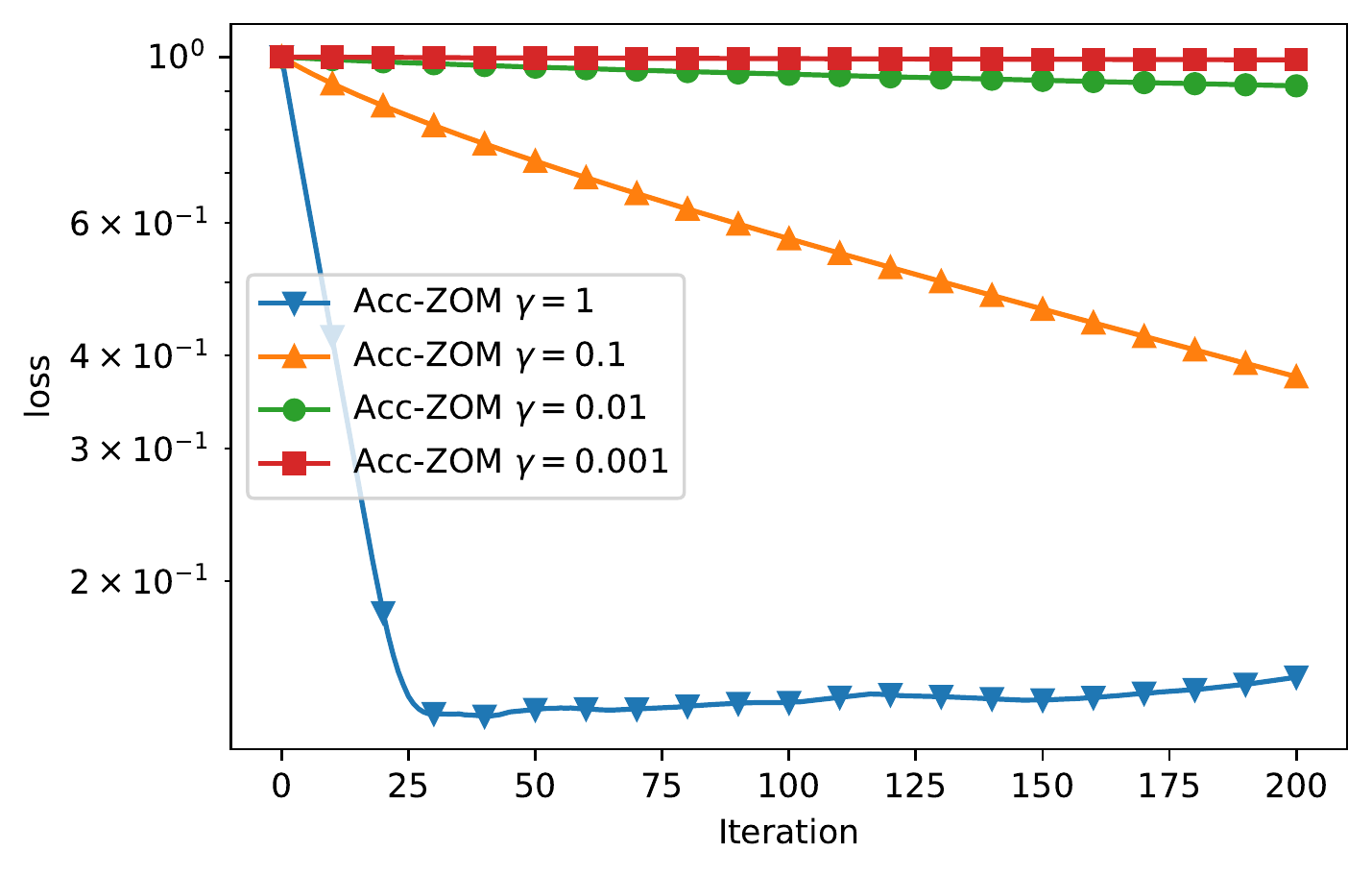}
\caption{Convergence for Generating  \textbf{PP}}%
\label{fig:bb-zom-pp-mnist}
\end{subfigure}
\caption{Impact of step size on Acc-ZOM on MNIST}
\label{fig:BB-ZOM-mnist}
\end{figure}
\begin{figure}
\centering
\begin{subfigure}{.5\textwidth}
  \centering
  \includegraphics[width=\linewidth]{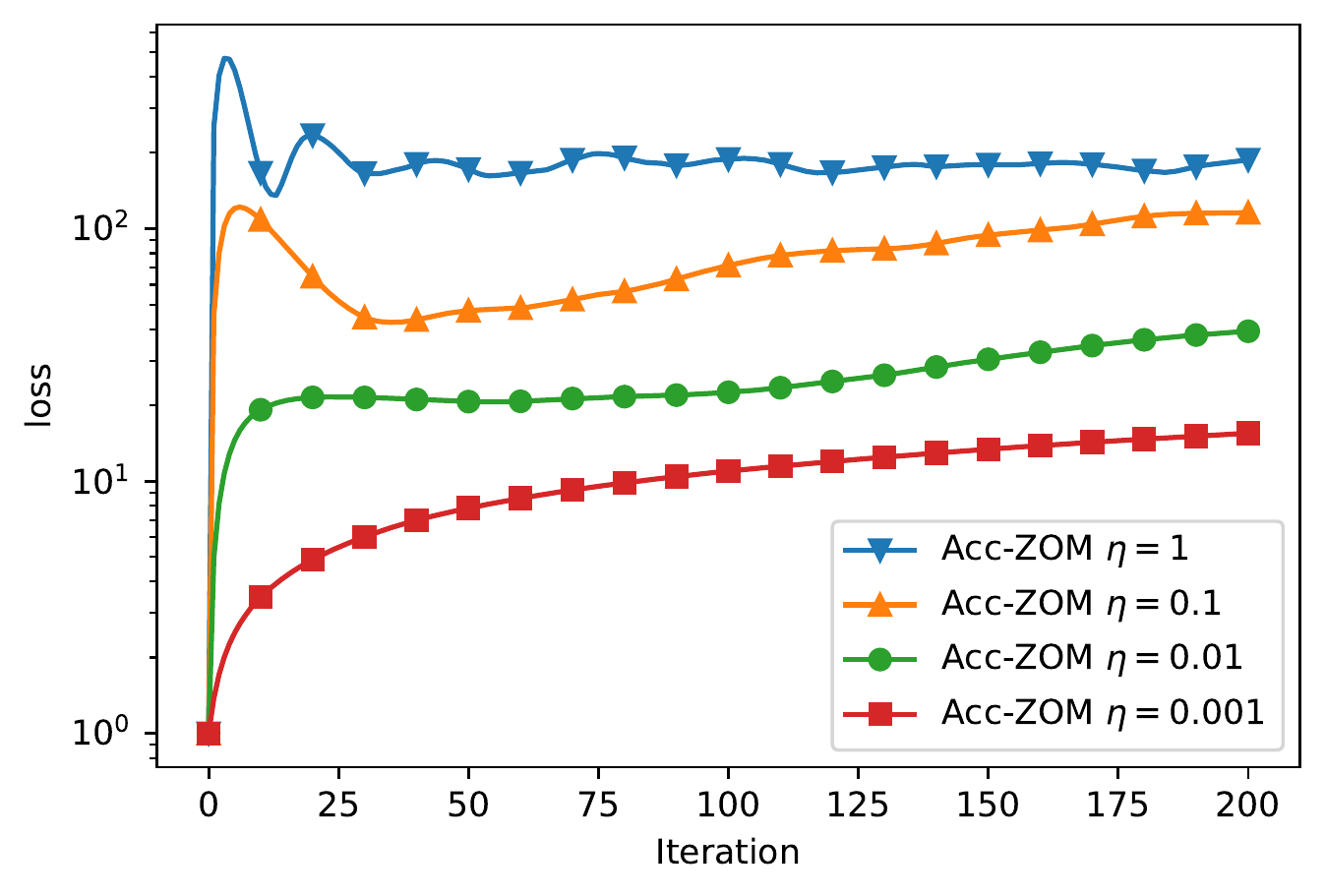}
\caption{Convergence for Generating \textbf{PN}}%
\label{fig:bb-zom-pn}
\end{subfigure}%
\begin{subfigure}{.5\textwidth}
  \centering
  \includegraphics[width=\linewidth]{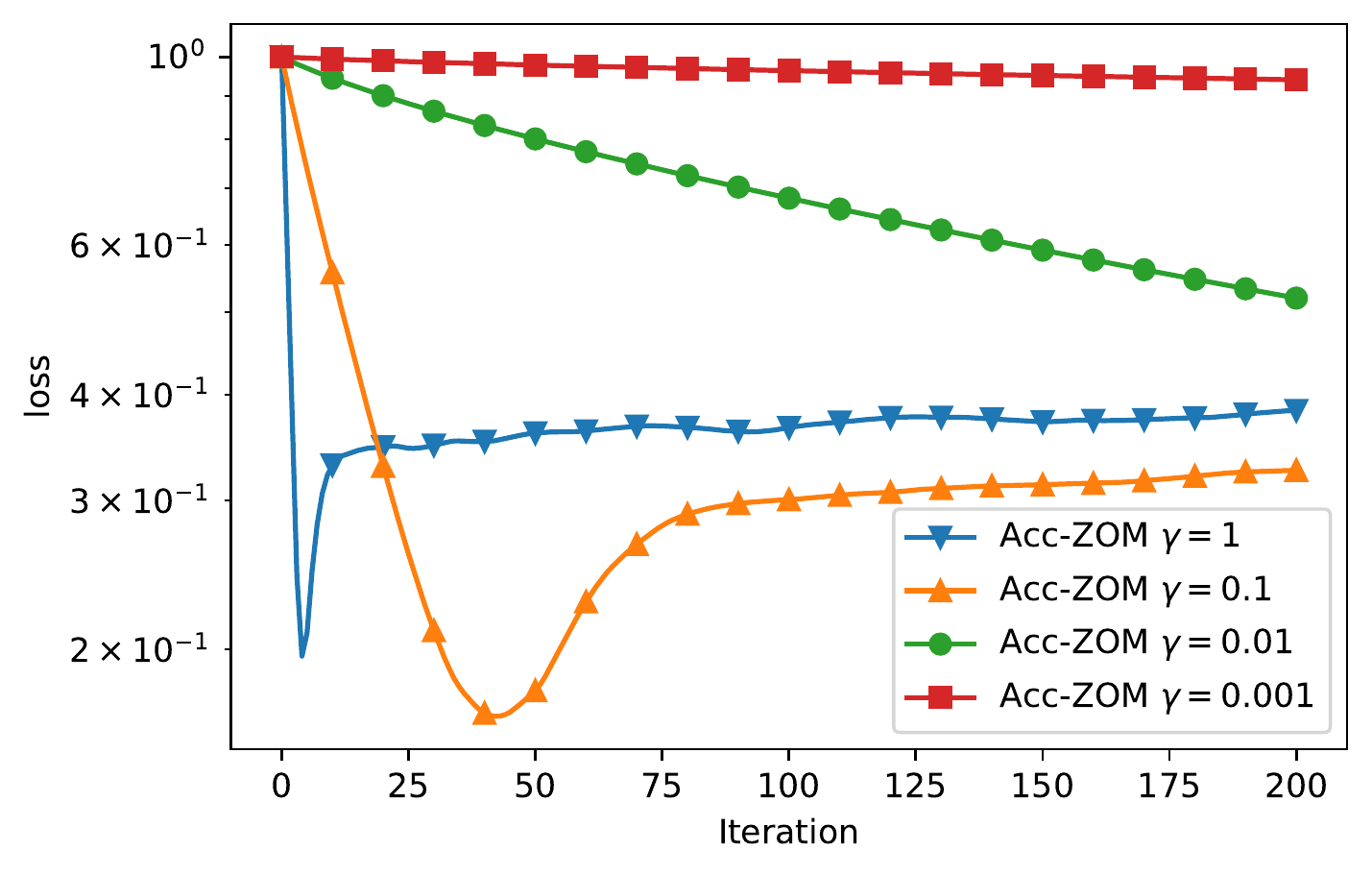}
\caption{Convergence for Generating  \textbf{PP}}%
\label{fig:bb-zom-pp}
\end{subfigure}
\caption{Impact of step size on Acc-ZOM on CIFAR-$10$}
\label{fig:BB-ZOM}
\end{figure}

\begin{figure}
\centering
\begin{subfigure}{.5\textwidth}
  \centering
  \includegraphics[width=\linewidth]{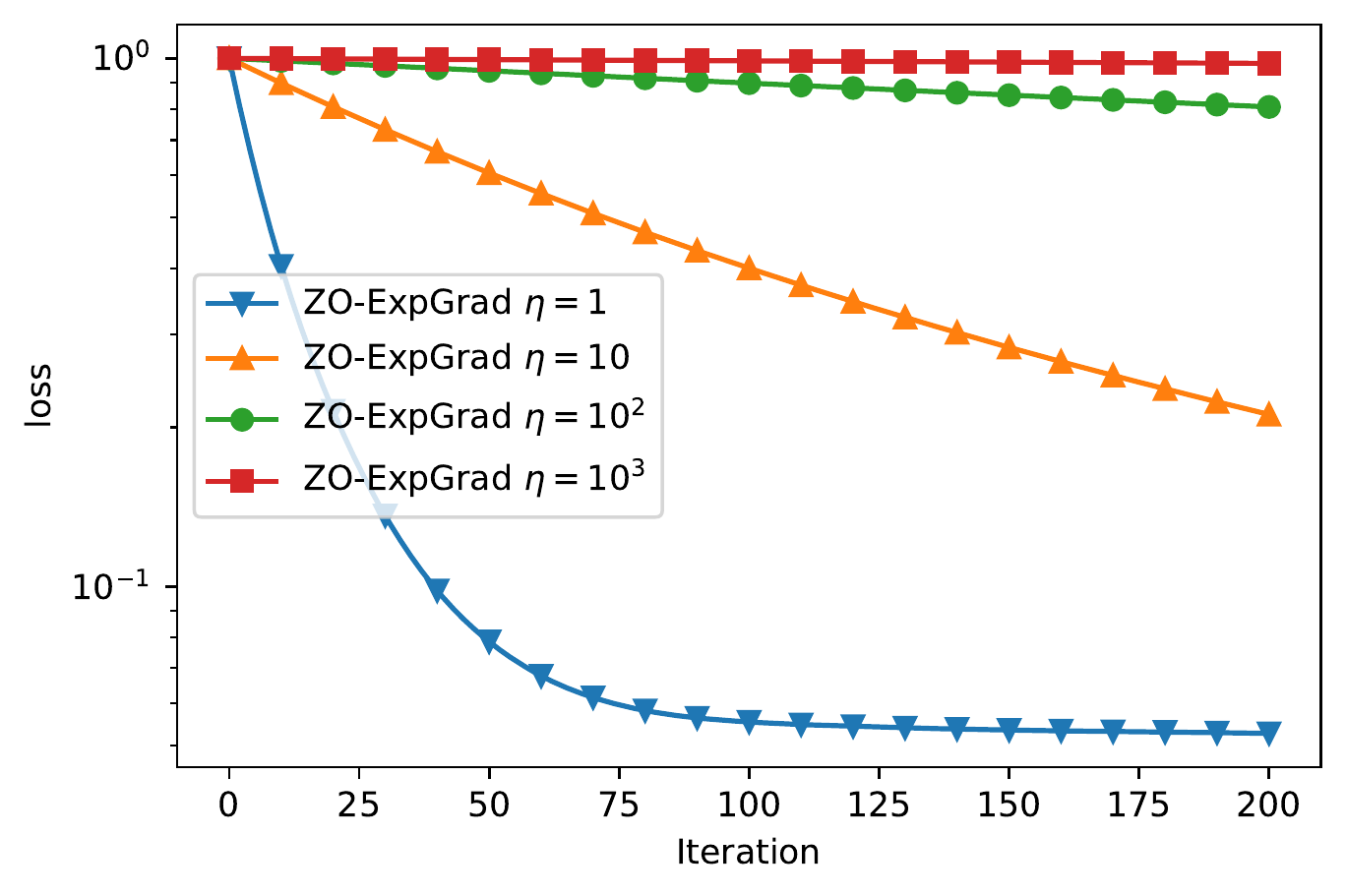}
\caption{Convergence for Generating \textbf{PN}}%
\label{fig:bb-exp-pn-mnist}
\end{subfigure}%
\begin{subfigure}{.5\textwidth}
  \centering
  \includegraphics[width=\linewidth]{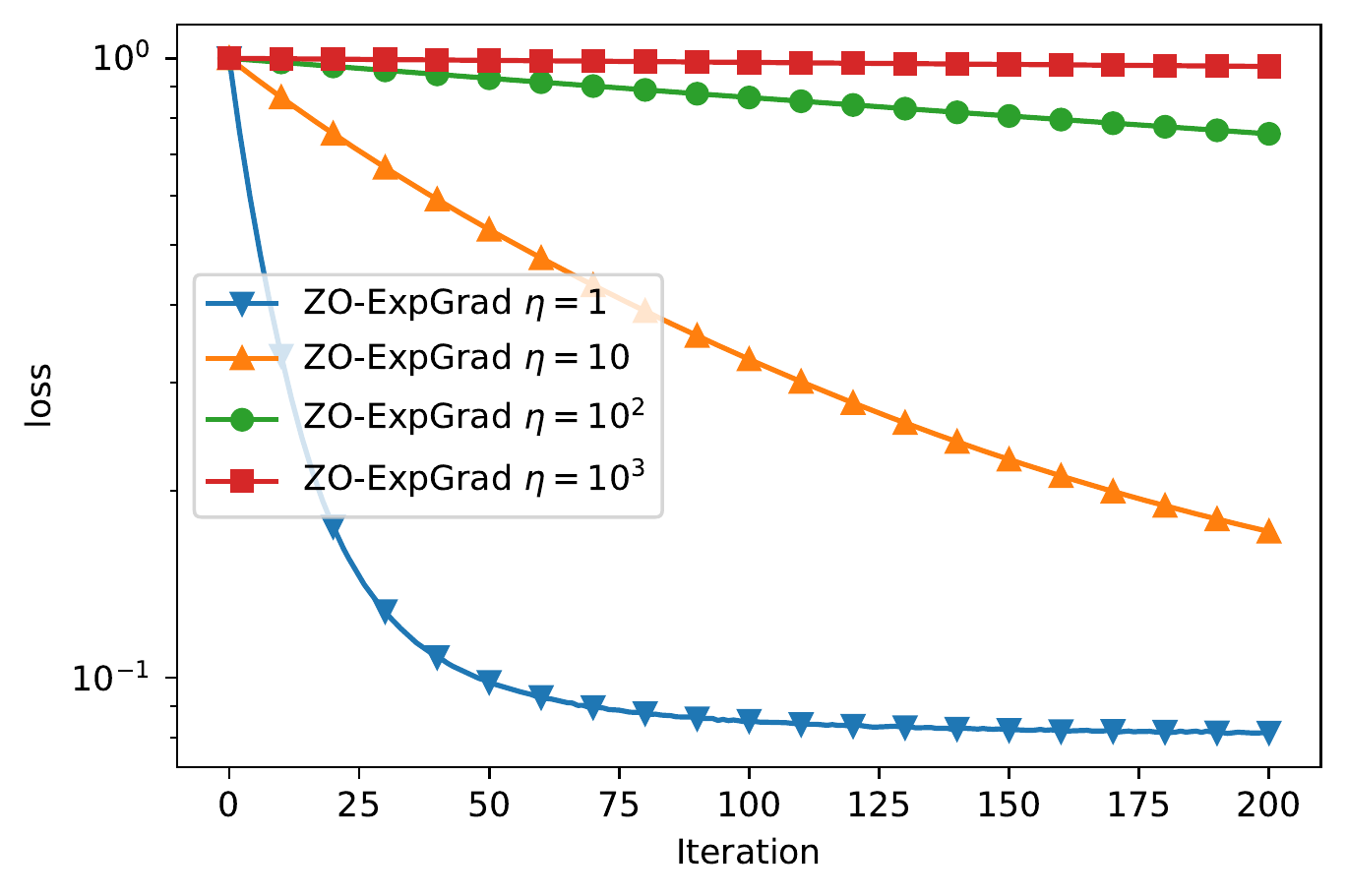}
\caption{Convergence for Generating  \textbf{PP}}%
\label{fig:bb-exp-pp-mnist}
\end{subfigure}
\caption{Impact of step size on ZO-ExpGrad on MNIST}
\label{fig:BB-exp-mnist}
\end{figure}
\begin{figure}
\centering
\begin{subfigure}{.5\textwidth}
  \centering
  \includegraphics[width=\linewidth]{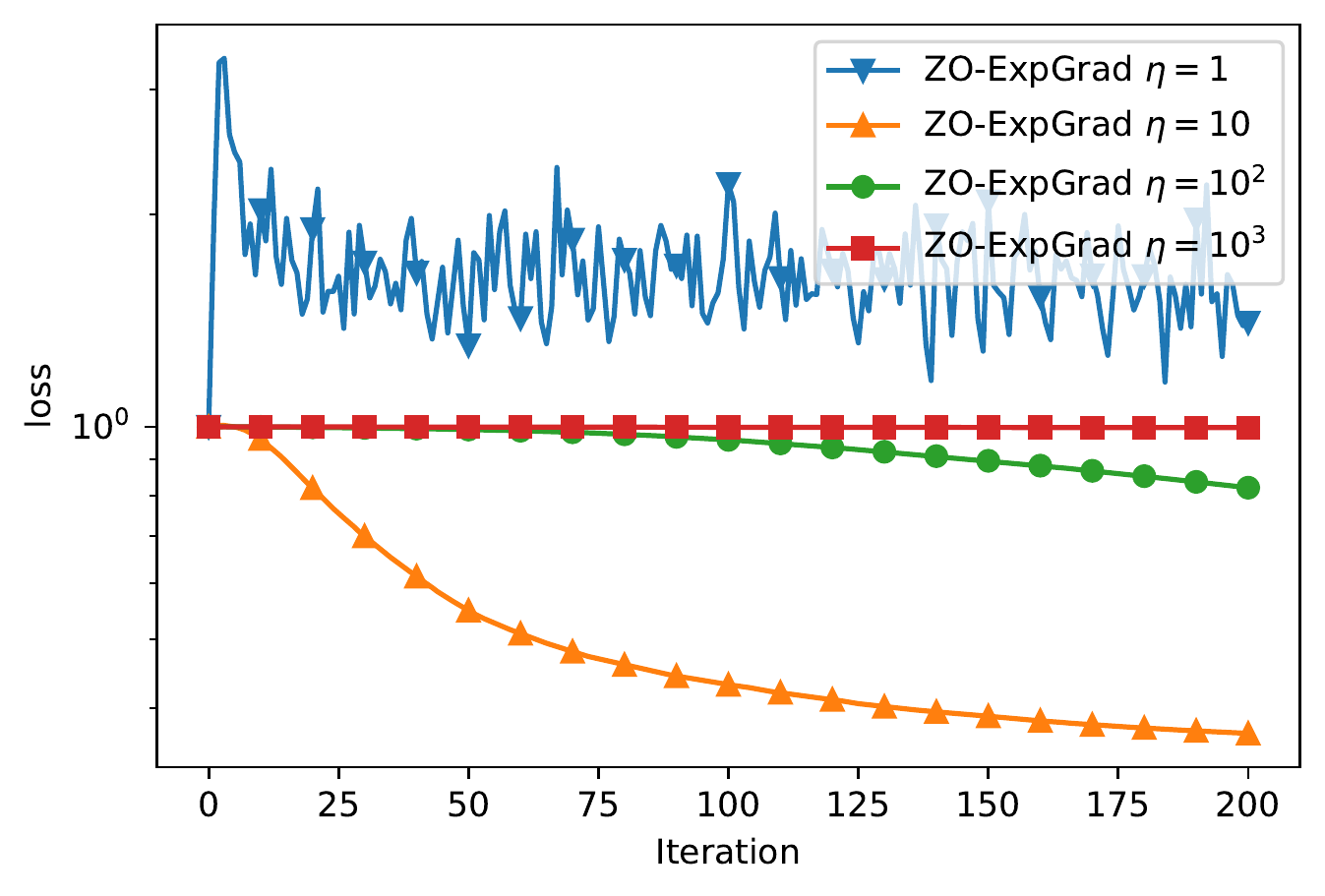}
\caption{Convergence for Generating \textbf{PN}}%
\label{fig:bb-exp-pn}
\end{subfigure}%
\begin{subfigure}{.5\textwidth}
  \centering
  \includegraphics[width=\linewidth]{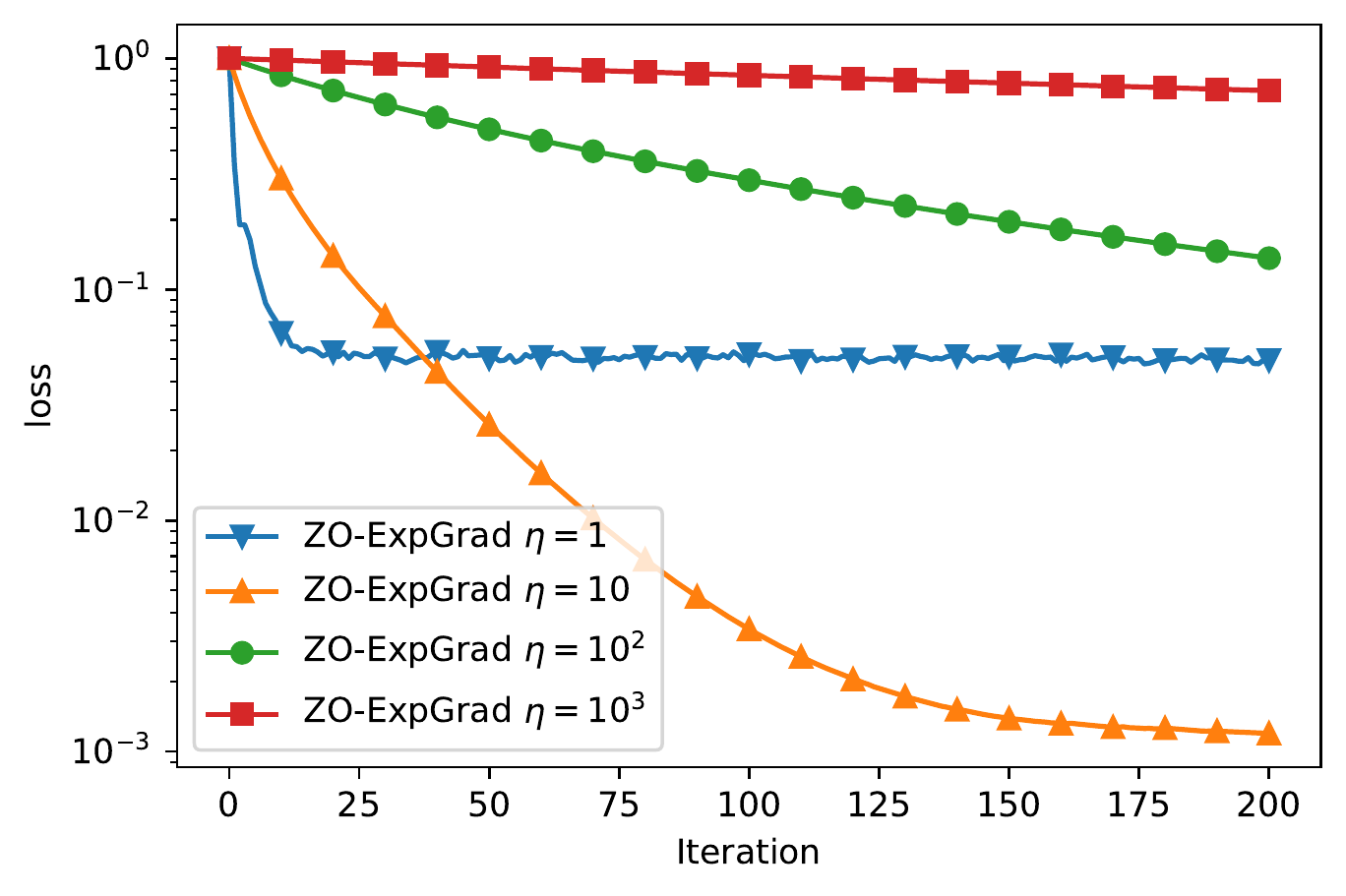}
\caption{Convergence for Generating  \textbf{PP}}%
\label{fig:bb-exp-pp}
\end{subfigure}
\caption{Impact of step size on ZO-ExpGrad on CIFAR-$10$}
\label{fig:BB-exp}
\end{figure}

\subsection{Zoomed-in Comparison of Proposed Algorithms}

\begin{figure}
\centering
\begin{subfigure}{.5\textwidth}
  \centering
  \includegraphics[width=\linewidth]{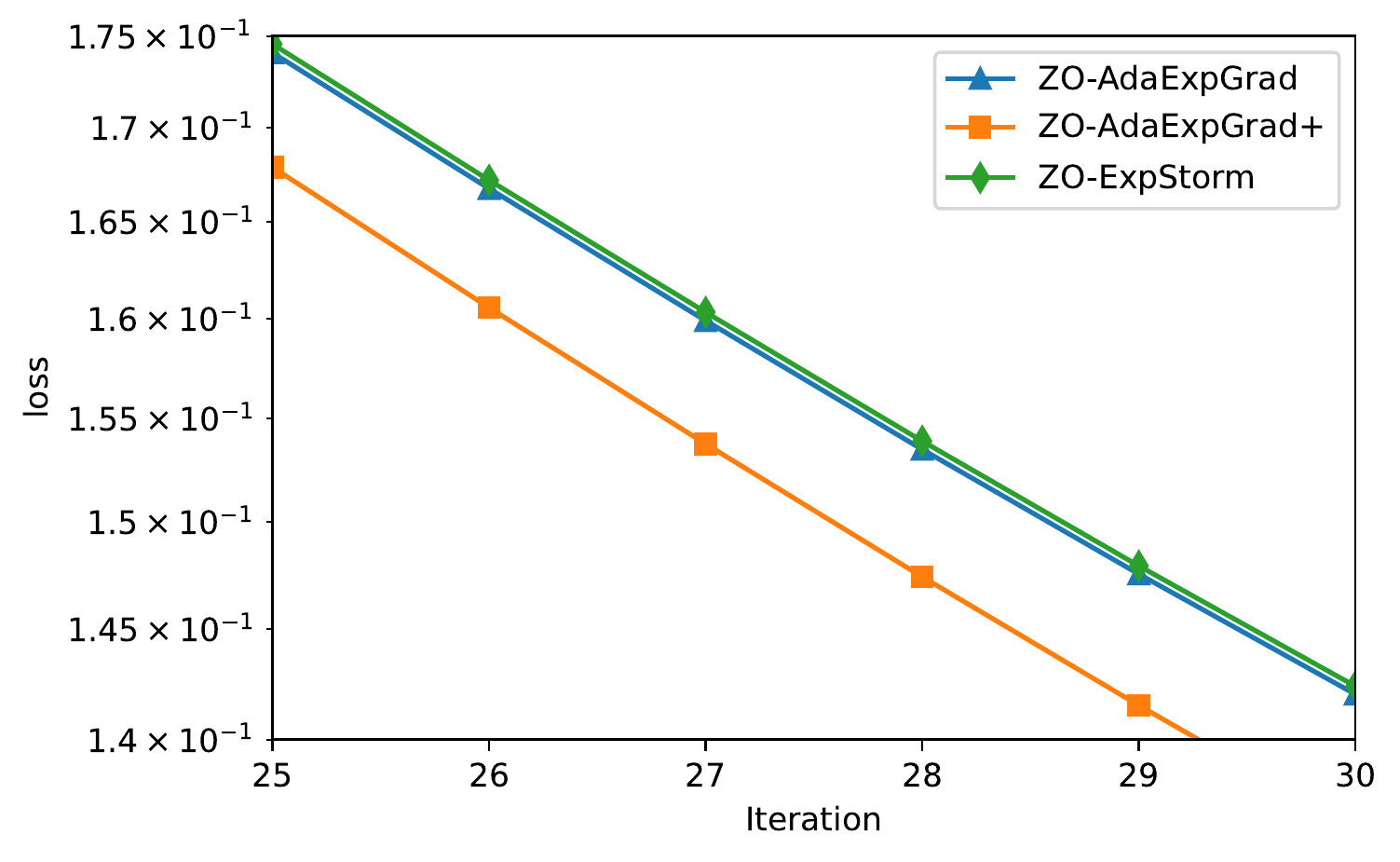}
\caption{Step $25$-$30$ of Generating \textbf{PN}}%
\end{subfigure}%
\begin{subfigure}{.5\textwidth}
  \centering
  \includegraphics[width=\linewidth]{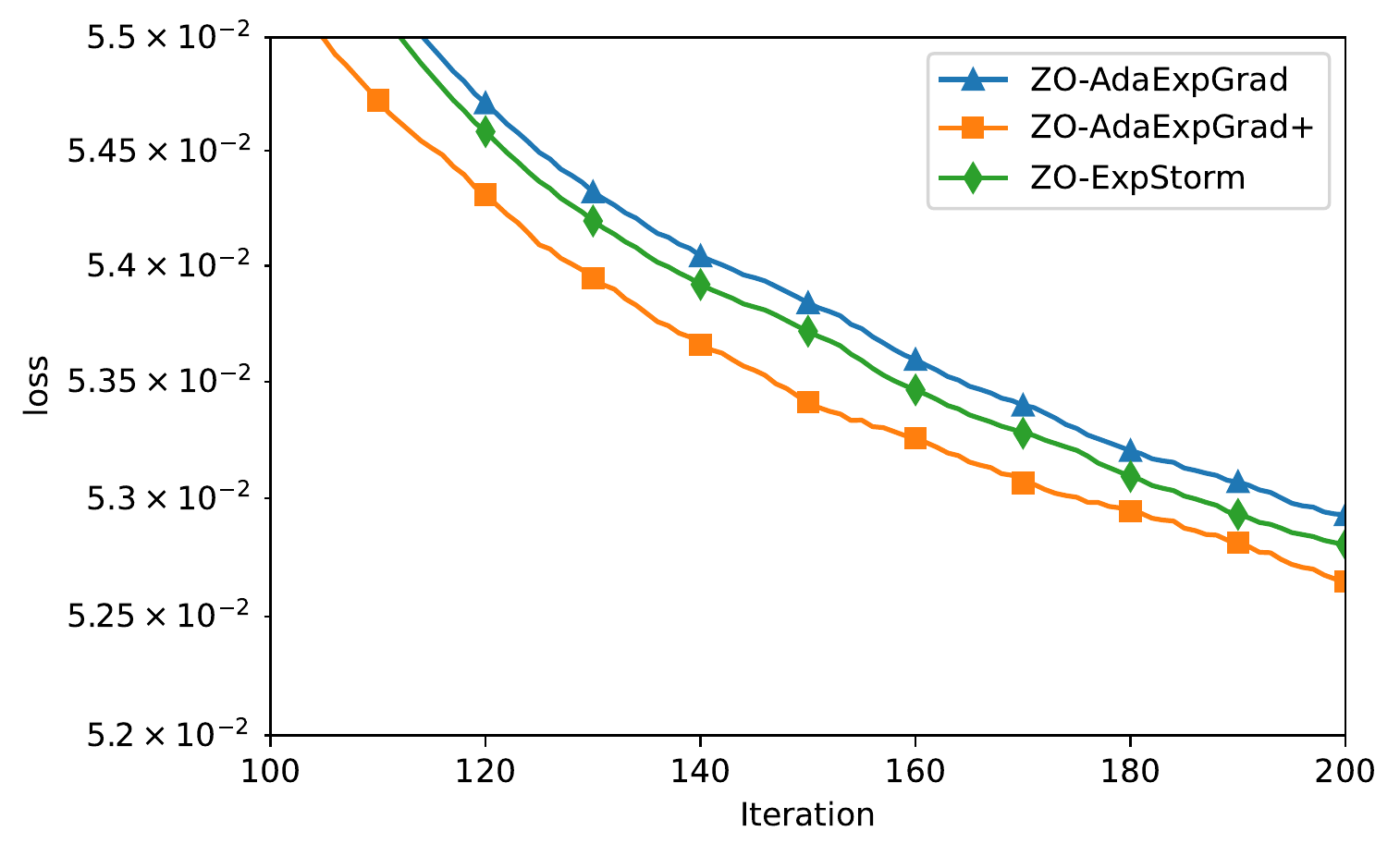}
\caption{Step $100$-$200$ of Generating \textbf{PN}}%
\end{subfigure}
\centering
\begin{subfigure}{.5\textwidth}
  \centering
  \includegraphics[width=\linewidth]{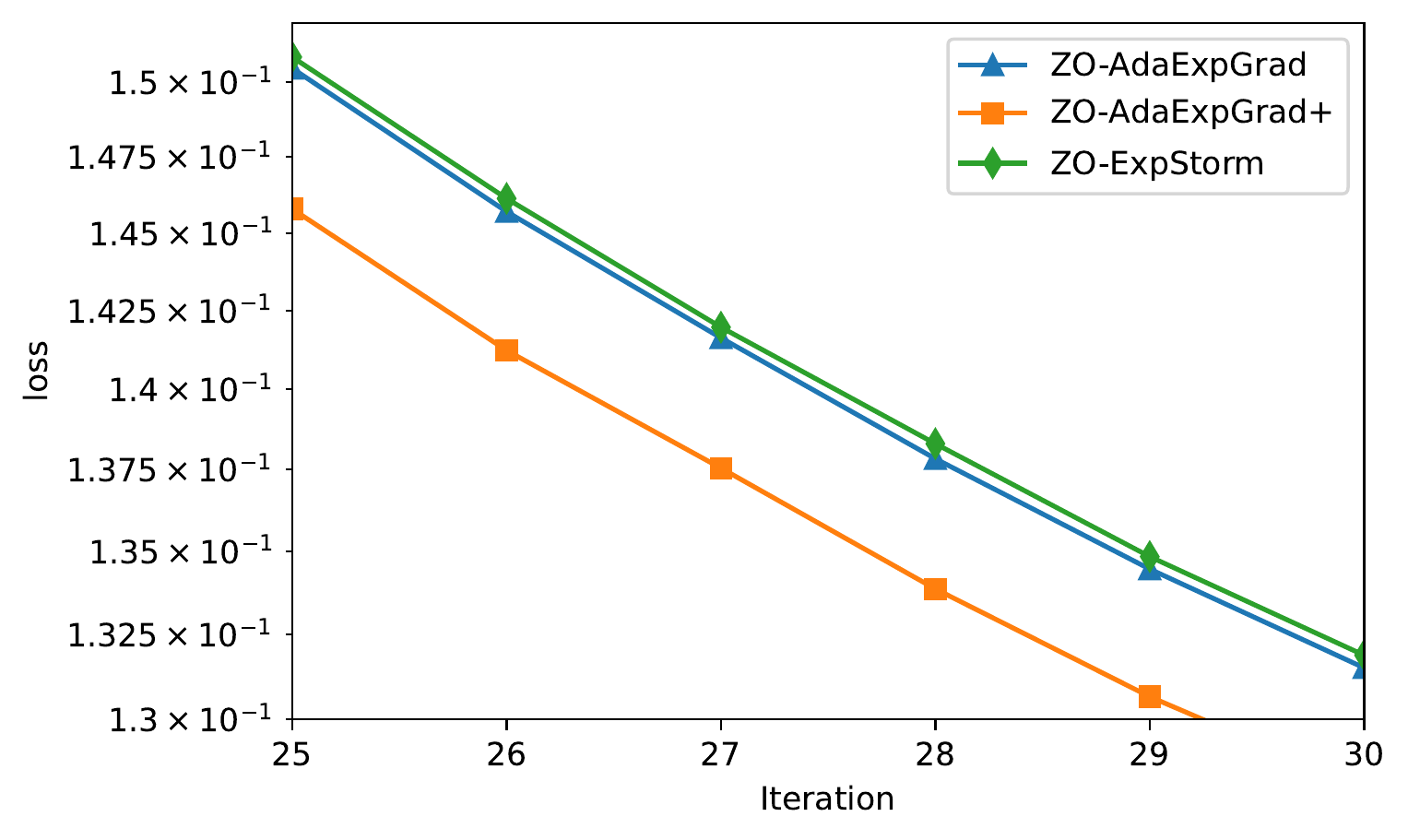}
\caption{Step $25$-$30$ of Generating  \textbf{PP}}%
\end{subfigure}%
\begin{subfigure}{.5\textwidth}
  \centering
  \includegraphics[width=\linewidth]{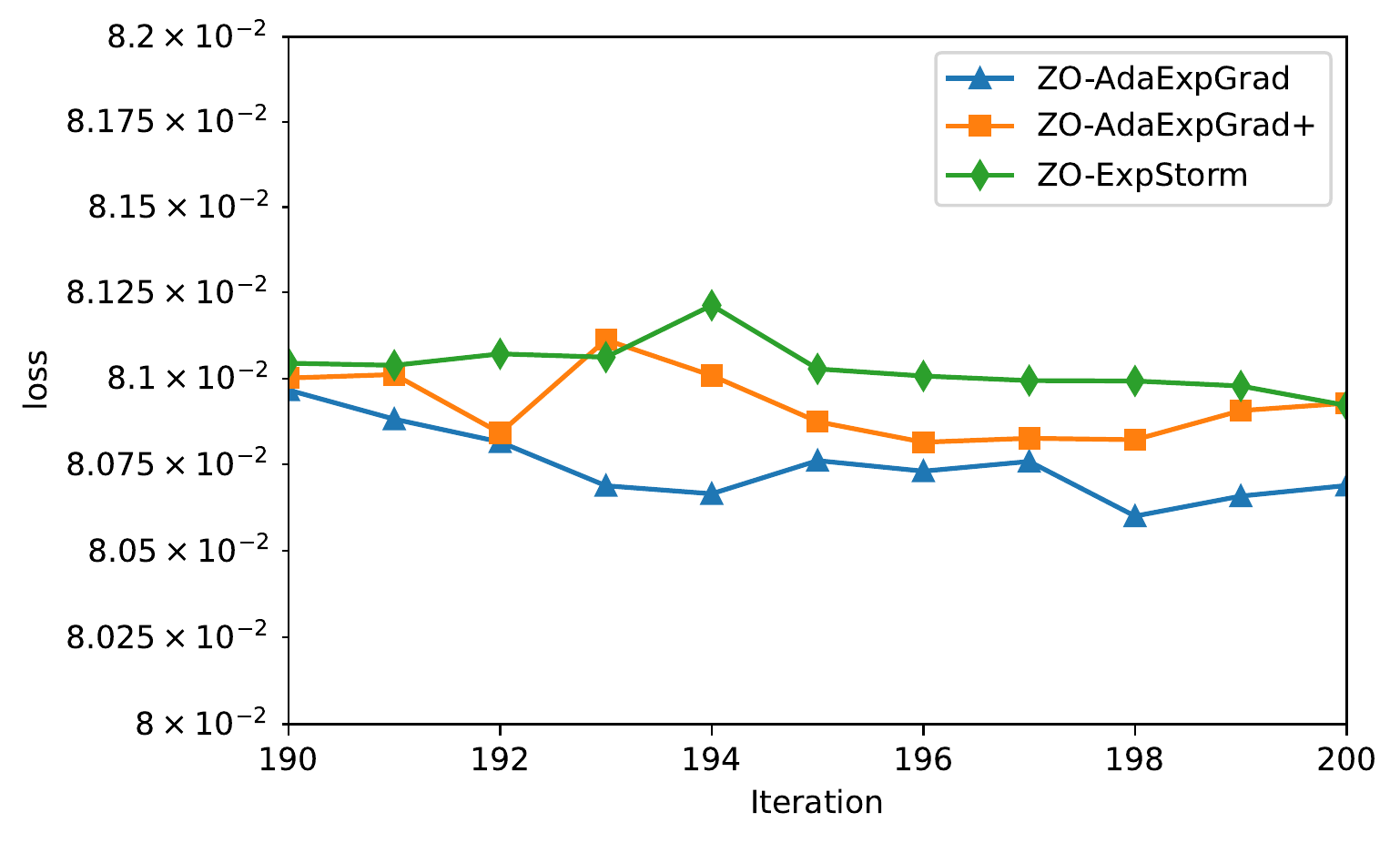}
\caption{Step $190$-$200$ for Generating \textbf{PP}}%
\end{subfigure}
\caption{Zoom-in Comparison on MNIST}
\end{figure}

\begin{figure}
\centering
\begin{subfigure}{.5\textwidth}
  \centering
  \includegraphics[width=\linewidth]{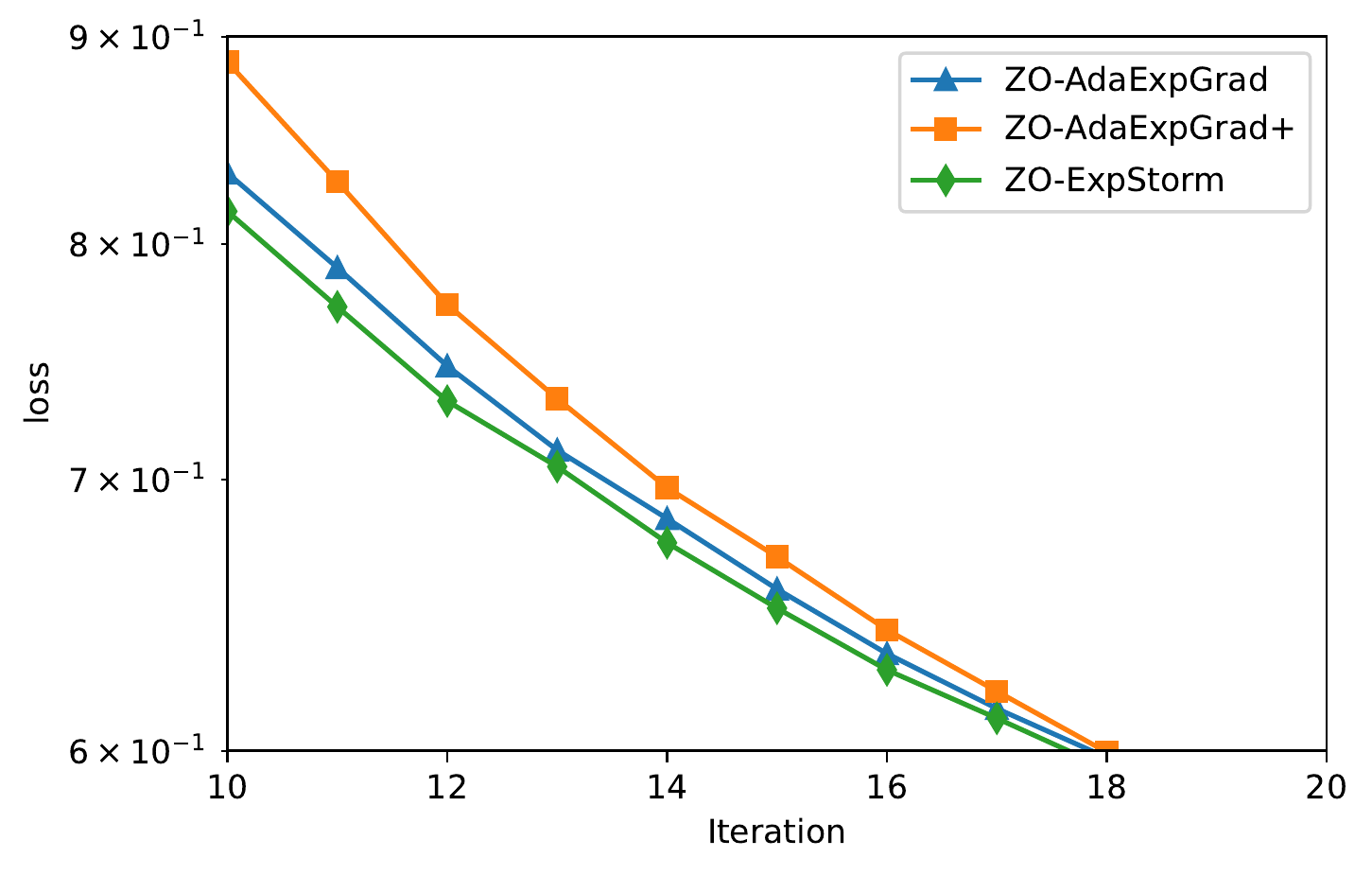}
\caption{Step $10$-$20$ of Generating \textbf{PN}}%
\end{subfigure}%
\begin{subfigure}{.5\textwidth}
  \centering
  \includegraphics[width=\linewidth]{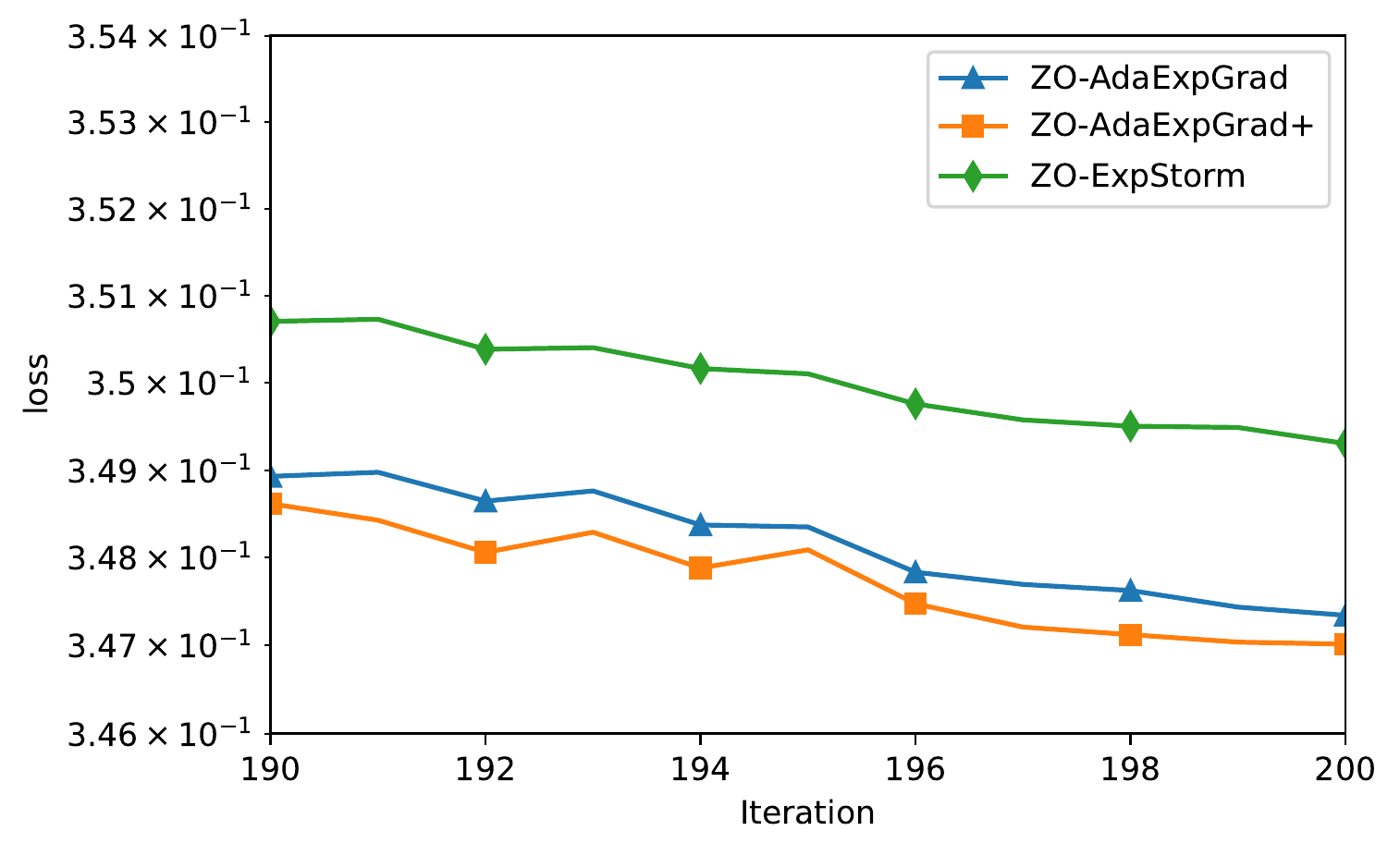}
\caption{Step $190$-$200$ of Generating \textbf{PN}}%
\end{subfigure}
\centering
\begin{subfigure}{.5\textwidth}
  \centering
  \includegraphics[width=\linewidth]{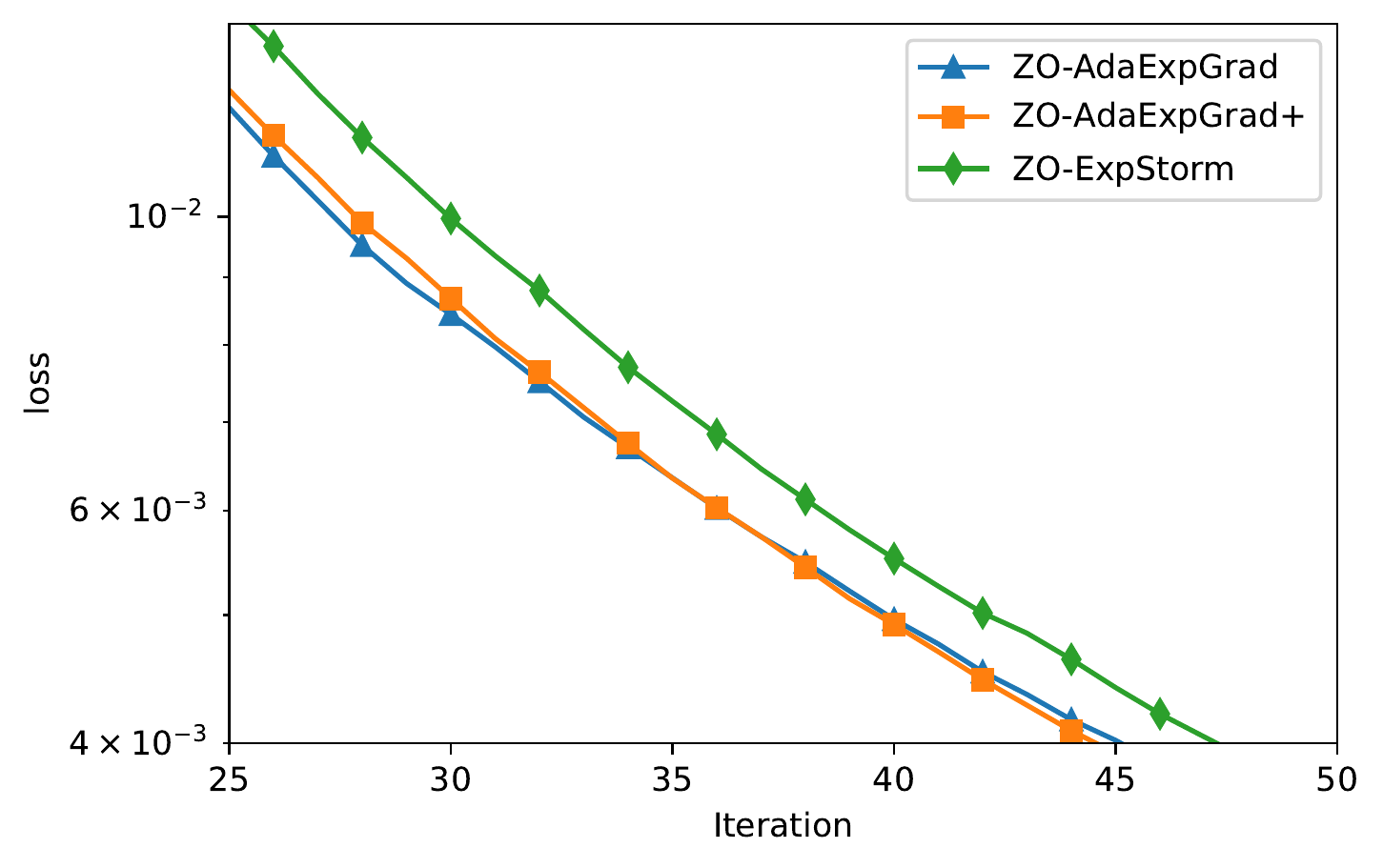}
\caption{Step $25$-$30$ of Generating  \textbf{PP}}%
\end{subfigure}%
\begin{subfigure}{.5\textwidth}
  \centering
  \includegraphics[width=\linewidth]{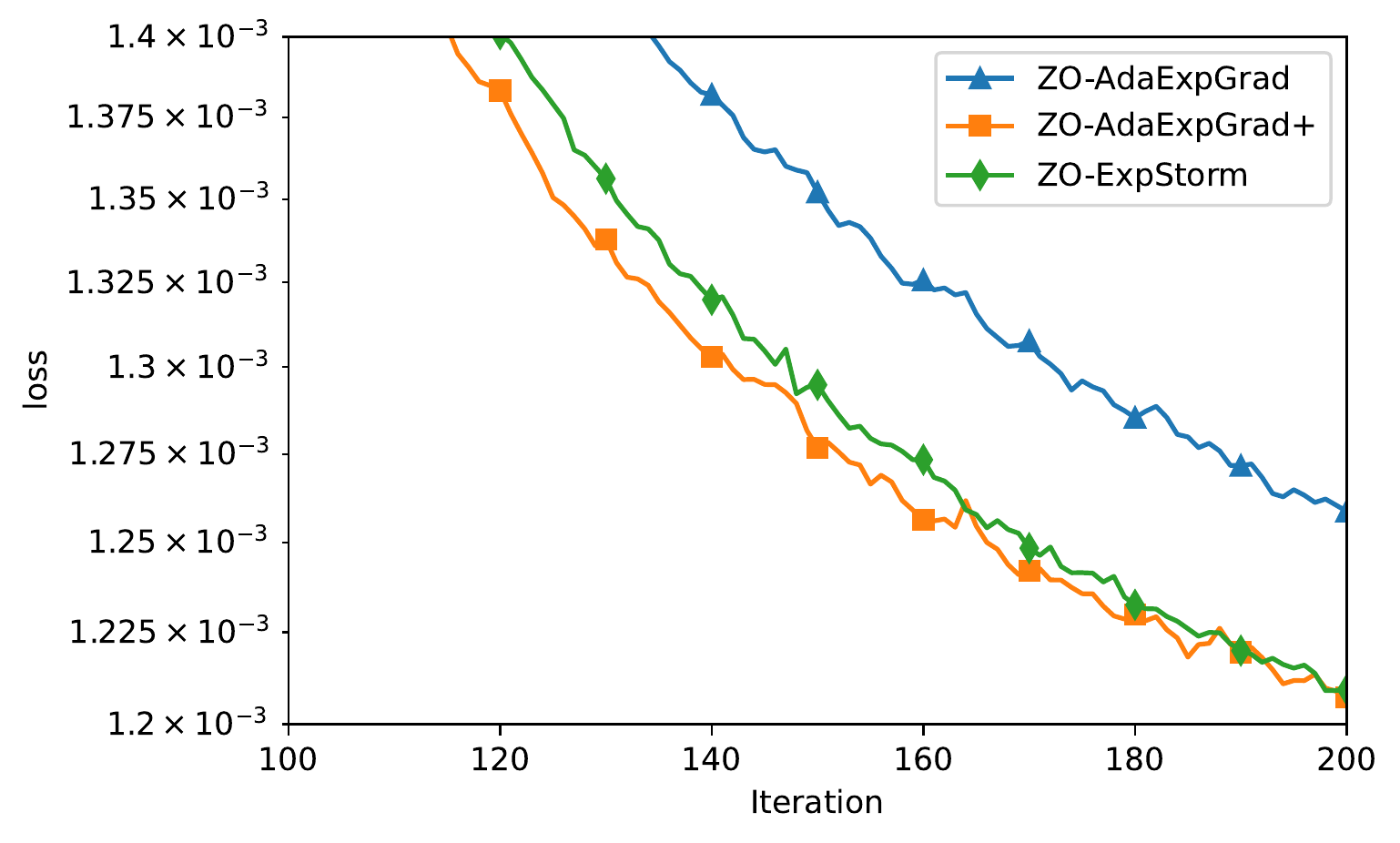}
\caption{Step $100$-$200$ for Generating \textbf{PP}}%
\end{subfigure}
\caption{Zoom-in Comparison on CIFAR}
\end{figure}

\end{document}